\documentclass{article}
\usepackage[english]{babel}
\usepackage{graphicx}
\usepackage{textcomp}
\usepackage{amssymb}
\usepackage{amsmath}
\usepackage{amsthm}
\usepackage{setspace}
\usepackage[margin=2.5cm]{geometry}
\usepackage{bm}
\usepackage{centernot}
\usepackage{relsize}
\usepackage{hyperref}
\usepackage{amsmath}
\usepackage{amssymb}
\usepackage{mathtools}
\usepackage{xcolor}
\DeclareFontFamily{U}{mathx}{}
\DeclareFontShape{U}{mathx}{m}{n}{
	<-> mathx10
}{}
\DeclareSymbolFont{mathx}{U}{mathx}{m}{n}
\DeclareMathAccent{\widecheck}{0}{mathx}{"71}
\newtheorem{proposition}{Proposition}
\newtheorem{corollary}{Corollary}
\newtheorem{remark}{Remark}
\newtheorem{definition}{Definition}

\newtheorem{lemma}{Lemma}

\makeatletter
\renewcommand*\env@matrix[1][*\c@MaxMatrixCols c]{%
  \hskip -\arraycolsep
  \let\@ifnextchar\new@ifnextchar
  \array{#1}}
\makeatother

\begin{document}

\title{Very weak solutions of the heat equation with anisotropically singular time-dependent diffusivity}
\author{Zhirayr Avetisyan, Zahra Keyshams, Monire Mikaeili Nia, Michael Ruzhansky}
\date{}
\maketitle
\section{Abstract}
We investigate the heat equation with a time-dependent, anisotropic, and potentially singular diffusivity tensor.
Since weak (in the Sobolev sense) or distributional solutions may not exist in this setting, we employ the framework of very weak solutions to establish the existence and uniqueness of solutions to the heat equation with singular, anisotropic, time-dependent diffusivity.
\section{Introduction}
The concept of a solution to a partial differential equation (PDE) has evolved remarkably in the last century or more, from the classical solutions which satisfied the equation in the normal sense (i.e., pointwise), to the weak solutions in sense of Sobolev, or the weak solutions in the sense of an optimization problem (e.g., Lagrangian), all the way to the distributional solutions in various spaces of distributions, ultradistributions or even Fourier hyperfunctions. While each of these approaches has its advantages, it appears that none of them is capable of handling PDEs with strongest (e.g., distributional) coefficients. The recently proposed very weak solutions seem to be a promising way of addressing such problems.

The concept of very weak solutions of PDEs with singular (i.e., non-smooth) coefficients goes back to the paper \cite{GaRu15} by C. Garetto and M. Ruzhansky, where the authors considered second-order weakly hyperbolic equations with singular time-dependent coefficients. While weak (in the sense of Sobolev) or distributional solutions may not exist, it is established that a certain regularization scheme, called very weak solution, always produces appropriately unique solutions to such equations. Moreover, in case the equation does admit distributional or ultra-distributional solutions, the very weak solutions coincide with them in the appropriate sense. The latter condition, called consistency, is the main advantage of the method of very weak solutions over another widely used approach to singular PDEs called the Colombeau algebras.

Since then, the subject of very weak solutions has gained much momentum, see \cite{ARST21a, ARST21b, ChRT22, Gar20, RY20} for more references. However, to our knowledge the only work where the heat equation with time-dependent singular coefficients is considered is \cite{ARST21}, where the authors address the heat equation with a singular potential. In the present paper we study the existence, uniqueness of very weak solutions of the heat equation without a potential, but with a time-dependent, fully anisotropic thermal diffusivity tensor, which is allowed to be strongly singular. Sharp, non-gradual (i.e., singular) coefficients in PDEs appear in many physically relevant problems, where the medium consists of more than one phases with a sharp transmission interface. In particular, when a phase transition in the medium happens much faster than the diffusion process under consideration, the diffusivity coefficients can be effectively seen to be suffering discontinuities. Such a rapid phase transition may happen during the crystallization of a supercooled fluid, and if the resulting crystal is anisotropic, then a diffusion process in this medium will be described with the kind of equations we handle in this paper. Another interesting situation where such equations may be relevant is cosmology. Our best understanding of the history of early universe is that it has undergone phase transitions, which have dramatically changed its chemical and thermal properties. Moreover, many models of early universe assume that it was anisotropic, which could result in anisotropic phase transitions. Thus, any diffusion process that has happened in the early universe at cosmic scales may have to be described by the kind of equations discussed in this work. Readers interested in PDE-theoretical and Fourier-analytical aspects of anisotropic cosmological models may consult \cite{AvVe13} and references therein. 

The structure of this paper is organized as follows. In Section 2, we present the foundational framework and classical results that underpin the analysis developed later. This section includes several key inequalities that will be instrumental throughout the paper.

Section 3 contains the main contributions of this work. We begin with Preliminaries, introducing the necessary functional and analytical tools. We then formulate the problem in a very weak sense in the first subsection. In the second subsection, then, we construct and study solution operators corresponding to this formulation. In the final subsection, we extend our analysis to the $L^{p}-L^{q}$-setting, highlighting the behavior of solutions in broader function spaces.
The last section of the paper concludes with a list of references.

\textbf{Convention.} 
Throughout this paper,  $C$ is used to represent positive constants whose values may change at every occurrence, which are independent of the main involved parameters. We use the notation $A\lesssim_{\mu}B$ to indicate that there exists a constant $C_{\mu}$, depending only on the parameter $\mu$, such that
$A\leq C_{\mu} B.$
When the dependence on $\mu$
is not essential or is clear from the context, we may write $A\lesssim B$ to simplify the notation to mean $A\leq CB$, for a positive constant $C$.
\section{The classical theory}

In this section, we will introduce the heat equation with anisotropic smooth time-dependent diffusivity, and present its solution by very classical means. The main technical challenge here is the fact that, due to anisotropic evolution, the time-dependence of the diffusivity cannot be factored out, resulting in somewhat more involved expressions. The methods and ideas, however, are classical.

Consider the following partial differential equation,
\begin{equation}
\partial_tu(t,x)-\mathcal{L}_tu(t,x)=f(t,x),\label{heatEq}
\end{equation}
where $\{\mathcal{L}_t\}_{t\in\mathbb{R}_+}$ is a smooth family of partial differential operators given by
$$
\mathcal{L}_tv(x)=\sum_{i,j=1}^n a_{ij}(t)\partial^2_{x_ix_j}v(x),\quad\forall x\in\mathbb{R}^n,\quad\forall t\in\mathbb{R}_+,\quad\forall v\in C^\infty(\mathbb{R}^n),\quad n\in\mathbb{N}.
$$
Here $a\in C^\infty(\mathbb{R}_+,\mathrm{GL}(n))\cap C([0,\infty),\mathrm{GL}(n))$ is a matrix valued smooth function, such that $a(t)$ is a symmetric positive definite matrix for all $t\in\mathbb{R}_+$. We will use the inner product in $\mathbb{R}^n$ to write
$$
\mathcal{L}_t=\langle\partial_x,a(t)\partial_x\rangle,\quad\forall t\in\mathbb{R}_+.
$$
Denote
$$
\Delta\doteq\left\{(s,t)\in[0,+\infty)\times\mathbb{R}_+\;\vline\quad s<t\right\},
$$
then its closure in $\mathbb{R}^2$ will be
$$
\overline{\Delta}=\left\{(s,t)\in[0,+\infty)^2\;\vline\quad s\le t\right\}.
$$
Introduce the matrix valued smooth function $A\in C^\infty(\mathbb{R}_+,\mathrm{GL}(n))$ by
$$
A(t)=\int\limits_0^ta(s)ds,\quad\forall t\in\mathbb{R}_+.
$$
It is clear that $A(t)-A(s)$ is a symmetric positive definite matrix for every $(s,t)\in\Delta$.

Define the function $\mathcal{W}\in\mathcal{S}(\mathbb{R}^n)\otimes C^\infty(\Delta)$ by
$$
\mathcal{W}(x;s,t)\doteq\frac1{\sqrt{(4\pi)^n\det[A(t)-A(s)]}}e^{-\frac14\langle x,[A(t)-A(s)]^{-1}x\rangle},\quad\forall x\in\mathbb{R}^n,\quad\forall(s,t)\in\Delta.
$$
\begin{remark}\label{remarkfourierinverse}
\begin{align}\label{fourierinverse}
	\mathcal{W}(x;s,t) = \mathcal{F}^{-1}\left(e^{-\langle[A(t)-A(s)]\cdot, \cdot \rangle}\right)(x)=\widecheck{\left(e^{-\langle[A(t)-A(s)]\cdot, \cdot \rangle}\right)}(x),~~~\forall x\in \mathbb{R}^{n},~~\forall (s,t)\in \Delta,
\end{align}
where $\mathcal{F}^{-1}$ stands for the inverse Fourier transform. This is because, for a symmetric positive definite matrix $B$ and an arbitrary vector $b$ and a  constant $c$, one can verify through computation and the application of the Gaussian integral that
\begin{align}\label{formulaforinverse}
	\int\limits_{\mathbb{R}^n} e^{-\frac{1}{2}x^{\top}Bx+b^{\top} x+c} dx = \sqrt{\dfrac{(2\pi)^n}{\det B}} e^{\frac{1}{2}b^{\top}B^{-1}b+c}.
\end{align}	
	Here, we have
\begin{align*}
	\mathcal{F}^{-1}\left(e^{-\langle[A(t)-A(s)]\cdot, \cdot\rangle} \right)(x)= (2\pi)^{-n}\int\limits_{\mathbb{R}^n} e^{-\frac{1}{2}\xi^{\top}(2[A(t)-A(s)])\xi+ix.\xi} d\xi,~~~\forall x\in \mathbb{R}^{n},~~~\forall (s,t)\in \Delta.
\end{align*}
Therefore, by letting $B=2[A(t)-A(s)]$ and $b^{\top}=ix$ in \eqref{formulaforinverse} we obtain the result in \eqref{fourierinverse}.
\end{remark}
\begin{proposition}\label{pr1}
The following properties hold:
\begin{itemize}
\item[1.]
$$
\int\limits_{\mathbb{R}^n}\mathcal{W}(x;s,t)dx=1,\quad\forall(s,t)\in\Delta;
$$

\item[2.]
$$
\partial_t\mathcal{W}(x;s,t)-\mathcal{L}_t\mathcal{W}(x;s,t)=0,\quad\forall x\in\mathbb{R}^n,\quad\forall(s,t)\in\Delta;
$$

\item[3.]
$$
\lim_{\epsilon\to0+}\mathcal{W}(x;s,s+\epsilon)=\delta,\quad\forall s\in[0,+\infty);
$$

\item[4.] 
$$
\lim_{\epsilon\to0+}\mathcal{W}(x; t-\epsilon,t)=\delta,\quad\forall t\in(0,+\infty).
$$
\end{itemize}
\end{proposition}
\begin{proof}
In order to prove part 1, consider the Gaussian integral as follows:
$$
\int\limits_{\mathbb{R}^n} e^{-\alpha y^{\top}y} dy = \left(\frac{\pi}{\alpha}\right)^{\frac{n}{2}}, ~~\alpha>0.
$$
Here, we have 
\begin{align}\label{gaussian}
	I &\doteq\int\limits_{\mathbb{R}^n} e^{-\frac{1}{4}\langle x, [A(t)-A(s)]^{-1}x\rangle} dx \notag\\
	&= \int\limits_{\mathbb{R}^n} e^{-\frac{1}{4}\langle [A(t)-A(s)]^{-\frac{1}{2}}x, [A(t)-A(s)]^{-\frac{1}{2}}x\rangle} dx,\quad\forall(s,t)\in\Delta.
\end{align}
Note that $[A(t)-A(s)]^{-\frac{1}{2}}$ exists because $A(t)-A(s)$ is a positive definite matrix. Let $y = [A(t)-A(s)]^{-\frac{1}{2}}x$, therefore, $dx = \det[(A(t)-A(s))^{\frac{1}{2}}] dy$.
Continuing the calculation in \eqref{gaussian}, we have
\begin{align*}
	I &= \int\limits_{\mathbb{R}^n} e^{-\frac{1}{4} y^Ty} dy \sqrt{\det[A(t)-A(s)]} = (4\pi)^{\frac{n}{2}} \sqrt{\det[A(t)-A(s)]},\quad\forall(s,t)\in\Delta.
\end{align*}
Hence,
\begin{align*}
	\int\limits_{\mathbb{R}^n}\mathcal{W}(x;s,t)dx= \frac1{\sqrt{(4\pi)^n\det[A(t)-A(s)]}}\, \mathrm{I} =1,\quad\forall(s,t)\in\Delta.
\end{align*}
To prove part 2, note that
$$
\mathcal{L}_t\mathcal{W}(x;s,t)=\sum_{i,j=1}^n a_{ij}(t)\partial^2_{x_ix_j}\mathcal{W}(x;s,t),\quad\forall x\in\mathbb{R}^n,\quad\forall (s,t)\in\Delta.
$$
Therefore, we need to calculate $\partial^2_{x_ix_j}\mathcal{W}(x;s,t)$.
We have
\begin{align}
	\partial_{x_j}\mathcal{W}(x;s,t) &= \frac{1}{\sqrt{(4\pi)^n\det[A(t)-A(s)]}}  \partial_{x_j}\left(e^{-\frac{1}{4}\langle x, [A(t)-A(s)]^{-1}x\rangle}\right) \notag\\ 
	&= -\frac{1}{4} \frac{1}{\sqrt{(4\pi)^n\det[A(t)-A(s)]}} e^{-\frac{1}{4}\langle x, (A(t)-A(s))^{-1}x\rangle} 2 \left((A(t)-A(s))^{-1}x\right)_{j}\notag\\
	&= -\frac{1}{2}\mathcal{W}(x;s,t)\left((A(t)-A(s))^{-1}x\right)_{j},\quad\forall x\in\mathbb{R}^n,\quad\forall(s,t)\in\Delta.
\end{align}
Then, with a similar calculation,
\begin{align}
	&\partial^2_{x_ix_j}\mathcal{W}(x;s,t) = 
	\partial_{x_i}(\partial_{x_j}\mathcal{W}(x;s,t))= \partial_{x_i}\left(-\frac{1}{2}\mathcal{W}(x;s,t)\left((A(t)-A(s))^{-1}x\right)_{j}\right)\notag\\
	&= \frac{1}{4} \mathcal{W}(x;s,t) \left((A(t)-A(s))^{-1}x\right)_{i}\left((A(t)-A(s))^{-1}x\right)_{j} -\frac{1}{2} \mathcal{W}(x;s,t)\left((A(t)-A(s))^{-1}\right)_{i,j},\\
	&\forall x\in\mathbb{R}^n, \forall(s,t)\in\Delta.\notag
\end{align}
Therefore,
\begin{align}\label{LtW} 
	&\mathcal{L}_t\mathcal{W}(x;s,t)\notag\\
	&= \mathcal{W}(x;s,t)\left(-\frac{1}{2} \sum_{i,j=1}^n a_{i,j}(t)\left([A(t)-A(s)]^{-1}\right)_{i,j}+\frac{1}{4}\sum_{i,j=1}^n a_{i,j}(t)\left([A(t)-A(s)]^{-1}x\right)_{i}\left([A(t)-A(s)]^{-1}x\right)_{j}\right)\notag\\
	&=\mathcal{W}(x;s,t)\left(-\frac{1}{2} \operatorname{Tr}\left(a(t)[A(t)-A(s)]^{-1}\right)+\frac{1}{4}\langle[A(t)-A(s)]^{-1}x, a(t)[A(t)-A(s)]^{-1}x\rangle\right),\\
	&(\forall x\in\mathbb{R}^n,\quad \forall(s,t)\in\Delta)\notag.
\end{align}
To complete the proof, we need to calculate 
$\partial_t \mathcal{W}(x;s,t),$ for $ \forall x\in\mathbb{R}^n, \forall(s,t)\in\Delta$.
\begin{align}\label{partialtW}
	\partial_t \mathcal{W}(x;s,t)&=\partial_t\left(\frac{1}{\sqrt{(4\pi)^n\det[A(t)-A(s)]}}\right)e^{-\frac{1}{4}\langle x, [A(t)-A(s)]^{-1}x\rangle}\notag\\
	&+\frac{1}{\sqrt{(4\pi)^n\det[A(t)-A(s)]}}\partial_t\left(e^{-\frac{1}{4}\langle x, [A(t)-A(s)]^{-1}x\rangle}\right)\notag\\
	&=-\frac{1}{2\sqrt{(4\pi)^n}}{\det[A(t)-A(s)]^{-\frac{1}{2}}}\operatorname{Tr}\left([A(t)-A(s)]^{-1}a(t)\right)e^{-\frac{1}{4}\langle x, [A(t)-A(s)]^{-1}x\rangle}\notag\\
	&+\frac{1}{4}\frac{1}{\sqrt{(4\pi)^n\det[A(t)-A(s)]}}e^{-\frac{1}{4}\langle x, [A(t)-A(s)]^{-1}x\rangle}\langle x, [A(t)-A(s)]^{-1}a(t)[A(t)-A(s)]^{-1}x\rangle\notag\\
	&=\mathcal{W}(x; s,t)\left(-\frac{1}{2} \operatorname{Tr}\left([A(t)-A(s)]^{-1}a(t)\right)+\frac{1}{4}\langle [A(t)-A(s)]^{-1}x, a(t)[A(t)-A(s)]^{-1}x\rangle\right).
\end{align}
Comparing \eqref{LtW} and \eqref{partialtW} leads us to $$
\partial_t\mathcal{W}(x;s,t)-\mathcal{L}_t\mathcal{W}(x;s,t)=0,\quad\forall x\in\mathbb{R}^n,~ \forall(s,t)\in\Delta.
$$
In order to prove part $3$, since $\mathcal{W}(\cdot;s,s+\epsilon) \in \mathcal{S}(\mathbb{R}^n)\subset \mathcal{S}'(\mathbb{R}^n)$, for $\forall s\in [0,\infty)$, and for any sufficiently small  $\epsilon>0$, as well as for every $\phi \in \mathcal{S}(\mathbb{R}^n)$,  using Remark \ref{remarkfourierinverse}, we have
\begin{align}
	\lim_{\epsilon\to0+}\mathcal{W}(\cdot;s,s+\epsilon)(\phi) &=\lim_{\epsilon\to0+} \int\limits_{\mathbb{R}^n} \mathcal{W}(x;s,s+\epsilon) \phi(x) dx\notag\\
	&=\lim_{\epsilon\to0+}\int\limits_{\mathbb{R}^n} \widecheck{\left(e^{-\langle [A(s+\epsilon)-A(s)]\xi, \xi \rangle}\right)}(x) \phi(x) \, dx
	\notag\\
	&=\lim_{\epsilon\to0+} \int\limits_{\mathbb{R}^n} \left(e^{-\langle[A(s+\epsilon)-A(s)]\xi, \xi \rangle}\right) \check{\phi}(\xi) d\xi \notag\\
	&=\int\limits_{\mathbb{R}^n} \lim_{\epsilon\to0+} \left(e^{-\langle[A(s+\epsilon)-A(s)]\xi, \xi \rangle}\right) \check{\phi}(\xi) d\xi \notag\\
	&=\int\limits_{\mathbb{R}^n} 1 \check{\phi}(\xi) d\xi= \check{1}(\phi) = \delta(\phi).
\end{align}
Note that by the Dominated Convergence Theorem, we can pass to the limit inside the integral because   $\left(e^{-\langle[A(s+\epsilon)-A(s)]\cdot, \cdot \rangle}\right)$  is a uniformly bounded function $\forall \epsilon>0$ and $\forall s\in [0,\infty)$ as well as  $\check{\phi}\in \mathcal{S}(\mathbb{R}^n)\subset L^{1}(\mathbb{R}^n)$.

The proof of part 4 follows similarly to part 3. For every \( t > 0 \) and sufficiently small \( \epsilon > 0 \), we have \( \mathcal{W}(\cdot; t - \epsilon, t) \in \mathcal{S}(\mathbb{R}^n) \subset \mathcal{S}'(\mathbb{R}^n) \). Therefore, for any \( \phi \in \mathcal{S}(\mathbb{R}^n) \), by applying Remark \ref{remarkfourierinverse}, we obtain 
\begin{align}
	\lim_{\epsilon\to0+}\mathcal{W}(\cdot;t-\epsilon,t)(\phi) &=\lim_{\epsilon\to0+} \int\limits_{\mathbb{R}^n} \mathcal{W}(x;t-\epsilon,t) \phi(x) dx\notag\\
	&=\lim_{\epsilon\to0+}\int\limits_{\mathbb{R}^n} \widecheck{\left(e^{-\langle[A(t)-A(t-\epsilon)]\xi, \xi \rangle}\right)}(x) \phi(x) dx \notag\\
	&=\lim_{\epsilon\to0+} \int\limits_{\mathbb{R}^n} \left(e^{-\langle[A(t)-A(t-\epsilon)]\xi, \xi \rangle}\right) \check{\phi}(\xi) d\xi \notag\\
	&=\int\limits_{\mathbb{R}^n} \lim_{\epsilon\to0+} \left(e^{-\langle[A(t)-A(t-\epsilon)]\xi, \xi \rangle}\right) \check{\phi}(\xi) d\xi \notag\\
	&=\check{1}(\phi) = \delta(\phi).
	\end{align}
\end{proof}

Before proceeding, we first collect several remarks concerning the kernel $\mathcal{W}$, which will play an important role throughout the paper.
\begin{remark}\label{partialtimes}
For every $v\in\mathcal{S}'(\mathbb{R}^n),~ (s,t)\in\Delta ~ \text{and} ~ \alpha, \beta \in \mathbb{N}_{0}$, we have 
\begin{align}
	\frac{\partial^{\alpha+\beta}}{\partial t^\alpha \partial s^{\beta}}\left(\mathcal{W}(\cdot; s,t)\ast v\right) = \left(	\frac{\partial^{\alpha+\beta}}{\partial t^\alpha \partial s^{\beta}}\mathcal{W}(\cdot; s,t)\right)\ast v,
\end{align}
since $\mathcal{W}\in\mathcal{S}(\mathbb{R}^n)\otimes C^\infty(\Delta)$ implies that the mapping 
$$
\Delta\longrightarrow \mathcal{S}(\mathbb{R}^n),
$$
$$
(s,t)\mapsto \mathcal{W}(\cdot;s,t),
$$ 
belongs to $ C^\infty(\Delta,\mathcal{S}(\mathbb{R}^n))$.  
\end{remark}
\begin{remark}\label{partialtimesx}
For every $v\in\mathcal{S}'(\mathbb{R}^n),~ (s,t)\in\Delta$ and $\beta \in \mathbb{N}^n_{0}$ as well as $x\in \mathbb{R}^n$, we have 
\begin{align}
	\frac{\partial^\beta}{\partial_{x_1}^{\beta_{1}}\cdots\partial_ {x_{n}}^{\beta_{n}}}\left(\mathcal{W}(\cdot; s,t)\ast v\right)(x) = \left(	\frac{\partial^\beta}{\partial_{x_1}^{\beta_{1}}\cdots\partial_ {x_{n}}^{\beta_{n}}}\mathcal{W}(\cdot; s,t)\right)\ast v(x),
\end{align}
since $\mathcal{W}\in\mathcal{S}(\mathbb{R}^n)\otimes C^\infty(\Delta)$ implies that the mapping
$$
\mathbb{R}^n\longrightarrow \mathcal{S}(\mathbb{R}^n),
$$
$$
x\mapsto \mathcal{W}(x-\cdot;s,t),
$$ 
belongs to $ C^\infty(\mathbb{R}^{n},\mathcal{S}(\mathbb{R}^n))$.  
\end{remark}
To avoid confusion, we include the following remark to clarify some notations. In this remark, we use Proposition 5.15 of \cite{folland}.

\begin{remark}\label{notation}
A distribution $h\in\mathcal{S}'(\mathbb{R}^n)$ if and only if there exists a finite set of multi-index pairs $F=\{(\alpha_i,\beta_{i})\}_{i=1}^d \subset \mathbb{N}_0^{2n}$ for some $d\in\mathbb{N}$ such that
\begin{align}
	\|h\|_F\doteq\sup_{\phi\in\mathcal{S}(\mathbb{R}^n)}\{|h(\phi)|; ~~ \sum_{i=1}^{d}\|\phi\|_{\alpha_{i},\beta_{i}}\leq1\}<\infty.
\end{align}
Therefore, for every $h\in\mathcal{S}'(\mathbb{R}^n)$, there exists $ F=\{(\alpha_i,\beta_{i})\}_{i=1}^d\subset \mathbb{N}_0^{2n}$ such that for every $\phi\in\mathcal{S}(\mathbb{R}^n)$, one has
$$
|h(\phi)|\leq \|h\|_F \sum_{i=1}^{d}\|\phi\|_{\alpha_{i},\beta_{i}},
$$
where $\|\cdot\|_{\alpha_{i}, \beta_{i}}$ is a seminorm on $\mathcal{S}(\mathbb{R}^n)$.
\end{remark}

Here, we include the following lemma specifically related to the proof of Part 5 of Proposition \ref{pr2}.

\begin{lemma}\label{upperboundderivative}
Let  $M$ be an $n \times n$  symmetric, positive definite, real valued matrix. Let $ \mathbb{N}_0^{n \times n}$ denote the set of all $n \times n$ matrices with entries in $ \mathbb{N}_0$. Suppose $\alpha \in \mathbb{N}_0^{n \times n}$, and denote its rows by $ \alpha_i = (\alpha_{i1}, \ldots, \alpha_{in}) \in \mathbb{N}_0^n$ for each $ i = 1, 2, \ldots, n$.
 Define the function $g$ by
\begin{align}
	g&: \mathbb{R}^n \longrightarrow \mathbb{R},\notag\\
	g&(\xi) \doteq e^{-\langle \xi, M \xi \rangle}, \qquad \forall \xi \in \mathbb{R}^n.
\end{align}
Then, the following inequality holds:
\begin{align}\label{derivatives66}
\left| \partial^\gamma_\xi g(\xi) \right| 
\leq C_\gamma \sum_{\substack{\alpha \in \mathbb{N}_0^{n \times n} \\ \sum_{j=1}^n \alpha_{ij} \leq\gamma_i}} 
\prod_{i,j=1}^n |m_{ij}|^{\alpha_{ij}} 
\prod_{j=1}^n |\xi_j|^{q_j} \, e^{-\langle \xi, M \xi \rangle}, \quad \forall \xi \in \mathbb{R}^n,~~\forall \gamma\in\mathbb{N}^{n}_{0},
\end{align}
where, for each $j = 1, \dots, n$, the exponent $q_j \in \mathbb{N}_0$ is defined by
$$
q_j \doteq \sum_{i=1}^n \alpha_{ij},
$$
and the total degree satisfies
$$
\sum_{j=1}^n q_{j} =  \sum_{j=1}^n \sum_{i=1}^n\alpha_{ij}= |\gamma|.
$$
\end{lemma}
\begin{proof}
For a fixed multi-index $\gamma = (\gamma_1, \dots, \gamma_n)$, we compute derivatives of $g(\xi) = e^{-\langle \xi, M \xi \rangle}$. First, observe that
$$
\partial_{\xi_i} g(\xi) = -2 (M \xi)_i \, e^{-\langle \xi, M \xi \rangle} = -2 \left( \sum_{j=1}^n m_{ij} \xi_j \right) e^{-\langle \xi, M \xi \rangle},\quad\forall\xi\in\mathbb{R}^{n},~\forall i=1,2,\cdots,n.
$$
Each derivative $\partial_{\xi_i}$ contributes a linear combination of the form $\sum_j m_{ij} \xi_j$, multiplied by the exponential term. Hence, the repeated differentiation produces a polynomial in $\xi$ of degree $|\gamma|$ times $e^{-\langle \xi, M \xi \rangle}$ whose highest-degree term is
\begin{equation}\label{linearcombination}
\prod_{i=1}^n \left( \sum_{j=1}^n m_{ij} \xi_j \right)^{\gamma_i}e^{-\langle \xi,M\xi\rangle},\quad \forall \xi\in\mathbb{R}^{n}.
\end{equation}
On the other hand, we have 
\begin{equation}\label{derivation55}
\left( \sum_{j=1}^n m_{ij} \xi_j \right)^{\gamma_i} = \sum_{\substack{\alpha_i \in \mathbb{N}_0^n \\ |\alpha_i| = \gamma_i}} \binom{\gamma_i}{\alpha_i} \prod_{j=1}^n m_{ij}^{\alpha_{ij}} \xi_j^{\alpha_{ij}},\quad \forall\xi\in\mathbb{R}^{n},~\forall i=1,2,\cdots,n.
\end{equation}

 Hence, by \eqref{derivation55}, one can see that 
\begin{align}\label{derivative11}
\left| \prod_{i=1}^n \left( \sum_{j=1}^n m_{ij} \xi_j \right)^{\gamma_i}\right|&\lesssim_{\gamma} \prod_{i=1}^{n} \left( \sum_{\substack{
\alpha_{i}=(\alpha_{i1},\cdots,\alpha_{in})\in \mathbb{N}_0^n \\
|\alpha_{i}| = \gamma_{i}
}}\prod_{j=1}^{n} \left| m_{ij}\right| ^{\alpha_{ij}} \left| \xi_j\right| ^{\alpha_{ij}} \right)\notag\\
&=\sum_{\substack{\alpha=(\alpha_{1},\alpha_{2},\cdots,\alpha_{n})\in\mathbb{N}^{n\times n}_{0}\\\sum^{n}_{j=1}\alpha_{ij}=\gamma_{i}}}\prod^{n}_{i,j=1}|m_{ij}|^{\alpha_{ij}}\prod^{n}_{j=1}|\xi_{j}|^{\sum^{n}_{i=1}\alpha_{ij}}.
\end{align}
Hence, by multiplying both sides of \eqref{derivative11} by $e^{-\langle \xi, M\xi \rangle}$, we get 	\begin{align}\label{derivative22}	 \left| \prod_{i=1}^n \left( \sum_{j=1}^n m_{ij} \xi_j \right)^{\gamma_i}\right| e^{-\langle \xi,M\xi\rangle} \lesssim_{\gamma} \sum_{\substack{\alpha=(\alpha_{1},\alpha_{2},\cdots,\alpha_{n})\in\mathbb{N}^{n\times n}_{0}\\\sum^{n}_{j=1}\alpha_{ij}=\gamma_{i}}}\prod^{n}_{i,j=1}|m_{ij}|^{\alpha_{ij}}\prod^{n}_{j=1}|\xi_{j}|^{\sum^{n}_{i=1}\alpha_{ij}}e^{-\langle \xi,M\xi\rangle}.
\end{align}
Therefore, relation \eqref{derivative22} leads us to the following 
\begin{align}\label{derivative44}
\left| \partial^{\gamma}_{\xi} g(\xi) \right| 	\lesssim_{\gamma} \sum_{\substack{\alpha=(\alpha_{1},\alpha_{2},\cdots,\alpha_{n})\in\mathbb{N}^{n\times n}_{0}\\\sum^{n}_{j=1}\alpha_{ij}\leq\gamma_{i}}}\prod^{n}_{i,j=1}|m_{i,j}|^{\alpha_{ij}}\prod^{n}_{j=1}|\xi_{j}|^{\sum^{n}_{i=1}\alpha_{ij}}e^{-\langle \xi,M\xi\rangle}.
\end{align}
To conclude \eqref{derivatives66} from \eqref{derivative44}, if for each $\alpha\in\mathbb{N}^{n\times n}_{0}$ in \eqref{derivative44} we define $q_{j}\doteq \sum^{n}_{i=1}\alpha_{ij}$ for each $j=1,2,\cdots,n$,  we can associate a unique $q=(q_{1},q_{2},\cdots,q_{n})$ with $\sum^{n}_{j=1}q_{j}=\sum^{n}_{i,j=1}\alpha_{ij}=|\gamma|$.  
 Hence, the sum in \eqref{derivative44} includes all the terms from \eqref{derivative11} and possibly more (terms with lower degree). Therefore, we obtain
\begin{equation}\label{derivative33}
\left| \partial^{\gamma}_{\xi} g(\xi) \right| 	\lesssim_{\gamma}\
\sum_{\substack{\alpha=(\alpha_{1},\alpha_{2},\cdots,\alpha_{n})\in\mathbb{N}^{n\times n}_{0}\\\sum^{n}_{j=1}\alpha_{ij}\leq\gamma_{i}}}
\prod_{i,j=1}^n |m_{ij}|^{\alpha_{ij}} 
\prod_{j=1}^n |\xi_j|^{q_j}e^{-\langle \xi,M\xi\rangle},~\forall \xi\in \mathbb{R}^{n},~\forall \gamma\in\mathbb{N}^{n}_{0}.
\end{equation}
Now, the proof is complete.

Note that since $\langle \xi, M\xi\rangle$ is a quadratic form, in the derivatives of  $e^{-\langle \xi, M\xi\rangle}$,
there are no terms in the expansion \eqref{derivative33} for which  ${\alpha_{ij}}=0$, $\forall i,j\in\{1,2,\cdots,n\}$. In other words, at least some powers of $|m_{i,j}|$ must appear in each term of the sum $\forall i,j=1,2,\cdots,n.$
\end{proof}

Here, consider the family of linear operators $\{\operatorname{W}_{s,t}\}_{(s,t)\in\Delta}$, acting as
$$
\operatorname{W}_{s,t}:\mathcal{S}'(\mathbb{R}^n)\to\mathcal{S}(\mathbb{R}^n),\quad\forall(s,t)\in\Delta,
$$

$$
\operatorname{W}_{s,t}v=\mathcal{W}(\cdot;s,t)\ast v,\quad\forall v\in\mathcal{S}'(\mathbb{R}^n),\quad\forall(s,t)\in\Delta.
$$

We state Proposition \ref{pr2} concerning this family of operators. 

\begin{proposition}\label{pr2}
The following properties hold:
\begin{itemize}
\item[1.]
$$
\operatorname{W}_{s,t}\in C(\mathcal{S}'(\mathbb{R}^n),\mathcal{S}(\mathbb{R}^n))\,\cap\,C(\mathcal{S}(\mathbb{R}^n),\mathcal{S}(\mathbb{R}^n)),\quad\forall(s,t)\in\Delta;
$$

\item[2.] For every $v\in\mathcal{S}'(\mathbb{R}^n)$, the map $(s,t)\mapsto\operatorname{W}_{s,t}v$ belongs to
$$
C^\infty(\Delta,\mathcal{S}(\mathbb{R}^n));
$$

\item[3.]
$$
\operatorname{W}_{r,s}\operatorname{W}_{s,t}=\operatorname{W}_{r,t},\quad\forall(r,s),(s,t)\in\Delta;
$$

\item[4.]
$$
\partial_t\operatorname{W}_{s,t}v-\mathcal{L}_t\operatorname{W}_{s,t}v=0,\quad\forall v\in\mathcal{S}'(\mathbb{R}^n),\quad\forall(s,t)\in\Delta;
$$
\item[5.]
$$
\lim_{(\delta,\epsilon)\to(0+,0+)}\operatorname{W}_{s-\delta,s+\epsilon}v=v,\quad\forall v\in\mathcal{S}(\mathbb{R}^n),\quad\forall s\in[0,+\infty);
$$
\item[6.]
$$
\lim_{(\delta,\epsilon)\to(0+,0+)}\operatorname{W}_{s-\delta,s+\epsilon}v=v,\quad\forall v\in\mathcal{S}'(\mathbb{R}^n),\quad\forall s\in[0,+\infty).
$$
\end{itemize}
\end{proposition}
\begin{proof}
To prove the first part, we demonstrate the validity of the following two statements.
\begin{itemize}
\item[(i)] 
$$
\operatorname{W}_{s,t}\in C(\mathcal{S}(\mathbb{R}^n),\mathcal{S}(\mathbb{R}^n)),\quad\forall(s,t)\in\Delta,
$$
\item[(ii)] 
$$
\operatorname{W}_{s,t}\in C(\mathcal{S}'(\mathbb{R}^n),\mathcal{S}(\mathbb{R}^n))\quad\forall(s,t)\in\Delta.
$$
\end{itemize}
Let $\{\phi_{m}\}_{m=1}^{\infty} \subset \mathcal{S}(\mathbb{R}^n)$ such that $\phi_{m} \xrightarrow{\mathcal{S}(\mathbb{R}^n)} 0$ when $m \rightarrow \infty$. It means that
\begin{equation}\label{assumption}
	\sup_{x \in \mathbb{R}^n} \left|x^\alpha \partial_x^\beta \phi_{m}(x)\right| \rightarrow 0, ~\text{as}\quad m\rightarrow \infty, \quad\forall\alpha, \beta \in \mathbb{N}_{0}^n.
\end{equation}
Therefore, to prove $(i)$ we need to show that for every $\alpha, \beta \in \mathbb{N}_{0}^n$,
$$
\sup_{x \in \mathbb{R}^n} \left|x^\alpha \partial_x^\beta \operatorname{W}_{s,t}\phi_{m}(x)\right| \rightarrow 0, ~\text{as}\quad m\rightarrow \infty, ~~\forall(s,t)\in\Delta,~~ \forall\alpha, \beta \in \mathbb{N}_{0}^n.
$$
We have
\begin{align}\label{c(s(R))}
	&\sup_{x\in \mathbb{R}^n}\left|x^\alpha \partial_x^\beta \operatorname{W}_{s,t}\phi_{m}(x)\right| =\sup_{x\in \mathbb{R}^n}\left|x^\alpha \partial_x^\beta (\phi_{m}\ast\mathcal{W}(\cdot; s,t)(x))\right|\notag\\
	&=\sup_{x\in \mathbb{R}^n} \left|x^\alpha\right|\left|\partial_x^\beta (\phi_{m}\ast\mathcal{W}(\cdot; s,t)(x))\right|=\sup_{x\in \mathbb{R}^n} \left|x^\alpha\right|\left| \int\limits_{\mathbb{R}^n}(\partial_x^\beta\phi_{m}(x-y))\mathcal{W}(y; s,t) dy\right| \notag\\
	&\leq \sup_{x \in \mathbb{R}^n} \int\limits_{\mathbb{R}^n} |x|^{|\alpha|} \left|\partial_x^\beta \phi_{m}(x-y)\right|\mathcal{W}(y; s,t)dy
	= \sup_{x \in \mathbb{R}^n} \int\limits_{\mathbb{R}^n} |x-y+y|^{|\alpha|} \left|\partial_x^\beta \phi_{m}(x-y)\right|\mathcal{W}(y; s,t)dy\notag\\
	&\leq 2^{|\alpha|} \sup_{x \in \mathbb{R}^n} \int\limits_{\mathbb{R}^n} |x-y|^{|\alpha|} \left|\partial_x^\beta \phi_{m}(x-y)\right|\mathcal{W}(y; s,t)dy+
	2^{|\alpha|} \sup_{x \in \mathbb{R}^n} \int\limits_{\mathbb{R}^n} |y|^{|\alpha|} \left|\partial_x^\beta \phi_{m}(x-y)\right|\mathcal{W}(y; s,t)dy\notag\\
	&\leq 2^{|\alpha|} \sup_{x \in \mathbb{R}^n} \int\limits_{\mathbb{R}^n} \left(1+|x-y|^2\right)^{|\alpha|} \left|\partial_x^\beta \phi_{m}(x-y)\right|\mathcal{W}(y; s,t)dy\notag\\
	&+
	2^{|\alpha|} \sup_{x \in \mathbb{R}^n} \int\limits_{\mathbb{R}^n} \left|\partial_x^\beta \phi_{m}(x-y)\right|\left(1+|y|^2\right)^{|\alpha|}\mathcal{W}(y; s,t)dy, \quad\forall (s,t)\in \Delta.
\end{align}
Here, we define
\begin{align}\label{g}
	\mathbb{R}^n&\longrightarrow \mathcal{S}(\mathbb{R}^n),\notag\\
	g_{m}(x)&\doteq\left(1+|x|^2\right)^{|\alpha|} \left|\partial_x^\beta \phi_{m}(x)\right|,~~~ \forall \alpha, \beta \in \mathbb{N}_{0}^n, ~  \forall m \in \mathbb{N},
\end{align}
and
\begin{align}\label{h}
	\mathbb{R}^n&\longrightarrow \mathcal{S}(\mathbb{R}^n),\notag\\
	h(x)&\doteq \left(1+|x|^2\right)^{|\alpha|}\mathcal{W}(x; s,t),~~~\forall \alpha \in \mathbb{N}_{0}^n, ~~ \forall(s,t)\in \Delta.
\end{align}
Using \eqref{g} and \eqref{h} in \eqref{c(s(R))}, we have
\begin{align}\label{convolution}
	\sup_{x\in \mathbb{R}^n} \left|x^\alpha \partial_x^\beta \operatorname{W}_{s,t}\phi_{m}(x)\right| &\leq 2^{|\alpha|} \sup_{x \in \mathbb{R}^n} \left(g_{m}\ast\mathcal{W}(\cdot; s,t)\right)(x)+ 2^{|\alpha|} \sup_{x \in \mathbb{R}^n} \left| \partial_x^\beta\phi_{m}\right|\ast h(x)\notag\\
	&=2^{|\alpha|} \left(\left\|g_{m}\ast \mathcal{W}(\cdot; s,t)\right\|_{\infty}+\left\||\partial_x^\beta\phi_{m}|\ast h\right\|_{\infty} \right)\notag\\
	&\leq 2^{|\alpha|} \left(\left\|g_{m}\right\|_{\infty} \left\|\mathcal{W}(\cdot; s,t)\right\|_{1}+\left\|\partial_x^\beta\phi_{m}\right\|_{\infty}\left\| h\right\|_{1} \right).
\end{align}
Therefore,
\begin{align*}
	\sup_{x\in\mathbb{R}^n} \left|x^\alpha \partial_x^\beta \operatorname{W}_{s,t}\phi_{m}(x)\right|\rightarrow 0,\quad \forall x\in \mathbb{R}^n,~~ \forall(s,t)\in\Delta,~~ \forall\alpha, \beta \in \mathbb{N}_{0}^n,~~\forall m\in \mathbb{N}.
\end{align*}
This is due to the facts that $\left\|\mathcal{W}(\cdot; s,t)\right\|_{1}<\infty, ~\text{and}~  \|h\|_{1} <\infty$, as well as the assumption \eqref{assumption}.

Now, we prove the statement $(ii)$ of part 1.
Let $\{T_{m}\}_{m=1}^{\infty} \subset \mathcal{S}'(\mathbb{R}^n)$ such that $T_{m} \xrightarrow{\mathcal{S}'(\mathbb{R}^n)} 0$. It means that for every $\phi \in \mathcal{S}(\mathbb{R}^n)$, we have $T_{m}(\phi) \rightarrow 0$ or equivalently, it means
$$
\lim_{m\rightarrow \infty}\|T_{m}\|_{F}=0, \quad \forall m\in\mathbb{N},
$$
for some   $F=\{(\alpha_{i},  \beta_{i})\}\subset \mathbb{N}_0^{2n}$.
Now, we need to show that 
for every $\alpha, \beta \in \mathbb{N}_{0}^n$,
$$
\sup_{x \in \mathbb{R}^n} \left| x^\alpha \partial_x^\beta \operatorname{W}_{s,t}T_{m}(x)\right| \rightarrow 0.
$$
We have
\begin{align}\label{continuoustempered}
	\operatorname{W}_{s,t}T_{m}(x) &= \left(\mathcal{W}(\cdot; s,t)\ast T_{m}\right)(x)=T_{m}\left(\mathcal{W}(x-\cdot;s,t)\right).
\end{align}
In addition, referring to Remark \ref{partialtimesx}, one has
\begin{align*}
	&\partial_x^{\beta} \operatorname{W}_{s,t} T_{m}(x) = \partial_x^{\beta}\left(\mathcal{W}(\cdot; s,t)\ast T_{m}\right)(x)=\left(\partial_x^{\beta}\mathcal{W}(\cdot; s,t)\right)\ast T_{m}(x),\\ &\quad \forall x\in \mathbb{R}^n, ~ \forall (s,t)\in \Delta, ~~\forall \beta \in \mathbb{N}_{0}^n, ~~\forall m\in \mathbb{N}.
\end{align*}
Now, since $T_{m} \in \mathcal{S}'(\mathbb{R}^n)$, it follows from Remark \ref{notation} that there exists a finite set $F=\{(\alpha_i, \beta_{i})\}_{i=1}^d$ for some $d\in\mathbb{N}$ such that
\begin{align*}
	&\sup_{x \in \mathbb{R}^n}\left|x^\alpha\partial_{x}^\beta \operatorname{W}_{s,t}T_{m}(x)\right| = \sup_{x \in \mathbb{R}^n}\left|x^\alpha\left(\partial_x^{\beta}\mathcal{W}(\cdot; s,t)\ast T_{m}\right)(x)\right|
	\\&=\sup_{x \in \mathbb{R}^n}\left|T_{m}(x^\alpha\partial_{x}^\beta \mathcal{W}(x-\cdot;s,t))\right|
	\leq \|T_m\|_{F} \sum_{i=1}^{d}\|(\cdot)^{\alpha}\partial_x^{\beta}\mathcal{W}(\cdot-\colon;s,t)\|_{\alpha_{i},\beta_{i}},\quad\forall  \alpha, \beta\in \mathbb{N}_0^n, ~\forall (s,t)\in\Delta.
\end{align*}
	Therefore, we conclude 
\begin{align*}
	\lim_{m\rightarrow \infty}& \left\| T_{m}((\cdot)^\alpha \partial_\cdot^{\beta} \mathcal{W}(\cdot-\colon;s,t))\right\|_{\infty} 
	\leq \lim_{m\rightarrow \infty} \|T_{m}\|_{F} \sum_{i=1}^d\|(\cdot)^{\alpha}\partial_x^{\beta}\mathcal{W}(\cdot-\colon;s,t)\|_{\alpha_i,\beta_i},\quad \forall \alpha, \beta \in \mathbb{N}_{0}^n,  ~\forall(s,t)\in \Delta.
\end{align*}
Hence, the proof of part $1$ is complete because 
$\lim_{m\rightarrow \infty} \|T_{m}\|_{F}=0$ and $\mathcal{W}(\cdot-\colon;s,t)\in \mathcal{S}(\mathbb{R}^n)$.

The proof of part 2 follows similarly to the reasoning in Remark \ref{partialtimes}.

To prove part 3, since 
$$
\operatorname{W}_{r,s}:\mathcal{S}'(\mathbb{R}^n) \longrightarrow \mathcal{S}(\mathbb{R}^n), \quad \forall (r,s) \in \Delta,
$$
and $$\operatorname{W}_{s,t}v \in \mathcal{S}(\mathbb{R}^n) \subset \mathcal{S}^{'}(\mathbb{R}^n),~~\forall v\in \mathcal{S}^{'}(\mathbb{R}^n),~~\forall (s,t)\in \Delta,$$ then we can calculate $\operatorname{W}_{r,s}(\operatorname{W}_{s,t}v)$. We have
\begin{align*}
	\widehat{\operatorname{W}_{r,s}(\operatorname{W}_{s,t}v)} &= \widehat{\mathcal{W}(\cdot;r,s)\ast \operatorname{W}_{s,t}v} = \widehat{\mathcal{W}(\cdot;r,s)} \widehat{\operatorname{W}_{s,t}v}\\ 
	&=\widehat{\mathcal{W}(\cdot;r,s)}(\widehat{ \mathcal{W}(\cdot;s,t)\ast v}) = \widehat{\mathcal{W}(\cdot;r,s)}\widehat{ \mathcal{W}(\cdot;s,t)}\hat{v},\\
	&\forall (s,t), (r,s)\in \Delta.
\end{align*}
Moreover,
\begin{align*}
	\widehat{\mathcal{W}(\cdot;r,s)}\widehat{ \mathcal{W}(\cdot;s,t)}(\xi) &= \widehat{\mathcal{W}(\xi;r,s)}\widehat{ \mathcal{W}(\xi;s,t)}\\
	&= e^{-\langle [A(s)-A(r)]\xi, \xi \rangle} e^{-\langle [A(t)-A(s)]\xi, \xi \rangle}\\
	&=e^{-\langle [A(t)-A(r)]\xi, \xi \rangle}=	\widehat{\mathcal{W}(\cdot;r,t)},\\
	&\forall (s,t), (r,s)\in \Delta.
\end{align*}
Therefore,
\begin{align*}
	\widehat{\operatorname{W}_{r,s}(\operatorname{W}_{s,t}v)}&= \widehat{\mathcal{W}(\cdot;r,s)}\widehat{ \mathcal{W}(\cdot;s,t)}\hat{v} = \widehat{\mathcal{W}(\cdot;r,t)}\hat{v}\\
	&=\widehat{\operatorname{W}_{r,t}v},\\
	&\forall (s,t), (r,s)\in \Delta.
\end{align*}
Thus,
$$
\operatorname{W}_{r,s}\operatorname{W}_{s,t}=\operatorname{W}_{r,t},~~~\forall (s,t), (r,s)\in \Delta.
$$
In part $4$, we want to prove that for every $v \in \mathcal{S}'(\mathbb{R}^n)$ and for every $(s,t)\in \Delta$ we have 
$$
\partial_{t}\operatorname{W}_{s,t}v-\mathcal{L}_t\operatorname{W}_{s,t}v=0.
$$
Using the definition of the operator $\operatorname{W}_{s,t}$, we have
\begin{equation}\label{n1}
	\partial_{t}\operatorname{W}_{s,t}v-\mathcal{L}_t\operatorname{W}_{s,t}v = \partial_{t}\left(\mathcal{W}(\cdot;s,t)\ast v\right)-\mathcal{L}_t\left(\mathcal{W}(\cdot;s,t)\ast v\right),\quad \forall v \in \mathcal{S}'(\mathbb{R}^n),\quad \forall(s,t)\in \Delta.
\end{equation}
From Remark \ref{partialtimes}, it follows that
\begin{equation}\label{n2}
	\partial_{t}\left(\mathcal{W}(\cdot;s,t)\ast v\right) = \partial_{t}\mathcal{W}(\cdot;s,t)\ast v,\quad \forall v \in \mathcal{S}'(\mathbb{R}^n),\quad \forall(s,t)\in \Delta.
\end{equation}	
In addition, using Remark \ref{partialtimesx}, we have
\begin{align}\label{n3}
	&\mathcal{L}_{t}\left(\mathcal{W}(\cdot; s,t)\ast v\right)= \sum_{i,j=1}^n a_{i,j}(t) \frac{\partial^2}{\partial_{x_i}\partial_{x_j}}\left(\mathcal{W}(\cdot; s,t)\ast v \right)= \left(\sum_{i,j=1}^n a_{i,j}(t)\frac{\partial^2}{\partial_{x_i}\partial_{x_j}}\mathcal{W}(\cdot; s,t)\right)\ast v,\\
	&\quad  \forall v \in \mathcal{S}'(\mathbb{R}^n),~~ \forall(s,t)\in \Delta.\notag 
\end{align}
Plugging \eqref{n2} and \eqref{n3} in \eqref{n1} we obtain the following:
\begin{align}
	\partial_{t}\operatorname{W}_{s,t}v-\mathcal{L}_t\operatorname{W}_{s,t}v&= \partial_{t}\mathcal{W}(\cdot;s,t)\ast v-\sum_{i,j=1}^n a_{i,j}(t)\frac{\partial^2}{\partial_{x_i}\partial_{x_j}}\mathcal{W}(\cdot; s,t)\ast v \notag\\
	&= \left(\partial_{t}\mathcal{W}(\cdot;s,t)-\sum_{i,j=1}^n a_{i,j}(t)\frac{\partial^2}{\partial_{x_i}\partial_{x_j}}\mathcal{W}(\cdot; s,t)\right)\ast v\notag\\
	&=\left(\partial_{t}\mathcal{W}(\cdot;s,t)-\mathcal{L}_{t}\mathcal{W}(\cdot; s,t)\right)\ast v
	=0, \quad  \forall v \in \mathcal{S}'(\mathbb{R}^n),~~ (s,t)\in \Delta.
\end{align}
The justification for the final step in the above computation follows directly from part 2 of Proposition \ref{pr1}.

To prove part $5$, it suffices to show that for every $\hat{f}\in \mathcal{S}(\mathbb{R}^n)$, we have
\begin{equation}\label{Ws}
	\widehat{\operatorname{W}_{s-\delta,s+\epsilon}f}\xrightarrow{\mathcal{S}(\mathbb{R}^n)}\hat{f},\quad\text{when}\quad\epsilon,\delta\rightarrow 0,\quad\forall s\in[0,\infty).
\end{equation}
Before proceeding with the proof, let us note some key points that will be used throught the proof.
Firstly, for any $|\alpha|-$times differentiable complex-valued functions $f$ and $g$,
\begin{equation}\label{leibnitz}
	D^{\alpha}(fg)=\sum_{k\leq\alpha}\binom{\alpha}{k}
	D^kfD^{\alpha-k}g,~~\forall f,g\in C^{|\alpha|}(\mathbb{R}^n),~\forall \alpha\in\mathbb{N}_{0}^n.
\end{equation}
Secondly, in the following calculation, we denote the matrix norm $\left\| a(\tau) \right\|_{n \times n}$ as
\[
\left\| a(\tau) \right\|_{n \times n} \doteq \max_{1 \leq j \leq n} \sum_{i=1}^{n} |a_{i,j}(\tau)|, \quad \forall \tau \in [0,\infty),
\]
where $a\in C^\infty(\mathbb{R}_+,\mathrm{GL}(n))\cap C([0,\infty), \mathrm{GL}(n))$.

Now, applying \eqref{leibnitz}, we obtain the following:
\begin{align}\label{alpgabtaderivative}
	&\sup_{\xi\in\mathbb{R}^{n}}	\left|\xi^{\alpha}\partial_{\xi}^{\beta}\left(   \widehat{\operatorname{W}_{s-\delta,s+\epsilon}f(\xi)}-\hat{f(\xi)}\right) \right| =\sup_{\xi\in\mathbb{R}^{n}}\left| \xi^{\alpha} \partial_{\xi}^{\beta}\left(\widehat{\mathcal{W}(\xi;s-\delta,s+\epsilon)}\hat{ f}-\hat{f}\right)\right|  \notag\\
	&=\sup_{\xi\in\mathbb{R}^{n}}\left| \xi^{\alpha} \partial_{\xi}^{\beta}\left(\widehat{\left( \mathcal{W}(\xi;s-\delta,s+\epsilon) -1\right) }\hat{f(\xi)}\right) \right|
	\leq\sup_{\xi\in\mathbb{R}^{n}}\sum_{\gamma\leq \beta}\binom{\beta}{\gamma}\left|\partial_{\xi}^{\gamma}\widehat{\left( \mathcal{W}(\xi;s-\delta,s+\epsilon) -1\right) }\left| \right| \xi^{\alpha}\partial_{\xi}^{\beta-\gamma}\hat{f(\xi)}\right|,\notag\\
	&\forall\hat{f}\in\mathcal{S}(\mathbb{R}^n),\quad \forall\alpha,\beta\in\mathbb{N}_{0}^n.
\end{align}
In view of \eqref{alpgabtaderivative}, we distinguish between two cases, $\gamma \neq 0$ and $\gamma = 0$. When $\gamma \neq 0$, it follows from Lemma \ref{upperboundderivative} that
\begin{align}\label{gamma}
	&	\left|\partial_{\xi}^{\gamma}\widehat{\left( \mathcal{W}(\xi;s-\delta,s+\epsilon) -1\right) }\right| =	\left|\partial_{\xi}^{\gamma}\widehat{ \mathcal{W}(\xi;s-\delta,s+\epsilon)}\right|\notag\\
	&\lesssim_{\gamma} \sum_{\substack{\alpha' \in \mathbb{N}_0^{n \times n} \\ \sum_{j=1}^n \alpha'_{ij} \leq\gamma_i}} 
	\prod_{i,j=1}^n\left| A_{i,j}(s+\epsilon)-A_{i,j}(s-\delta)\right|  ^{\alpha'_{i,j}}\prod^{n}_{k=1}|\xi_{k}|^{q_{k}}e^{-\langle \xi, [A(s+\epsilon)-A(s-\delta)]\xi\rangle}\notag\\
	&\leq \sum_{\substack{\alpha' \in \mathbb{N}_0^{n \times n} \\ \sum_{j=1}^n \alpha'_{ij} \leq\gamma_i}}\prod^{n}_{i,j=1}\left(\max_{\tau\in [0,s+1]}\left| a_{i,j}(\tau)\right| \right)^{\alpha'_{i,j}} (\epsilon+\delta)^{\alpha'_{i,j}}\prod^{n}_{k=1}(1+|\xi|)^{q_{k}}\notag\\
	&\leq \sum_{\substack{\alpha' \in \mathbb{N}_0^{n \times n} \\ \sum_{j=1}^n \alpha'_{ij} \leq\gamma_i}}\prod^{n}_{i,j=1}\left(\max_{\tau\in [0,s+1]}\left\|  a(\tau)\right\|_{n\times n} \right)^{\alpha'_{i,j}} (\epsilon+\delta)^{\alpha'_{i,j}}(1+|\xi|)^{|q|}\notag\\
	&\lesssim_{\gamma} \sum_{\substack{\alpha' \in \mathbb{N}_0^{n \times n} \\ \sum_{j=1}^n \alpha'_{ij} \leq\gamma_i}}\left(1+\max_{\tau\in [0,s+1]}\left\|  a(\tau)\right\|_{n\times n} \right)^{n^{2}|\gamma|} (\epsilon+\delta)^{\alpha'_{i,j}}(1+|\xi|)^{|\gamma|}\notag\\
	&=\left(1+\max_{\tau\in [0,s+1]}\left\|  a(\tau)\right\|_{n\times n} \right)^{n^{2}|\gamma|}(1+|\xi|)^{|\gamma|}\sum_{\substack{\alpha' \in \mathbb{N}_0^{n \times n} \\ \sum_{j=1}^n \alpha'_{ij} \leq\gamma_i}}(\epsilon+\delta)^{\alpha'_{i,j}},\quad\forall \xi\in\mathbb{R}^{n},~\forall s\in[0,\infty),~\forall \gamma\in\mathbb{N}^{n}_{0}.
\end{align}
Thus,
\begin{align}\label{epsilondelta1}
	&\sup_{\xi\in\mathbb{R}^{n}}	\left|\xi^{\alpha}\partial_{\xi}^{\beta}\left(   \widehat{\operatorname{W}_{s-\delta,s+\epsilon}f(\xi)}-\hat{f(\xi)}\right) \right| \leq \sup_{\xi\in\mathbb{R}^{n}}\sum_{\substack{\gamma\leq \beta\\\gamma\neq 0}}\binom{\beta}{\gamma}\left|\partial_{\xi}^{\gamma}\widehat{\left( \mathcal{W}(\xi;s-\delta,s+\epsilon) -1\right) }\left| \right| \xi^{\alpha}\partial_{\xi}^{\beta-\gamma}\hat{f(\xi)}\right|\notag\\
	&\lesssim_{\gamma}\left(1+\max_{\tau\in [0,s+1]}\left\|  a(\tau)\right\|_{n\times n} \right)^{n^{2}|\gamma|}\sup_{\xi\in\mathbb{R}^n}\left((1+|\xi|)^{|\gamma|}\left| \xi^{\alpha}\partial_{\xi}^{\beta-\gamma}\hat{f(\xi)}\right|\right)\sum_{\substack{\alpha' \in \mathbb{N}_0^{n \times n} \\ \sum_{j=1}^n \alpha'_{ij} \leq\gamma_i}}(\epsilon+\delta)^{\alpha'_{i,j}}\notag\\
	&\lesssim_{\gamma, s, n, \alpha, \beta}\sum_{\substack{\alpha' \in \mathbb{N}_0^{n \times n} \\ \sum_{j=1}^n \alpha'_{ij} \leq\gamma_i}}(\epsilon+\delta)^{\alpha'_{i,j}}.
\end{align}
We now turn to the case $\gamma=0$ as follows:
\begin{align}\label{eta0}
	&\sup_{\xi\in\mathbb{R}^{n}} 
	\left| \widehat{ \mathcal{W}(\xi; s-\delta, s+\epsilon) - 1 } \right| 
	\left| \xi^{\alpha} \partial_{\xi}^{\beta} \widehat{f}(\xi) \right| \notag\\
	&= \sup_{\xi\in\mathbb{R}^{n}} 
	\left| e^{-\langle \xi, [A(s+\epsilon) - A(s-\delta)] \xi \rangle} - 1 \right| 
	\left| \xi^{\alpha} \partial_{\xi}^{\beta} \widehat{f}(\xi) \right| \notag \\
	&= \sup_{\xi\in\mathbb{R}^{n}} 
	\left| \langle \xi, [A(s+\epsilon) - A(s-\delta)] \xi \rangle \right| 
	\cdot \frac{ \left| e^{-\langle \xi, [A(s+\epsilon) - A(s-\delta)] \xi \rangle} - 1 \right| }
	{ \left| \langle \xi, [A(s+\epsilon) - A(s-\delta)] \xi \rangle \right| }
	\left| \xi^{\alpha} \partial_{\xi}^{\beta} \widehat{f}(\xi) \right| \notag \\
	&\leq \left\| 
	\frac{ \left| e^{-\langle \cdot, [A(s+\epsilon) - A(s-\delta)] \cdot \rangle} - 1 \right| }
	{ \left| \langle \cdot, [A(s+\epsilon) - A(s-\delta)] \cdot \rangle \right| }
	\right\|_{\infty} \cdot 
	\sup_{\xi\in\mathbb{R}^{n}} |\xi|^2 \left\| A(s+\epsilon) - A(s-\delta) \right\|_{n\times n}
	\left| \xi^{\alpha} \partial_{\xi}^{\beta} \widehat{f}(\xi) \right| \notag \\
	&\leq \left\| A(s+\epsilon) - A(s-\delta) \right\|_{n\times n} 
 \sup_{\xi\in\mathbb{R}^{n}} |\xi|^{|\alpha| + 2} 
	\left| \partial_{\xi}^{\beta} \widehat{f}(\xi) \right| \notag \\
	&\leq (\epsilon + \delta) \max_{\tau \in [0, s+1]} \| a(\tau) \|_{n\times n} 
 \sup_{\xi\in\mathbb{R}^{n}} |\xi|^{|\alpha| + 2} 
	\left| \partial_{\xi}^{\beta} \widehat{f}(\xi) \right| \notag \\
	&\lesssim_{s,\alpha,\beta} (\epsilon + \delta), \quad \forall \epsilon, \delta \in(0,1). 
\end{align}
Note that based on the following points, the above calculations are valid:
\begin{equation}
	\langle \xi,[A(s+\epsilon)-A(s-\delta)]\xi\rangle>0,\quad \forall \xi\neq 0, ~\forall s\in[0,\infty),~~\forall\epsilon,\delta\in (0,1),
\end{equation}
and
$$
\left\|  \dfrac{1-e^{-\langle\cdot,[A(s+\epsilon)-A(s-\delta)]\cdot\rangle}}{\langle\cdot,[A(s+\epsilon)-A(s-\delta)]\cdot\rangle}\right\|_{\infty}=1,\quad \forall s\in [0,\infty),~\forall \epsilon, \delta \in (0,1),
$$
and
$$
\left| \langle \xi,[A(s+\epsilon)-A(s-\delta)]\xi\rangle\right| \leq |\xi|^{2} \left\| A(s+\epsilon)-A(s-\delta)\right\| _{n\times n},\forall s\in [0,\infty),~\forall\xi \in\mathbb{R}^{n},~\forall \epsilon, \delta \in (0,1),
$$
and
$$
\left\| A(s+\epsilon)-A(s-\delta)\right\| _{n\times n}\leq \int\limits^{s+\epsilon}_{s-\delta}\left\| a(\tau)\right\| _{n\times n}d\tau\leq (\epsilon+\delta)\left( \max_{\tau \in[0,s+1]}\left\| a(\tau)\right\| _{n\times n}\right) <\infty,$$
$$~~~~ \forall s\in [0,\infty),~~\forall\epsilon, \delta \in (0,1),
$$
as well as
$$
\sup_{\xi\in\mathbb{R}^{n}}\left|\xi\right|^{|\alpha|+2}\left| \partial^{\beta}\widehat{f}(\xi)\right|<\infty,\quad\forall \xi\in\mathbb{R}^{n},\quad\forall \alpha,\beta\in \mathbb{N}_{0}^{n},
$$  
given that $\widehat{f}\in \mathcal{S}(\mathbb{R}^n)$.
By substituting \eqref{epsilondelta1} and \eqref{eta0} into \eqref{alpgabtaderivative}, it follows that
\begin{align}\label{epsilon}
	&\sup_{\xi\in\mathbb{R}^{n}}\sum_{\gamma\le\beta}\binom{\beta}{\gamma}\left|\partial^{\gamma}\widehat{\left( \mathcal{W}(\xi;s-\delta,s+\epsilon) -1\right) }\left| \right| \xi^{\alpha}\partial^{\beta-\gamma}\widehat{f}(\xi)\right|\notag\\
	&\lesssim_{\gamma, s, n, \alpha, \beta}\sum_{\substack{\alpha' \in \mathbb{N}_0^{n \times n} \\ \sum_{j=1}^n \alpha'_{ij} \leq\gamma_i}}(\epsilon+\delta)^{\alpha'_{i,j}}+(\epsilon+\delta),
\end{align}
Approaching $\delta$ and $\epsilon$ to zero in \eqref{epsilon}, we obtain the goal. 

Finally, to prove part 6, let $v\in\mathcal{S}'(\mathbb{R}^n)$. Then for every $\phi \in \mathcal{S}(\mathbb{R}^n)$, using the previous part along with parts  3 and 4 of Proposition \ref{pr1}, we have
\begin{align}
	\lim_{(\delta,\epsilon)\to(0+,0+)}\operatorname{W}_{s-\delta,s+\epsilon}v(\phi)&=\lim_{(\delta,\epsilon)\to(0+,0+)}\mathcal{W}(\cdot;s-\delta,s+\epsilon)\ast v(\phi)\notag\\
	&=\lim_{(\delta,\epsilon)\to(0+,0+)}v\left(\mathcal{W}(\cdot;s-\delta,s+\epsilon)\ast \phi\right)\notag\\
	&=v\left(\lim_{(\delta,\epsilon)\to(0+,0+)}\mathcal{W}(\cdot;s-\delta,s+\epsilon)\ast \phi\right)=v(\phi).
\end{align}
Thus, the proof of part 6 and consequently, Proposition \ref{pr2} is complete.
\end{proof}

The operator family \(\{\operatorname{W}_{s,t}\}_{(s,t) \in \Delta}\) allows us to solve the Cauchy problem for the equation (\ref{heatEq}) in the classical sense in terms of the Duhamel's principle. Prior to addressing the solution, it is necessary to discuss some remarks.

\begin{remark}\label{Leibniz}
Assume that $f\in C^\infty(\mathbb{R}_{+})\otimes\mathcal{S}(\mathbb{R}^n)$.
Hence,
\begin{align}
	\partial_{t}^\gamma\left(\int_{0}^{t}\operatorname{W}_{s,t}f(s,x)ds\right)&=\gamma\partial_{t}^{\gamma-1}W_{t,t}f(t,x)+\int_{0}^{t}\partial_{t}^{\gamma}\left(\operatorname{W}_{s,t}f(s,x)\right)ds\notag\\
	&=\gamma\partial_{t}^{\gamma-1}f(t,x)+\int_{0}^{t}\partial_{t}^{\gamma}\left(\operatorname{W}_{s,t}f(s,x)\right)ds,
	\quad\forall x\in \mathbb{R}^n,~~\forall t\in \mathbb{R}_{+}, ~~\forall \gamma\in \mathbb{N}.
\end{align}
\end{remark}

\begin{remark}\label{changeorderforx}
Assume that $f\in  C^\infty(\mathbb{R}_+)\otimes \mathcal{S}(\mathbb{R}^n)\cap C([0,\infty),\mathcal{S}(\mathbb{R}^{n}))$. We claim that 
\begin{align}\label{change4}
	&\partial^{\beta}_{x}\left( \int\limits_{0}^{t}f(s)\ast\partial_{t}^\gamma\mathcal{W}({\cdot;s,t}) (x)ds\right) =\int\limits_{0}^{t}f(s)\ast\left( \partial^{\beta}_{x}\partial_{t}^\gamma\mathcal{W}({\cdot;s,t}) \right) (x)ds,\\
	& \forall x\in \mathbb{R}^n,~~\forall t\in \mathbb{R}_{+},~~\forall \beta\in \mathbb{N}^n_{0},~~ \forall \gamma\in \mathbb{N}_{0}\notag.
\end{align}
To show \eqref{change4}, referring to Theorem 2.27 in \cite{folland}, we need to show the following:
\begin{align}\label{change5}
	&\int\limits_0^t\left|f(s)\ast\partial_{t}^\gamma\mathcal{W}({\cdot;s,t})(x)\right|ds<\infty,
	\quad\forall x\in \mathbb{R}^n,~~\forall t\in \mathbb{R}_{+},~~ \forall \gamma\in \mathbb{N}_{0}.
\end{align}
and the mapping
\begin{equation}\label{change6}
	x\mapsto \left(f(s)\ast\partial_{t}^\gamma\mathcal{W}({\cdot;s,t})\right)(x), ~~\forall x\in \mathbb{R}^n,~~\forall (s,t)\in \Delta,~~ \forall \gamma\in \mathbb{N}_{0},  
\end{equation} 
is $|\beta|-$times differentiable for $\forall \beta\in \mathbb{N}^n_{0}$, as well as, 
\begin{equation}\label{change7}
	\int\limits_0^t\left|f(s)\ast\partial_{x}^\beta\partial_{t}^\gamma\mathcal{W}({\cdot;s,t})(x)\right|ds<\infty,~~ \forall x\in \mathbb{R}^n,~~\forall t\in \mathbb{R}_{+},~~\forall \beta\in \mathbb{N}^n_{0},~~ \forall \gamma\in \mathbb{N}_{0}.
\end{equation}
To prove \eqref{change5},  since $f\in C^\infty(\mathbb{R}_{+})\otimes\mathcal{S}(\mathbb{R}^n)$ and  $\mathcal{W}\in\mathcal{S}(\mathbb{R}^n)\otimes C^\infty(\Delta)$, for $\forall x\in \mathbb{R}^n,$ for $\forall t\in \mathbb{R}_{+}$, and for $\forall \gamma\in \mathbb{N}_{0}$, we have
\begin{align*}
	\int\limits_0^t|f(s)\ast\partial_{t}^\gamma\mathcal{W}({\cdot;s,t})(x)|ds&\leq 	\int\limits_0^t\|f(s)\ast\partial_{t}^\gamma\mathcal{W}({\cdot;s,t})\|_{\infty}ds\\
	&\leq \int\limits_0^t\|f(s)\|_{\infty}\|\partial^{\gamma}_{t}\mathcal{W}({\cdot;s,t}))\|_{1}ds\\
	&\leq \max_{s\in[0,t]}\|f(s)\|_{\infty}\int\limits_0^t \|\partial^{\gamma}_{t}\mathcal{W}({\cdot;s,t})\|_{1}ds<\infty.
\end{align*}
	
To prove \eqref{change6}, observe that $f\in  C^\infty(\mathbb{R}_{+})\otimes\mathcal{S}(\mathbb{R}^n)$  and  $\mathcal{W}\in\mathcal{S}(\mathbb{R}^n)\otimes C^\infty(\Delta)$, which implies that
$$
\partial^{\beta}_{x}\left(f(s)\ast\partial_{t}^\gamma\mathcal{W}({\cdot;s,t})\right)=f(s)\ast\partial^{\beta}_{x}\partial_{t}^\gamma\mathcal{W}({\cdot;s,t})~~,\forall x\in \mathbb{R}^n,~~\forall (s,t)\in \Delta,~~\forall \beta\in \mathbb{N}^n_{0},~~ \forall \gamma\in \mathbb{N}_{0}.
$$
	
To prove \eqref{change7}, we have 
\begin{align*}
	\int\limits_0^t\left|f(s)\ast\partial_{x}^\beta\partial_{t}^\gamma\mathcal{W}({\cdot;s,t})(x)\right|ds&\leq	\int\limits_0^t\left\|f(s)\ast\partial_{x}^\beta\partial_{t}^\gamma\mathcal{W}({\cdot;s,t})\right\|_{\infty}ds\\
	&\leq \int\limits_0^t\|f(s)\|_{\infty}\left\|\partial^{\beta}_{x}\partial_{t}^{\gamma}\mathcal{W}({\cdot;s,t})\right\|_{1}ds\\
	&\leq \max_{s\in[0,t]}\|f(s)\|_{\infty}\int\limits_0^t\left\|\partial^{\beta}_{x}\partial^{\gamma}_{t}\mathcal{W}({\cdot;s,t})\right\|_{1} ds<\infty,\\
	& \forall x\in \mathbb{R}^n,~~\forall t\in \mathbb{R}_{+},~~\forall \beta\in \mathbb{N}^n_{0},~~ \forall \gamma\in \mathbb{N}_{0}.
\end{align*}
\end{remark}

Having established the necessary framework, we are now ready to proceed with the key theorem.

\begin{proposition}\label{pr3}
Let $u_0\in \mathcal{S}'(\mathbb{R}^n)$ and $f\in  C^\infty(\mathbb{R}_+)\otimes \mathcal{S}(\mathbb{R}^n)\cap C([0,\infty), \mathcal{S}(\mathbb{R}^n))$. Then the function $u:\mathbb{R}_+\times\mathbb{R}^n\to\mathbb{C}$ defined by
\begin{equation}\label{solution}
	u(t,x)\doteq\operatorname{W}_{0,t}u_0(x)+\int\limits_0^t\operatorname{W}_{s,t}f(s,x)ds,\quad \forall x\in\mathbb{R}^n,~~ t\in\mathbb{R}_+,
\end{equation}
satisfies
\begin{equation}\label{tensor}
	u\in C^\infty(\mathbb{R}_+)\otimes\mathcal{S}(\mathbb{R}^n)\cap C([0,\infty), \mathcal{S}'(\mathbb{R}^{n})),
\end{equation}
and the Cauchy problem
\begin{equation}\label{cauchy}
\begin{cases}
	\partial_t u(t,x)-\mathcal{L}_{t}u(t,x)=f(t,x),\quad \forall x\in\mathbb{R}^n,~~\forall t\in\mathbb{R}_+,\\
	\lim\limits_{t\to0+}u(t,\cdot)=u_0.
\end{cases}
\end{equation}
Moreover, if $u_0\ge0$ and $f\ge0$ then $u\ge0$.
\end{proposition}
\begin{proof} 
To show \eqref{tensor}, it suffices to prove the following:
\begin{align}\label{seminorm1}
	&\operatorname{W}_{0,\cdot}u_{0} \in C^{\infty}(\mathbb{R}_{+}\times \mathbb{R}^n),
\end{align}
and
\begin{align}\label{seminorm2}
	&\sup_{x\in\mathbb{R}^n}\sup_{t\in K\subset \mathbb{R}_{+}}\left|x^{\alpha}\partial^{\beta}_{x}\partial^{\gamma}_{t}\operatorname{W}_{0,t}u_{0}(x)\right|<\infty,\quad\forall \alpha,\beta\in \mathbb{N}^{n}_{0},~~\forall \gamma\in \mathbb{N}_{0}, ~~\forall~\text{compact}~K\subset \mathbb{R}_{+},
\end{align}
as well as 
\begin{align}\label{seminorm3}
	&\int\limits_0^\cdot\operatorname{W}_{s,\cdot}f(s,:)ds\in C^{\infty}(\mathbb{R}_{+}\times \mathbb{R}^n),
\end{align}
and
\begin{align}\label{seminorm4}
	\sup_{x\in\mathbb{R}^n}&\sup_{t\in K\subset \mathbb{R}_{+}}\left|x^{\alpha}\partial^{\beta}_{x}\partial^{\gamma}_{t}\int\limits_0^t\operatorname{W}_{s,t}f(s,x)ds\right|<\infty,\quad\forall \alpha,\beta\in \mathbb{N}^{n}_{0},~~\forall \gamma\in \mathbb{N}_{0},~~\forall~ \text{compact}~K\subset \mathbb{R}_{+}.
\end{align}
To show \eqref{seminorm1}, one can see that 
\begin{align*}
	\partial^{\beta}_{x}\operatorname{W}_{0,t}u_{0}(x)&=\partial^{\beta}_{x}\left((\mathcal{W}(\cdot; 0,t)\ast u_{0})(x)\right)=\partial^{\beta}_{x}(u_{0}(\mathcal{W}(x-\cdot;0,t))\\&=u_{0}(\partial^{\beta}_{x}\mathcal{W}(x-\cdot;0,t))=(u_{0}\ast\partial^{\beta}_{x}\mathcal{W}(\cdot;0,t))(x),\\
	&\quad \forall x\in \mathbb{R}^n,~ \forall t\in \mathbb{R}_{+}, ~\forall \beta \in \mathbb{N}^{n}_{0}.
\end{align*}
Therefore, 
$\operatorname{W}_{0,t}u_{0}\in C^{\infty}(\mathbb{R}^{n}),~~\forall t\in\mathbb{R}_{+}.$

Moreover, we have
\begin{align*}
	\partial^{\gamma}_{t}\operatorname{W}_{0,t}u_{0}(x)&=\partial^{\gamma}_{t}\left((\mathcal{W}(\cdot; 0,t)\ast u_{0})(x)\right)=\partial^{\gamma}_{t}(u_{0}(\mathcal{W}(x-\cdot;0,t))\\
	&=u_{0}(\partial^{\gamma}_{t}\mathcal{W}(x-\cdot;0,t))=(u_{0}\ast\partial^{\gamma}_{t}\mathcal{W}(\cdot;0,t))(x),\\
	&\forall x\in \mathbb{R}^n,~ \forall t\in \mathbb{R}_{+},~~\forall \gamma\in\mathbb{N}_{0}. 
\end{align*}
Thus, we have 
$\operatorname{W}_{0,\cdot}u_{0}(x)\in C^{\infty}(\mathbb{R}_{+}),~~\forall x\in\mathbb{R}^{n}.$

To prove \eqref{seminorm2}, observe that since \(\mathcal{W} \in \mathcal{S}(\mathbb{R}^n) \otimes C^\infty(\Delta)\), the mapping \(t \mapsto \mathcal{W}(x - \cdot; 0, t) \in \mathcal{S}(\mathbb{R}^n)\) belongs to \(C^\infty(\mathbb{R}_+, \mathcal{S}(\mathbb{R}^n))\). Therefore, for every compact $K\subset \mathbb{R}_+$, one has
\begin{align*}
	&\sup_{x\in\mathbb{R}^n}\sup_{t\in K\subset \mathbb{R}_{+}}\left|x^{\alpha}\partial^{\beta}_{x}\partial^{\gamma}_{t}\operatorname{W}_{0,t}u_{0}(x)\right|\\&=\sup_{x\in\mathbb{R}^n}\sup_{t\in K\subset \mathbb{R}_{+}}\left|x^{\alpha}\partial^{\beta}_{x}\partial^{\gamma}_{t}((\mathcal{W}(\cdot;0,t)\ast u_{0})(x))\right|\\
	&=\sup_{x\in\mathbb{R}^n}\sup_{t\in K\subset \mathbb{R}_{+}}\left|x^{\alpha}\partial^{\beta}_{x}\partial^{\gamma}_{t}(u_{0}(\mathcal{W}(x-\cdot;0,t)))\right|\notag\\
	&=\sup_{x\in\mathbb{R}^n}\sup_{t\in K\subset \mathbb{R}_{+}}\left|x^{\alpha}\partial^{\beta}_{x}(u_{0}(\partial^{\gamma}_{t}\mathcal{W}(x-\cdot;0,t)))\right|.
\end{align*}
Moreover, the fact that $\mathcal{W}\in\mathcal{S}(\mathbb{R}^n)\otimes C^\infty(\Delta)$ implies that the mapping $x\mapsto \partial^{\gamma}_{t}\mathcal{W}(x-\cdot,0,t)\in \mathcal{S}(\mathbb{R}^n)$ belongs to $C^{\infty}(\mathbb{R}^{n},\mathcal{S}(\mathbb{R}^{n}))$. Now, by continuing the above computation, we get
\begin{align*}
	\sup_{x\in\mathbb{R}^n}\sup_{t\in K\subset \mathbb{R}_{+}}\left|x^{\alpha}\partial^{\beta}_{x}\partial^{\gamma}_{t}\operatorname{W}_{0,t}u_{0}(x)\right|&=\sup_{x\in\mathbb{R}^n}\sup_{t\in K\subset \mathbb{R}_{+}}\left|x^{\alpha}\partial^{\beta}_{x}(u_{0}(\partial^{\gamma}_{t}\mathcal{W}(x-\cdot;0,t)))\right|\\
	&=\sup_{x\in\mathbb{R}^n}\sup_{t\in K\subset \mathbb{R}_{+}}\left|u_{0}(x^{\alpha}\partial^{\beta}_{x}\partial^{\gamma}_{t}\mathcal{W}(x-\cdot;0,t))\right|.
\end{align*}
Since $u_0\in\mathcal{S}'(\mathbb{R}^n)$, using Remark \ref{notation}, there exits a finite set $F=\{(\alpha_{i}, \beta_{i})\}_{i=1}^d \subset \mathbb{N}_0^{2n}$ for some $d\in \mathbb{N}$ such that
\begin{align}\label{seminormu0}
	|u_{0}(x^{\alpha}\partial^{\beta}_{x}\partial^{\gamma}_{t}\mathcal{W}&(x-\colon;0,t))|\leq \|u_{0}\|_{F} \sum_{i=1}^d\left\|(\cdot)^{\alpha}\partial_{x}^{\beta}\partial_{t}^{\gamma}\mathcal{W}(\cdot-\colon;0,t)\right\|_{\alpha_i,\beta_i},\notag\\
	&\quad \forall x\in\mathbb{R}^n, ~ \forall t\in\mathbb{R}_+,~ \forall \alpha, \beta \in \mathbb{N}_0^n,~ \forall\gamma\in\mathbb{N}_0.
\end{align}
From \eqref{seminormu0}, it follows that 
\begin{align}\label{supsup}		\sup_{x\in\mathbb{R}^n} \sup_{t\in K\subset \mathbb{R}_{+}}	|u_{0}(x^{\alpha}\partial^{\beta}_{x}\partial^{\gamma}_{t}\mathcal{W}&(x-\cdot;0,t))| \leq  \|u_{0}\|_{F} \left(\sup_{t\in K\subset \mathbb{R}_{+}} \sum_{i=1}^d\left\|(\cdot)^{\alpha} \partial_{x}^{\beta}\partial_{t}^{\gamma}\mathcal{W}(\cdot-\colon;0,t)\right\|_{\alpha_i,\beta_i}\right) <\infty,\\
&\forall \alpha, \beta \in \mathbb{N}_{0}^n, ~\forall \gamma\in \mathbb{N}_0, ~\forall K\subset \mathbb{R}_+ ~\text{compact}\notag.
\end{align}
This completes the proof of \eqref{seminorm2}.
To show \eqref{seminorm3}, referring to \ref{Leibniz}, one has 
\begin{align*}
	\partial^{\gamma}_{t}\int\limits^{t}_{0}\operatorname{W}_{s,t}f(s,x)ds&=\partial^{\gamma}_{t}\int\limits^{t}_{0}(\mathcal{W}(\cdot; s,t)\ast f(s))(x)ds\\
	&=\gamma\partial_{t}^{\gamma-1}f(t,x)+\int_{0}^{t}\partial_{t}^{\gamma}\left(\operatorname{W}_{s,t}f(s,x)\right)ds\\
	&=\gamma\partial_{t}^{\gamma-1}f(t,x)+\int_{0}^{t}\partial_{t}^{\gamma}\left(\mathcal{W}(\cdot;s,t)\ast f(s)\right)(x)ds\\
	&=\gamma\partial_{t}^{\gamma-1}f(t,x)+\int_{0}^{t}f(s)\ast\partial^{\gamma}_{t}\mathcal{W}(\cdot;s,t)(x)ds,\\
	&\forall x\in \mathbb{R}^{n},~~\forall t\in \mathbb{R}_{+},~~\forall \gamma\in \mathbb{N}_{0}.
\end{align*}
Therefore, 
$\int\limits^{\cdot}_{0}\operatorname{W}_{s,\cdot}f(s,x)ds\in C^{\infty}(\mathbb{R}_{+}),~~\forall x\in\mathbb{R}^{n}$.

Moreover, referring to Remark \ref{changeorderforx}, we have
\begin{align*}
	\partial^{\beta}_{x}\int\limits^{t}_{0}\operatorname{W}_{s,t}f(s,x)ds&=\partial^{\beta}_{x}\int\limits^{t}_{0}(f(s)\ast\mathcal{W}(\cdot;s,t))(x)ds\\
	&=\int_{0}^{t}\partial_{x}^{\beta}\left(f(s)\ast\mathcal{W}(\cdot;s,t)(x)\right) ds\\
	&=\int_{0}^{t}f(s)\ast\partial^{\beta}_{x}\mathcal{W}(\cdot;s,t)(x)ds,\\
	&\forall x\in \mathbb{R}^{n},~~\forall t\in \mathbb{R}_{+},~~\forall \beta\in \mathbb{N}^{n}_{0}.
\end{align*}
Thus, 
$\int\limits^{t}_{0}\operatorname{W}_{s,t}f(s,\cdot)ds\in C^{\infty}(\mathbb{R}^{n}),~~\forall t\in\mathbb{R}_{+}$.

To prove \eqref{seminorm4}, one has
\begin{align*}
	&\sup_{x\in\mathbb{R}^n} \sup_{t\in K\subset \mathbb{R}_{+}}\left|x^{\alpha}\partial^{\beta}_{x}\partial^{\gamma}_{t}\int\limits_0^t\operatorname{W}_{s,t}f(s,x)ds\right|\\
	&=\sup_{x\in\mathbb{R}^n} \sup_{t\in K\subset \mathbb{R}_{+}}\left|x^{\alpha}\partial^{\beta}_{x}\partial^{\gamma}_{t}\int\limits_0^tf(s)\ast\mathcal{W}({\cdot;s,t}) (x)ds\right|\\
	&=\sup_{x\in\mathbb{R}^n} \sup_{t\in K\subset \mathbb{R}_{+}}\left|x^{\alpha}\partial^{\beta}_{x}\left[\gamma\partial^{\gamma-1}_{t}f(t,x)+\int\limits_0^tf(s)\ast\partial_{t}^\gamma\mathcal{W}({\cdot;s,t}) (x)ds\right]\right|\\
	&=\sup_{x\in\mathbb{R}^n}\sup_{t\in K\subset \mathbb{R}_{+}}\left|\gamma x^{\alpha}\partial^{\beta}_{x}\partial_{t}^{\gamma-1}f(t,x)+x^{\alpha}\partial^{\beta}_{x}\int_{0}^{t}f(s)\ast\partial_{t}^\gamma\mathcal{W}({\cdot;s,t}) (x)ds\right|\notag\\
	&\leq\sup_{x\in\mathbb{R}^n} \sup_{t\in K\subset \mathbb{R}_{+}}\left|\gamma x^{\alpha}\partial^{\beta}_{x}\partial_{t}^{\gamma-1}f(t,x)\right| +\sup_{x\in\mathbb{R}^n} \sup_{t\in K\subset \mathbb{R}_{+}}\left| x^{\alpha}\partial^{\beta}_{x}\int_{0}^{t}f(s)\ast\partial_{t}^\gamma\mathcal{W}({\cdot;s,t}) (x)ds\right|\\	
	&=\sup_{x\in\mathbb{R}^n} \sup_{t\in K\subset \mathbb{R}_{+}}\left|\gamma x^{\alpha}\partial^{\beta}_{x}\partial_{t}^{\gamma-1}f(t,x)\right| +\sup_{x\in\mathbb{R}^n} \sup_{t\in K\subset \mathbb{R}_{+}}\left| \int_{0}^{t}f(s)\ast x^{\alpha}\partial^{\beta}_{x}\partial_{t}^\gamma\mathcal{W}({\cdot;s,t}) (x)ds\right|\\
	&\forall \alpha, \beta\in \mathbb{N}^{n}_{0}, ~\forall\gamma\in\mathbb{N}_0,
\end{align*}
where in the last lines we are allowed to interchange the order of differentiation and integration referring to Remark \ref{Leibniz} and Remark \ref{changeorderforx}. Moreover, since $\mathcal{W}\in\mathcal{S}(\mathbb{R}^n)\otimes C^\infty(\Delta)$ implies the following:

the mapping $$t\mapsto \mathcal{W}(x-\cdot;0,t)\in \mathcal{S}(\mathbb{R}^n),~~\forall x\in \mathbb{R}^n,$$ belongs to $ C^\infty(\mathbb{R}_{+},\mathcal{S}(\mathbb{R}^n)),$ as well as the mapping $$x\mapsto \partial^{\gamma}_{t}\mathcal{W}(x-\cdot;0,t)\in \mathcal{S}(\mathbb{R}^n),~~\forall t\in \mathbb{R}_{+},~~ \forall\gamma\in \mathbb{N}_{0},$$ 
belongs to $ C^\infty(\mathbb{R}^{n},\mathcal{S}(\mathbb{R}^n))$. One may continue
\begin{align}
	&\sup_{x\in\mathbb{R}^n} \sup_{t\in K\subset \mathbb{R}_{+}}\left|x^{\alpha}\partial^{\beta}_{x}\partial^{\gamma}_{t}\int\limits_0^t\operatorname{W}_{s,t}f(s,x)ds\right|\notag\\
	&\leq \sup_{x\in\mathbb{R}^n} \sup_{t\in K\subset \mathbb{R}_{+}}\left|\gamma x^{\alpha}\partial^{\beta}_{x}\partial_{t}^{\gamma-1}f(t,x)\right| +\sup_{x\in\mathbb{R}^n} \sup_{t\in K\subset \mathbb{R}_{+}}\left| \int_{0}^{t}f(s)\ast x^{\alpha}\partial^{\beta}_{x}\partial_{t}^\gamma\mathcal{W}({\cdot;s,t}) (x)ds\right|\notag\\
	&\leq \sup_{x\in\mathbb{R}^n} \sup_{t\in K\subset \mathbb{R}_{+}}\left|\gamma x^{\alpha}\partial^{\beta}_{x}\partial_{t}^{\gamma-1}f(t,x)\right| +\sup_{t\in K\subset \mathbb{R}_{+}}\int_{0}^t\|f(s)\|_{\infty} \|(\cdot)^{\alpha}\partial^{\beta}_{x}\partial_{t}^\gamma\mathcal{W}({\cdot;s,t})\|_{1}ds\notag\\
	&\leq   \sup_{x\in\mathbb{R}^n}\sup_{t\in K\subset \mathbb{R}_{+}}\left|\gamma x^{\alpha}\partial^{\beta}_{x}\partial_{t}^{\gamma-1}f(t,x)\right| +\sup_{s\in [0,\max K]}\|f(s)\|_{\infty}\sup_{t\in K\subset \mathbb{R}_{+}}\int\limits^{t}_{0}\|(\cdot)^{\alpha}\partial^{\beta}_{x}\partial_{t}^\gamma\mathcal{W}({\cdot;s,t})\|_{1}ds <\infty,\notag\\
	&\forall \alpha, \beta\in \mathbb{N}^{n}_{0},~ \forall\gamma\in\mathbb{N}_0,\notag
\end{align}
where the estimate
$$
\sup_{t \in K\subset \mathbb{R}_{+}} \int_0^t \left\|(\cdot)^{\alpha} \partial_x^{\beta} \partial_t^{\gamma} \mathcal{W}(\cdot; s, t) \right\|_{1} ds < \infty,
$$
holds because $ \mathcal{W} \in \mathcal{S}(\mathbb{R}^n) \otimes C^\infty(\Delta)$, implying that the map 
$$
(s, t) \mapsto\left\|(\cdot)^{\alpha} \partial_x^{\beta} \partial_t^{\gamma} \mathcal{W}(\cdot; s, t) \right\|_{1},\quad\forall \alpha, \beta\in \mathbb{N}^{n}_{0}, \forall \gamma\in\mathbb{N}_0
$$
is continuous on $\Delta $. Thus, for fixed $t \in \mathbb{R}_+ $, the integrand is continuous in $ s $. The compactness of $K \subset \mathbb{R}_+$ then ensures finiteness of the supremum. Thus, the proof of \eqref{seminorm4} is complete.

To show that \eqref{tensor} satisfies \eqref{cauchy} by using part $5$ in Proposition \ref{pr2}, we may compute $\partial _{t}u(t,x)$ and $\mathcal{L}_tu(t,x)$ separately. Namely,
\begin{align}\label{utsatisfying}
	\partial _{t}u(t,x)&=\partial_{t}\operatorname{W}_{0,t}u_{0}(x)+\partial_{t}\left(\int\limits_0^t\operatorname{W}_{s,t}f(s,x)ds\right)\notag\\
	&=\partial_{t}(\mathcal{W}(\cdot;0,t)\ast u_{0})(x)+\operatorname{W}_{t,t}f(t,x)+\int\limits_0^t\frac{\partial}{\partial t}\operatorname{W}_{s,t}f(s,x)ds\notag\\
	&=\partial_{t}u_{0}\left(\mathcal{W}(x-\cdot;0,t)\right)+f(t,x)+\int\limits_{0}^t\frac{\partial}{\partial t}(\mathcal{W}(\cdot;s,t)\ast f(s)(x))ds\notag\\
	&=\partial_{t}u_{0}(\mathcal{W}(x-\cdot;0,t)+f(t,x)+\int\limits_{0}^t\frac{\partial}{\partial t}(f(s)(\mathcal{W}(x-\cdot;s,t)))ds\notag\\
	&=u_{0}(\partial_{t}\mathcal{W}(x-\cdot;0,t))+f(t,x)+\int\limits_{0}^tf(s)(\partial_{t}\mathcal{W}(x-\cdot;s,t))ds\notag\\
	&=\partial_{t}\mathcal{W}(\cdot;0,t)\ast u_{0}(x) +f(t,x)+\int\limits_{0}^t\partial_{t}\mathcal{W}(\cdot;s,t)\ast f(s,\cdot)(x)ds,\quad\forall x\in\mathbb{R}^n,~ \forall t\in\mathbb{R}_+.
\end{align}
To compute $\mathcal{L}_{t}u(t,x)$, we use Remark \ref{partialtimesx} and Remark \ref{changeorderforx}. Accordingly, we can write
\begin{align}\label{ltsatisfying}
	\mathcal{L}_{t}u(t,x)&=\mathcal{L}_{t}(\operatorname{W}_{0,t}u_{0}(x))+\mathcal{L}_{t}\left(\int\limits_0^t\operatorname{W}_{s,t}f(s,x)ds\right)\notag\\
	&=\mathcal{L}_{t}(\mathcal{W}(\cdot;0,t)\ast u_{0}(x))+\mathcal{L}_{t}\left(\int\limits_0^t\mathcal{W}(\cdot;s,t)\ast f(s,\cdot)(x)ds\right)\notag\\
	&=\sum_{i,j=1}^{n} a_{ij}(t)\partial^2_{x_ix_j}(\mathcal{W}(\cdot;0,t)\ast u_{0}(x))+\sum_{i,j=1}^{n} a_{ij}(t)\partial^2_{x_ix_j}\left(\int\limits_0^t\mathcal{W}(\cdot;s,t)\ast f(s,\cdot)(x)ds\right)\notag\\
	&=\sum_{i,j=1}^{n} a_{ij}(t)(\partial^2_{x_ix_j}\mathcal{W}(\cdot;0,t))\ast u_{0}(x)+\sum_{i,j=1}^{n} a_{ij}(t)\int\limits_0^t\partial^2_{x_ix_j}\left( \mathcal{W}(\cdot;s,t)\ast f(s,\cdot)(x)\right) ds\notag\\
	&=\left(\sum_{i,j=1}^{n} a_{ij}(t)\partial^2_{x_ix_j}\mathcal{W}(\cdot;0,t)\ast u_{0}\right)(x)+\int\limits_0^t\sum_{i,j=1}^{n} a_{ij}(t)(\partial^2_{x_ix_j}\mathcal{W}(\cdot;s,t))\ast f(s,\cdot)(x)ds\notag\\
	&=\mathcal{L}_{t}\mathcal{W}(\cdot,0,t)\ast u_{0}(x)+\int\limits_0^t\mathcal{L}_{t}\mathcal{W}(\cdot,s,t)\ast f(s,\cdot)(x)ds,\quad\forall x\in\mathbb{R}^n,~~ \forall t\in\mathbb{R}_+.
\end{align}
Now, by using \eqref{utsatisfying} and \eqref{ltsatisfying}, we conclude that
\begin{align}
	\partial_{t}u(t,x)-\mathcal{L}_{t}u(t,x)&=\partial_{t}\mathcal{W}(\cdot;0,t)\ast u_{0}(x) +f(t,x)+\int\limits_{0}^t\partial_{t}\mathcal{W}(\cdot;s,t)\ast f(s,\cdot)(x)ds\notag\\
	&-\mathcal{L}_{t}\mathcal{W}(\cdot,0,t)\ast u_{0}(x)-\int\limits_0^t\mathcal{L}_{t}\mathcal{W}(\cdot,s,t)\ast f(s,\cdot)(x)ds\notag\\
	&=(\partial_{t}\mathcal{W}(\cdot;0,t)-\mathcal{L}_{t}\mathcal{W}(\cdot;0,t))\ast u_{0}(x)+f(t,x)\notag\\
	&+\int\limits_{0}^t(\partial_{t}\mathcal{W}(\cdot;s,t)-\mathcal{L}_{t}\mathcal{W}(\cdot;s,t))\ast f(s,\cdot)(x)ds\notag\\
	&=f(t,x), \quad\forall x\in\mathbb{R}^n,~\forall t\in\mathbb{R}_+.
\end{align}
The last line follows from part 2 of Proposition \ref{pr1}.

Additionally, we have 
\begin{equation*}
	\lim\limits_{t\to0+}u(t,\cdot)=u_0, \quad\forall t\in \mathbb{R}_{+}.
\end{equation*}
The reason is that when we take limit from \eqref{solution} as $t$ goes to $0+$, using part $5$ of Proposition \ref{pr2}, we have 
\begin{align}
	\lim_{t\rightarrow 0^{+}}u(t,\cdot)&=\lim_{t\rightarrow 0^{+}}\operatorname{W}_{0,t}u_0+\lim_{t\rightarrow 0^{+}}\int\limits_0^t\operatorname{W}_{s,t}f(s,\cdot)ds\notag\\
	&=\operatorname{W}_{0,0}u_0=u_{0},\quad\forall t\in\mathbb{R}_+.
\end{align} 
To verify the non-negativity property of the solution $u$,  we note that positivity of $$\mathcal{W}(x; s, t)>0,~~ \forall x \in \mathbb{R}^n, ~ \forall(s, t) \in \Delta,$$ as well as non-negativity of $u_{0}$ and $f$ imply the following: 
$$
\operatorname{W}_{0,t} u_{0}(x) = u_{0}(\mathcal{W}(x - \cdot; 0, t)) \geq 0,\quad\forall x\in\mathbb{R}^n,~~\forall t\in\mathbb{R}_{+},
$$
as well as  
$$
\int_{0}^{t} \operatorname{W}_{s,t} f(s, x) ds = \int_{0}^{t} \mathcal{W}(\cdot; s, t) \ast f(s)(x) ds \geq 0,\quad\forall x\in\mathbb{R}^n,~~\forall t\in\mathbb{R}_{+}.
$$
Thus, by \eqref{solution} the solution $u$  remains non-negative. Thus, the proof is complete. 
\end{proof}

We now derive a collection of estimates, including the energy estimate and $L^p$-$L^q$
bounds, which will be used in subsequent arguments.

\begin{remark}\label{energyestim}
Let $u(t)\in L^{p}(\mathbb{R}^n)\cap \mathcal{D}om(\mathcal{L}_t)$ for $1<p<\infty$ and $ t\in \mathbb{R}_{+}$ and define the energy function as follows: \begin{equation}\label{energy}
	E(t)\doteq \|u(t)\|^{p}_{p}=\int_{\mathbb{R}^n}|u(t,x)|^{p}dx, \quad \forall t\in\mathbb{R}_+.
\end{equation}
Then, taking derivative with respect to $t$ from both sides of  \eqref{energy}, we have 
\begin{align}
	\frac{d}{dt}E(t)&=p\int_{\mathbb{R}^n}|u(t,x)|^{p-1}\frac{u(t,x)}{|u(t,x)|}\partial_{t}u(t,x)dx\notag\\
	&=p\int_{\mathbb{R}^n}|u(t,x)|^{p-2}u(t,x)\partial_{t}u(t,x)dx, \quad \forall t\in\mathbb{R}_+.\notag
\end{align}
Now, multiplying both sides of the heat equation \eqref{cauchy} by $p|u(t,x)|^{p-2}u(t,x)$ and then integrating them over $\mathbb{R}^n$ with respect to $x$, it follows that
\begin{align}\label{energyestimate1}
	&p\int_{\mathbb{R}^n}|u(t,x)|^{p-2}u(t,x)\partial_{t}u(t,x)dx-p\sum_{i,j=1}a_{i,j}(t)\int_{\mathbb{R}^n}|u(t,x)|^{p-2}u(t,x)\frac{\partial^{2}}{\partial x_{i}\partial x_{j}}u(t,x)dx\notag\\&=\int_{\mathbb{R}^n}f(t,x)|u(t,x)|^{p-2}u(t,x)dx,\quad\forall t\in\mathbb{R}_+.
\end{align}
We see that the first term on the left hand side of \eqref{energyestimate1}, is exactly $\frac{d}{dt}E(t)$. To discuss the second term on the left hand side of \eqref{energyestimate1}, using integration by parts, we have
\begin{equation}
	\int\limits_{\mathbb{R}^n}|u(t,x)|^{p-2}u(t,x)\frac{\partial^{2}}{\partial x_{i}\partial x_{j}}u(t,x)dx=-\int\limits_{\mathbb{R}^n}\frac{\partial}{\partial x_{i}}u(t,x)\frac{\partial}{\partial x_{j}}\left(|u(t,x)|^{p-2}u(t,x)\right)dx,\quad\forall t\in\mathbb{R}_+.
\end{equation}
Since 
\begin{equation*}
	\frac{\partial}{\partial x_{j}}\left(|u(t,x)|^{p-2}u(t,x)\right)=(p-2)|u(t,x)|^{p-2}\frac{\partial}{\partial x_{j}}u(t,x),
\end{equation*}
we have
\begin{align}
	&\sum_{i,j=1}a_{i,j}(t)\int_{\mathbb{R}^n}|u(t,x)|^{p-2}u(t,x)\frac{\partial^{2}}{\partial x_{i}\partial x_{j}}u(t,x)dx\notag\\
	&=-(p-1)\sum_{i,j=1}a_{i,j}(t)\int_{\mathbb{R}^n}\frac{\partial}{\partial{x_i}}u(t,x)|u(t,x)|^{p-2}\frac{\partial}{\partial{x_j}}u(t,x)dx,\quad\forall t\in\mathbb{R}_+.
\end{align}
It is worth noting that in the integration by parts calculation, we used the theorem in Exercise 1.16 of \cite{Ruzhansky} and the point that $u\in \mathcal{D}om(\mathcal{L}_t)$.

Considering the source term $f=0$, from \eqref{energyestimate1} one can obtain 
\begin{align}\label{nonincreasingenergy1}
	\frac{d}{dt} E(t)=-p(p-1)\sum_{i,j=1}a_{i,j}(t)\int_{\mathbb{R}^n}\frac{\partial}{\partial{x_i}}u(t,x)|u(t,x)|^{p-2}\frac{\partial}{\partial{x_j}}u(t,x)dx<0,\quad\forall t\in\mathbb{R}_+.
\end{align}
Therefore, the energy function is non-increasing.
Integrating both sides of \eqref{nonincreasingenergy1} with respect to $t$, we have
\begin{equation}
	E(t)-E(0)\leq 0, ~\forall t\in\mathbb{R}_+.
\end{equation}
Hence,
\begin{align}\label{nonincreasingenergy}
	\|u(t,\cdot)\|_p\leq \|u_0\|_p,\quad\forall t\in\mathbb{R}_+,~\forall 1<p<\infty.
\end{align}
\end{remark}

To have $L^p$-$L^q$ estimates for the operator family $\{\operatorname{W}_{s,t}\}_{(s,t)\in\Delta}$, we need the following remark.

\begin{remark}\label{normpkernel}
For every $(s,t)\in\Delta$, we have the identity
$$
\|\mathcal{W}(\cdot;s,t)\|_p=\frac1{p^{\frac{n}{2p}}\left((4\pi)^n\det[A(t)-A(s)]\right)^{\frac{p-1}{2p}}},\quad\forall p\in[1,+\infty],
$$
where $p=+\infty$ is understood as the limit,
$$
\|\mathcal{W}(\cdot;s,t)\|_\infty=\frac1{\sqrt{(4\pi)^n\det[A(t)-A(s)]}}.
$$
This can be easily obtained in view of
$$\|\mathcal{W}(\cdot;s,t)\|_p^p=\int\limits_{\mathbb{R}^n} \dfrac{1}{\sqrt{(4\pi)^{np}(\det[A(t)-A(s)])^p}}e^{-\frac{p}{4}\langle x,[A(t)-A(s)]^{-1}x\rangle}dx,
$$
using the calculation in the proof of part $1$ of Proposition \ref{pr1}. 
\end{remark}
This leads to the following estimates.

\begin{proposition}\label{pr4}
Let $1\le p,q,r\le+\infty$ be such that
$$
\frac1p+\frac1q=\frac1r+1.
$$
\begin{itemize}
\item[1.] For $\forall(s,t)\in\overline{\Delta}$, the operator $\operatorname{W}_{s,t}$ with domain $\mathcal{S}(\mathbb{R}^n)$ extends to a bounded operator $\operatorname{W}_{s,t}:L^q(\mathbb{R}^n)\to L^r(\mathbb{R}^n)$ with operator norm
\begin{equation}\label{estimationws,t}
	\|\operatorname{W}_{s,t}\|_{q\to r}\le\frac1{p^{\frac{n}{2p}}\left((4\pi)^n\det[A(t)-A(s)]\right)^{\frac{p-1}{2p}}}.
\end{equation}
Here, $\operatorname{W}_{t,t}=\mathrm{1}$ is understood for $\forall t\in[0,+\infty)$;

\item[2.] For $\forall v\in L^q(\mathbb{R}^n)$, the map $(s,t)\mapsto\operatorname{W}_{s,t}v$ belongs to
$$
C^\infty(\Delta,L^r(\mathbb{R}^n))\,\cap\, C(\overline{\Delta},L^r(\mathbb{R}^n));
$$

\item[3.] For $\forall(s,t)\in\overline{\Delta}$,
$$
\operatorname{W}_{s,t}L^q(\mathbb{R}^n)\subset W^{q,\infty}(\mathbb{R}^n);
$$

\item[4.] For $\forall v\in L^q(\mathbb{R}^n)$,
$$
\partial_t\operatorname{W}_{s,t}v-\mathcal{L}_t\operatorname{W}_{s,t}v=0.
$$
\end{itemize}
\end{proposition}
\begin{proof} 
To prove part $1$, we have
\begin{eqnarray*}
	\operatorname{W}_{s,t}&: \mathcal{S}(\mathbb{R}^n)\subset L^q(\mathbb{R}^n) \longrightarrow \mathcal{S}(\mathbb{R}^n)\subset L^{r}(\mathbb{R}^n),\\
	&v \longmapsto \operatorname{W}_{s,t}v, \quad \forall ~1 \leq q, r \leq \infty.
\end{eqnarray*}
Let $v\in\mathcal{S}(\mathbb{R}^n)\subset L^{q}(\mathbb{R}^n)$ for every $1 \leq q \leq \infty$. In addition, we have
$\mathcal{W}(\cdot;s,t)\in \mathcal{S}(\mathbb{R}^n) \subset L^{p}(\mathbb{R}^n)$, for every $1\leq p \leq \infty$. Therefore, by applying Young's inequality for convolutions, we obtain
\begin{align}\label{Young1}
	\|\operatorname{W}_{s,t}v\|_{r} = \|\mathcal{W}(\cdot;s,t) \ast v \|_{r} \leq \|\mathcal{W}(\cdot;s,t)\|_{p} \|v\|_{q}<\infty,~ \text{where}, 
	~ 1\leq p,q,r \leq \infty, \text{and} ~ \frac{1}{p}+\frac{1}{q}=1+\frac{1}{r}.
\end{align}
Therefore using \eqref{Young1}, and Remark \ref{normpkernel} we have
\begin{align}
	\left\|\operatorname{W}_{s,t}\right\|_{q\rightarrow r}&=
	\sup_{\| v \|_{q}=1} \left\|\operatorname{W}_{s,t}v\right\|_{r}\leq\sup_{\| v \|_{q}=1}\left\|\mathcal{W}(\cdot;s,t)\right\|_{p} \|v\|_{q}\notag\\
	&=\left\|\mathcal{W}(\cdot;s,t)\right\|_{p} = \left((4\pi)^n \det[A(t)-A(s)]\right)^{\frac{1-p}{2p}} p^{-\frac{n}{2p}},~~1\leq p<\infty.\notag
\end{align}
In addition, when $p=\infty$, then
\begin{align*}
	\left\|\operatorname{W}_{s,t}\right\|_{q\rightarrow r} \leq \dfrac{1}{\sqrt{(4\pi)^n \det[A(t)-A(s)]}}, \quad 1\leq q,r \leq \infty.
\end{align*}
Thus, the proof of part $1$ is complete. To begin with proving part 2, we show that for every $v \in L^q(\mathbb{R}^n)$, the mapping
$$
(s,t)\mapsto \operatorname{W}_{s,t}v,
$$
belongs to $C^{\infty}(\Delta, L^r(\mathbb{R}^n))$. 
As  $L^q(\mathbb{R}^n) \subset \mathcal{S}'(\mathbb{R})$, then part $2$ of Proposition \ref{pr2} still remains valid. For every $1 \leq r \leq \infty $, we have $ C^{\infty}(\Delta, \mathcal{S}(\mathbb{R}^n)) \subset C^{\infty}(\Delta, L^r(\mathbb{R}^n))$, based on the following points:  
Firstly, if $v \in \mathcal{S}'(\mathbb{R}^n)$, then $\operatorname{W}_{\cdot,\colon}v \in C^{\infty}(\Delta, \mathcal{S}(\mathbb{R}^n))$.  
Secondly, because $\mathcal{S}(\mathbb{R}^n)$ is continuously embedded in $L^r(\mathbb{R}^n)$, derivatives with respect to the topology of $\mathcal{S}(\mathbb{R}^n)$ are automatically endowed with the topology of  $L^r(\mathbb{R}^n)$.
 Therefore
$$
\operatorname{W}_{\cdot,\colon}v\in C^{\infty}(\Delta, L^r(\mathbb{R}^n)).
$$
To complete the proof of part $2$, we need to show that for every $v \in L^q(\mathbb{R}^n)$, the mapping
$$
(s,t)\mapsto \operatorname{W}_{s,t}v,
$$
belongs to $C(\bar{\Delta}, L^r(\mathbb{R}^n))$.
By the previous calculation we just need to prove that this mapping is continuous at $s=t$.
More precisely, we need to show that if
$$
\forall(\{s_{m}\}^{\infty}_{m=1},\{r_{m}\}^{\infty}_{r=1})\xrightarrow{m\rightarrow \infty} (t,t), ~~\forall (s_{m},r_{m})\in \bar{\Delta},
$$
then,
$$\operatorname{W}_{s_{m},r_{m}}v\xrightarrow{m\rightarrow \infty} v,\quad\forall v \in L^q(\mathbb{R}^n).$$
By parts 5 and 6 of Proposition~\ref{pr2}, this  result follows directly. 

Continuing the proof we are going to prove part 3.
We know that for every $v \in \mathcal{S}'(\mathbb{R}^n)$, $\operatorname{W}_{s,t}v \in \mathcal{S}(\mathbb{R}^n)$ for every $(s,t)\in \Delta$.
Since $L^q(\mathbb{R}^n)\subset \mathcal{S}'(\mathbb{R}^n)$, then for every $v \in L^q(\mathbb{R}^n)$, $\operatorname{W}_{s,t}v \in \mathcal{S}(\mathbb{R}^n)$ for every $(s,t)\in \Delta$. In addition, $ \mathcal{S}(\mathbb{R}^n) \subset W^{q,\infty}(\mathbb{R}^n)$ for every $1\leq q \leq \infty$. Hence, the proof is achieved.

The proof of part 4 follows from part 4 of Proposition \ref{pr2}, since $L^{q}(\mathbb{R}^n)\subset \mathcal{S}'(\mathbb{R}^n)$ for $1\leq q \leq \infty$. 
Hence, the proof of Proposition \ref{pr4} is obtained.
\end{proof}

It is worth noting that, as mentioned in the proof regarding the contraction, we can state the following fact more concisely. If we take $p=1$ and $q=r$ in the above proposition, then $\{\operatorname{W}_{s,t}\}_{(s,t)\in\overline{\Delta}}$ is a monoid of non-expansive operators on $L^q(\mathbb{R}^n)$, $q\in[1,+\infty]$.

Now, we are ready to state the well-posedness of the $L^q$-Cauchy problem of our heat equation.

\begin{proposition}\label{pr5}
Let $u_0\in L^q(\mathbb{R}^n)$ and $f\in C^{\infty}(\mathbb{R}_+,L^{q}(\mathbb{R}^n))\cap C([0,\infty),L^{q}(\mathbb{R}^{n}))$, for some $q\in[1,+\infty]$. Then the function $u:\mathbb{R}_+\times\mathbb{R}^n\to\mathbb{C}$ defined by
\begin{equation}\label{solution1}
	u(t)\doteq\operatorname{W}_{0,t}u_0+\int\limits_0^t\operatorname{W}_{s,t}f(s)ds,\quad\forall t\in\mathbb{R}_+,
\end{equation}
 is the  solution of the Cauchy problem
\begin{align}\label{cauchypr5}
\begin{cases}
	\partial_t u(t)-\mathcal{L}_tu(t)=f(t),\quad\forall t\in\mathbb{R}_+,\\
	\lim\limits_{t\to0+}u(t)=u_0,
\end{cases}
\end{align}
that satisfies
$$
u\in C^\infty(\mathbb{R}_+,L^{q}(\mathbb{R}^n))\,\cap\,C([0,+\infty),L^q(\mathbb{R}^n)),
$$
and 
\begin{equation}\label{normestimatesimple}
	\|u(t)\|_q\le\|u_0\|_q+\int\limits_0^t\|f(s)\|_qds,\quad\forall t\in[0,+\infty).
\end{equation} 
Moreover, \eqref{solution1} is the unique solution of the  Cauchy problem \eqref{cauchypr5} for $q\in(1,\infty)$.
Furthermore, if $u_0\ge0$ and $f\ge0$ then $u\ge0$.
\end{proposition}
\begin{proof}
Referring to Proposition \ref{pr4}, part $2$, we have
$$
\operatorname{W}_{\cdot, \colon}v \in C^\infty(\Delta, L^q(\mathbb{R}^n)),\quad \forall v \in L^{q}(\mathbb{R}^n).
$$
Therefore,  $\operatorname{W}_{0,\colon}u_{0} \in C^{\infty}(\mathbb{R}_+, L^{q}(\mathbb{R}^n)),~ \text{for} ~ u_{0}\in L^q(\mathbb{R}^n)\subset \mathcal{S}'(\mathbb{R}^n)$.
In addition,
$f(s)\in L^q(\mathbb{R}^n)\subset\mathcal{S}'(\mathbb{R}^n)$, for every $s \in \mathbb{R}_{+}$. Therefore, for $\forall s\in \mathbb{R}_{+}$, 
$\operatorname{W}_{s,\colon}f(s)\in C^\infty(\mathbb{R}_{+}, L^q(\mathbb{R}^n))$.
Hence, using induction similar to Remark \ref{Leibniz}, we have
\begin{align}
	\frac{\partial^\gamma}{\partial t^\gamma}\left(\int\limits_0^{t}\operatorname{W}_{s,t}f(s,x)ds\right)&=\gamma\partial_{t}^{\gamma-1}W_{t,t}f(t,x)+\int\limits_0^{t}\frac{\partial^\gamma}{\partial_{t}^\gamma}\operatorname{W}_{s,t}f(s,x)ds\notag\\
	&=\gamma\partial_{t}^{\gamma-1}f(t,x)+\int\limits_0^{t}\frac{\partial^\gamma}{\partial_{t}^\gamma}\operatorname{W}_{s,t}f(s,x)ds,\quad \forall x\in \mathbb{R}^{n},~~ \forall t\in \mathbb{R}_{+},~~
	\forall \gamma\in \mathbb{N}_{0}.
\end{align}
Thus, 
$$
\int\limits_0^{\cdot} \operatorname{W}_{s,\cdot}f(s) ds \in C^{\infty}(\mathbb{R}_{+}, L^q(\mathbb{R}^n)).
$$
Finally, it follows that
\begin{equation}\label{cinftylq}
	u \in C^{\infty}(\mathbb{R}_{+},L^{q}(\mathbb{R}^{n})),\quad 1\leq q \leq\infty.
\end{equation}
Now, we want to show that $u \in C([0,\infty), L^q(\mathbb{R}^n))$. To do so, considering \eqref{cinftylq}, we only need to show continuity at $t=0$. We have
$u(0)=\operatorname{W}_{0,0}u_{0}=u_0$ by the definition.
It suffices to prove the following:
\begin{align}\label{contntyatzero}
	\lim_{t \rightarrow 0^{+}}\left\| \operatorname{W}_{0,t}u_{0}-u_{0}\right\|_{q}=0,	\quad \forall t\in\mathbb{R}_{+}.
\end{align}
To prove \eqref{contntyatzero}, since $\mathcal{S}(\mathbb{R}^{n})$ is dense in $L^{q}(\mathbb{R}^{n}),$ then for $u_{0}\in L^{q}(\mathbb{R}^{n})$ and every $\epsilon\in(0,1)$, there exists $\phi\in \mathcal{S}(\mathbb{R}^{n})$ such that \begin{equation}
	\left\| \phi-u_{0}\right\| _{q}<\epsilon.
\end{equation}
On the other hand, using Young's inequality, and part $5$ of Proposition \ref{pr2}, and the fact that convergence in Schwartz implies convergence in $L^{q}(\mathbb{R}^{n})$,
it follows that
\begin{align}
\lim_{t\rightarrow 0+}\left\| \operatorname{W}_{0,t}u_{0}-u_{0}\right\|_{q}&\leq\lim_{t\rightarrow 0+}\left(  \left\| \phi-u_{0}\right\| _{q}+\left\| \mathcal{W}(\cdot;0,t)*(\phi-u_{0})\right\|_{q}+\left\|\phi-\mathcal{W}(\cdot;0,t)*\phi\right\|_{q}\right) \notag\\
&\leq \epsilon+\lim_{t\rightarrow 0+} \left( \left\| \mathcal{W}(\cdot;0,t)\right\| _{1}\left\| \phi-u_{0}\right\| _{q}\right) \notag\\
&\leq \epsilon+\lim_{t\rightarrow 0+} \left\| \phi-u_{0}\right\| _{q}<2\epsilon\notag.
\end{align}
Since $\epsilon$ is arbitrary, then $\lim_{t\rightarrow 0+}\left\| \operatorname{W}_{0,t}u_{0}-u_{0}\right\|_{q}=0.$

Now, we are going to prove the estimate \eqref{normestimatesimple}.
To this end, we take the $L^q$-norm of both sides of \eqref{solution1} and also consider $q=r$ and $p=1$ in \eqref{estimationws,t}. Then we have
\begin{align}
	\|u(t)\|_q &\leq \|\operatorname{W}_{0,t}u_{0}\|_{q}+\int\limits_0^{t}\|\operatorname{W}_{s,t}f(s)\|_{q}ds\notag\\
	&\leq \|u_{0}\|_{q}+\int\limits_0^{t}\|f(s)\|_{q}ds, \quad\forall t\in \mathbb{R}_{+}.
\end{align}
At this stage, we need to show that the solution $u$ satisfies the Cauchy problem \eqref{cauchypr5}. Since $\mathcal{S}(\mathbb{R}^n)$ is continuously embedded in $L^q(\mathbb{R}^n)$, any derivative taken with respect to the topology of $\mathcal{S}(\mathbb{R}^n)$ naturally inherits the topology of $L^q(\mathbb{R}^n)$. Consequently, the verification follows similarly to the corresponding argument in Proposition \ref{pr3}.
Additionally by \eqref{contntyatzero}, we have
\begin{equation*}
	\lim_{t\rightarrow 0^{+}}u(t)=u_{0}, \quad\forall t\in \mathbb{R}_{+}.
\end{equation*}
Here, we aim to demonstrate uniqueness. Suppose $\tilde{u}$ is another solution to \eqref{cauchypr5}. Then,
\begin{equation}\label{cauchypr52}
\begin{cases}
	\partial_t \tilde{u}(t)-\mathcal{L}_t\tilde{u}(t)=f(t),\quad\forall t\in\mathbb{R}_+,\\
	\lim\limits_{t\to0+}\tilde{u}(t)=u_0.
\end{cases}
\end{equation}
Subtracting \eqref{cauchypr5} and \eqref{cauchypr52}, we have
\begin{equation}\label{cauchypr53}
\begin{cases}
	\partial_t \left(u(t)-\tilde{u}(t)\right)-\mathcal{L}_t(u(t)-\tilde{u}(t))=0,\quad\forall t\in\mathbb{R}_+,\\
	\lim\limits_{t\to0+}\left(u(t)-\tilde{u}(t)\right)=0.
\end{cases}
\end{equation}
Using \eqref{nonincreasingenergy} for \eqref{cauchypr53}, we establish uniqueness as follows:
\begin{equation}
	\|u(t)-\tilde{u}(t)\|_p\le\|u_0\|_p=0,\quad\forall t\in[0,+\infty),~~\forall 1<p<\infty.
\end{equation}
To prove the final part concerning the positivity of the solution $u$, we employ arguments analogous to those used in Proposition \ref{pr3}.
\end{proof}

We will see bellow that under certain assumptions we have also growth estimates.

\begin{proposition}\label{pr6}
Let $1\le p,q,r\le+\infty$ be such that
$$
\frac1p+\frac1q=\frac1r+1.
$$
In the terminology of Proposition \ref{pr5}, assume that $n(p-1)\alpha<2p$ and
$$
f\in L^\beta((0,t),L^q(\mathbb{R}^n)),\quad\beta>\frac1{1-\frac{(p-1)n}{2p}},\quad t\in\mathbb{R}_+,\quad 1<\alpha,\beta<\infty,\quad
\frac1\alpha+\frac1\beta=1.
$$
Moreover, if $\lambda_{\min}(a(t))$ is the minimum of the eigenvalues of the matrix $a(t)$, we define the function
$$
F(t)\doteq \int\limits_0^t \lambda_{\min}(a(\tau))d\tau,\quad \forall t\in \mathbb{R}_{+}.
$$
Then, there exist 
$C=\frac{1}{p^{\frac{n}{2p}}(4\pi)^{\frac{n(p-1)}{2p}}}>0$,
such that
\begin{equation}\label{constant}
	\|u(t)\|_r\le C(F(t))^{-\frac{n(p-1)}{2p}}\|u_0\|_q+C\|\|f(\cdot)\|_q\|_{L^\beta((0,t))}\left(\int\limits_0^t\frac{1}{\left(F(t)-F(s)\right)^{\frac{(p-1)n\alpha}{2p}}}ds\right)^{\frac{1}{\alpha}}.
\end{equation}
In particular, if $\lambda_{\min}(a(t))$ is 
bounded below by a positive constant $\gamma$, then there exists a constant $C>0$, as in the previous estimate \eqref{constant}, such that
$$
\|u(t)\|_r\le  C\gamma ^{-\frac{n(p-1)}{2p}}\left(t^{-\frac{n(p-1)\alpha}{2p}}\|u_{0}\|_q+ (t-s)^{\frac{1}{\alpha}-\frac{n(p-1)}{2p}}\|\|f(\cdot)\|_{q}\|_{L^\beta((0,t))}\right),\quad \forall t\in\mathbb{R}_{+}.
$$
\end{proposition}
\begin{proof}
We start the proof by the following calculation
\begin{align}\label{detA(t)lambda}
	\det[A(t)]&\geq (\lambda_{\min}[A(t)])^n \geq(\lambda_{\min}[\int\limits_0^t (a(\tau))d\tau])^n
	\notag\\
	&\geq\left(\int\limits_0^t \lambda_{min}(a(\tau))d\tau\right)^n=(F(t))^n,\quad\forall t\in \mathbb{R}_{+}.
\end{align}
In addition,
\begin{align}\label{detA(t)-A(s)lambda}
	\det[A(t)-A(s)]&\geq (\lambda_{\min}[A(t)-A(s)])^n \geq(\lambda_{\min}[\int\limits_s^t (a(\tau))d\tau])^n\notag\\
	&\geq\left(\int\limits_s^t \lambda_{min}(a(\tau))d\tau\right)^n=(F(t)-F(s))^n,\quad\forall (s,t)\in \Delta.
\end{align}
We now take the $L^r-norm$ from both sides of the solution $u$. Applying the equations \eqref{detA(t)lambda}, \eqref{detA(t)-A(s)lambda}, and \ref{normestimatesimple},  together with H\"{o}lder's inequality for $1<\alpha, \beta<\infty$ satisfying $\frac{1}{\alpha}+\frac{1}{\beta}=1$, and proceeding under the assumption that
$n(p-1)\alpha<2p$ and
$
f\in L^\beta((0,t),L^q(\mathbb{R}^n))$ where $\beta>\frac1{1-\frac{(p-1)n}{2p}}$, we get
\begin{align}
	\|u(t)\|_{r}&\leq \|\operatorname{W}_{0,t}u_{0}\|_r+\int\limits_0^{t}\|\operatorname{W}_{s,t}f(s)\|_{r}ds\notag\\
	&\leq \frac{1}{p^{\frac{n}{2p}}\left((4\pi)^n \det[A(t)]\right)^{\frac{p-1}{2p}}}\|u_{0}\|_{q}
	+\int\limits_0^{t}\frac{1}{p^{\frac{n}{2p}}\left((4\pi)^n\det[A(t)-A(s)]\right)^{\frac{p-1}{2\pi}}}\|f(s)\|_{q}ds\notag\\
	&\leq \frac{1}{p^{\frac{n}{2p}}\left((4\pi)^n(F(t))^n\right)^{\frac{p-1}{2p}}}\|u_{0}\|_{q}
	+\int\limits_0^{t}\frac{1}{p^{\frac{n}{2p}}\left(4\pi(F(t)-F(s))^n\right)^{\frac{p-1}{2\pi}}}\|f(s)\|_{q}ds\notag\\
	&\leq C F(t)^{-\frac{n(p-1)}{2p}}\|u_{0}\|_{q}+C
	\left(\int\limits_0^t \frac{1}{(F(t)-F(s))^{\frac{n(p-1)\alpha}{2p}}}ds\right)^{\frac{1}{\alpha}}\left(\int\limits_0^t \|f(s)\|^\beta_{q}ds\right)^{\frac{1}{\beta}}\notag\\
	&=CF(t)^{-\frac{n(p-1)}{2p}}\|u_{0}\|_q+C\left(\int\limits_0^t \frac{1}{(F(t)-F(s))^{\frac{n(p-1)\alpha}{2p}}}ds\right)^{\frac{1}{\alpha}}\|\|f(\cdot)\|_{q}\|_{L^\beta((0,t))}.
\end{align}	
Now, we apply the special condition $\lambda_{\min}(a(t)) >\gamma>0$, then we can conclude that 
\begin{align}\label{detA(t)gamma}
	\det[A(t)]&\geq\left(\int\limits_0^t \lambda_{min}(a(\tau))d\tau\right)^n\geq (t\gamma)^n,\quad\forall t\in \mathbb{R}_{+}.
\end{align}
In addition,
\begin{align}\label{detA(t)-A(s)gamma}
	\det[A(t)-A(s)]\geq\left(\int_{s}^t \lambda_{min}(a(\tau))d\tau\right)^n\geq ((t-s)\gamma)^n,\quad\forall s,t\in \Delta.
\end{align}
Therefore, we conclude that
\begin{align*}
	\|u(t)\|_{L^r}&\leq \|\operatorname{W}_{0,t}u_{0}\|_r+\int\limits_0^{t}\|\operatorname{W}_{s,t}f(s)\|_{r}ds\notag\\
	&\leq \frac{1}{p^{\frac{n}{2p}}\left((4\pi)^n \det[A(t)]\right)^{\frac{p-1}{2p}}}\|u_{0}\|_{q}
	+\int\limits_0^{t}\frac{1}{p^{\frac{n}{2p}}\left((4\pi)^n\det[A(t)-A(s)]\right)^{\frac{p-1}{2\pi}}}\|f(s)\|_{q}ds\notag\\
	&\leq \frac{1}{p^{\frac{n}{2p}}(4\pi)^{\frac{n(p-1)}{2p}}\gamma^{\frac{n(p-1)}{2p}} t^{\frac{n(p-1)}{2p}}}\|u_{0}\|_{q}
	+\int\limits_0^{t}\frac{1}{p^{\frac{n}{2p}}(4\pi)^{\frac{n(p-1)}{2p}}\gamma^{\frac{n(p-1)}{2p}} (t-s)^{\frac{n(p-1)}{2p}}}\|f(s)\|_{q}ds\notag\\
	&\leq C\gamma ^{-\frac{n(p-1)}{2p}} t^{-\frac{n(p-1)}{2p}}\|u_{0}\|_{q}+C
	\left(\int\limits_0^t \frac{1}{\gamma^{\frac{n(p-1)}{2p}}(t-s)^{\frac{n(p-1)\alpha}{2p}}}ds\right)^{\frac{1}{\alpha}}\left(\int\limits_0^t \|f(s)\|^\beta_{q}ds\right)^{\frac{1}{\beta}}\notag\\
	&=C\gamma ^{-\frac{n(p-1)}{2p}}\left(t^{-\frac{n(p-1)\alpha}{2p}}\|u_{0}\|_q+ (t-s)^{\frac{1}{\alpha}-\frac{n(p-1)}{2p}}\|\|f(\cdot)\|_{q}\|_{L^\beta((0,t))}\right),\quad \forall t\in\mathbb{R}_{+}.
\end{align*}	
\end{proof}

\section{The very weak formulation}
\subsection{Preliminaries}

We begin with, by now standard, definitions in the theory of very weak solutions of PDEs. Let $\mathcal{X}$ be a (real or complex) locally convex topological space, with the system of seminorms $\{\|\cdot\|_\alpha\}_{\alpha\in\aleph}$. In this theory, elements of $h\in\mathcal{X}$ are thought of as ``regular'' items, whereas the ``singular'' items $h$ lying beyond $\mathcal{X}$ are represented by approximation schemes $\{h_\epsilon\}_{\epsilon\in(0,1)}\subset\mathcal{X}$ in one sense or another, with the understanding that $h_\epsilon\xrightarrow[\epsilon\to0]{}h$ in a manner not captured by the topology of $\mathcal{X}$ itself. The theory does not necessarily concern itself with how exactly $h_\epsilon$ approximates a singular object $h$, but rather how badly that approximation fails to correspond to the topology of $\mathcal{X}$. Namely, all seminorms $\|h_\epsilon\|_\alpha$ are allowed to have at most a certain maximal growth in $\frac1\epsilon$. It is clear that this definition is almost vacuous in absolute terms, since a reparameterisation of $\epsilon$ can produce rather arbitrary behaviour as $\epsilon\to0$. And this is in fact a feature, a flexibility of the theory, where the choice of a particular parameterisation is left to the context of application. But we will step even beyond this arbitrariness, by working in the more general framework of asymptotic scales.

Let $(\Upsilon,\prec)$ be a (non-empty) directed set, and $\{\varsigma_\ell\}_{\ell\in\Upsilon}\subset C((0,1),\mathbb{R}_+)$ a family of functions increasing in a neighbourhood of $0$, such that
\begin{equation}
	(\forall\ell,\jmath\in\Upsilon)\quad\ell\prec\jmath\quad\Rightarrow\quad\varsigma_\ell(\epsilon)=\mathfrak{o}[\varsigma_\jmath(\epsilon)]_{\epsilon\to0},\label{ScaleCond1}
\end{equation}
\begin{equation}
	(\forall\ell,\jmath\in\Upsilon)\quad(\exists\hbar\in\Upsilon)\quad\varsigma_\hbar(\epsilon)=\mathfrak{o}[\varsigma_\ell(\epsilon)\varsigma_\jmath(\epsilon)]_{\epsilon\to0}.\label{ScaleCond2}
\end{equation}
Then the system $\mho\doteq\{\varsigma_\ell\}_{\ell\in\Upsilon}$ is called an asymptotic scale.

\begin{definition} 
A net $\{h_\epsilon\}_{\epsilon\in(0,1)}\subset\mathcal{X}$ is called $\mho$-moderate if
$$
(\forall\alpha\in\aleph)\,(\exists\ell_\alpha\in\Upsilon)\quad\varsigma_{\ell_\alpha}(\epsilon)\|h_\epsilon\|_\alpha=\mathfrak{o}[1]_{\epsilon\to0}.
$$
The vector space of all $\mho$-moderate nets in $\mathcal{X}$ will be denoted by $\mathcal{M}^\mho(\mathcal{X})$.
\end{definition}

On the other hand, there is a rate of decay as $\epsilon\to0$ which is deemed to always mean convergence to the zero element $0\in\mathcal{X}$. That rate is the reciprocal of the finite scale growth - the infinite scale decay, and nets with such decay are called negligible.

\begin{definition} A net $\{h_\epsilon\}_{\epsilon\in(0,1)}\subset\mathcal{X}$ is called $\mho$-negligible if
$$
(\forall\alpha\in\aleph)\,(\forall\ell\in\Upsilon)\quad\|h_\epsilon\|_\alpha=\mathfrak{o}[\varsigma_\ell(\epsilon)]_{\epsilon\to0}.
$$
The vector space of all $\mho$-negligible nets in $\mathcal{X}$ will be denoted by $\mathcal{N}^\mho(\mathcal{X})$.
\end{definition}

If the difference of two $\mho$-moderate nets $\{h_\epsilon\}_{\epsilon\in(0,1)},\{g_\epsilon\}_{\epsilon\in(0,1)}\in\mathcal{M}^\mho(\mathcal{X})$ happens to be $\mho$-negligible, $\{h_\epsilon-g_\epsilon\}_{\epsilon\in(0,1)}\in\mathcal{N}^\mho(\mathcal{X})$, then these two approximations are thought to represent the same object. Thus, singular items in our theory are identified with the equivalence classes
$$
[h_\epsilon]_{\epsilon\in(0,1)}\in\mathcal{G}^\mho(\mathcal{X})\doteq\mathcal{M}^\mho(\mathcal{X})/\mathcal{N}^\mho(\mathcal{X}).
$$
We want to think of regular items $h\in\mathcal{X}$ as a particular class of singular ones, which means that we need an embedding
$$
\imath_\mathcal{X}:\mathcal{X}\hookrightarrow\mathcal{G}^\mho(\mathcal{X}).
$$
It is here that the flexibility in choosing $\mho$ comes in handy; it should be such that a practically useful embedding $\imath_\mathcal{X}$ fits in the framework. In practice, the embedding $\imath_\mathcal{X}$ is given by a certain mollifier, but we will not restrict ourselves now to a particular choice. In this context, $\mathcal{G}^\mho(\mathcal{X})$ is often called the $\mho$-Colombeau extension of $\mathcal{X}$.

Let $\mathbb{F}\in\{\mathbb{R},\mathbb{C}\}$. When the $\mathbb{F}$-linear locally convex space $\mathcal{X}$ is the trivial Banach space $\mathbb{F}$, the corresponding Colombeau extension is called the ring of generalised scalars,
$$
\widetilde{\mathbb{F}}^\mho=\mathcal{G}^\mho(\mathbb{F}).
$$
It is called a ring because, as can be checked directly, the $\epsilon$-componentwise multiplication defines a commutative product in it. Moreover, for every $\mathbb{F}$-linear locally convex space $\mathcal{X}$, the Colombeau extension $\mathcal{G}^\mho(\mathcal{X})$ is a $\widetilde{\mathbb{F}}^\mho$-bimodule. In this case, the embedding $\imath_\mathbb{F}$ is given by the constant embedding,
$$
\imath_\mathbb{F}(\lambda)=[\lambda]_{\epsilon\in(0,1)},\quad\forall\lambda\in\mathbb{F}.
$$
We will think of $\mathbb{F}$ as identified with its image $\imath_\mathbb{F}(\mathbb{F})$ in $\widetilde{\mathbb{F}}^\mho$. The restriction of $\widetilde{\mathbb{F}}^\mho$-action to $\mathbb{F}$ agrees with its original $\epsilon$-componentwise action.

If $\Gamma\subset\mathcal{X}$ is a subset, then $\mathcal{M}^\mho_\mathcal{X}(\Gamma)\doteq\mathcal{M}^\mho(\mathcal{X})\cap\Gamma^{(0,1)}$ can be defined, and then
$$
\mathcal{G}^\mho_\mathcal{X}(\Gamma)\doteq\mathcal{M}^\mho_\mathcal{X}(\Gamma)/\mathcal{N}^\mho(\mathcal{X})\subset\mathcal{G}^\mho(\mathcal{X}).
$$
For $\Gamma=\mathbb{R}_+\subset\mathbb{F}$, we will denote $\widetilde{\mathbb{R}}_+^\mho\doteq\mathcal{G}_{\mathbb{F}}^\mho(\mathbb{R}_+)$. If $\Gamma$ is an $\mathbb{R}_+$-cone in $\mathcal{X}$, then $\mathcal{G}^\mho(\Gamma)$ is naturally an $\mathbb{R}_+$-cone in $\mathcal{G}^\mho(\mathcal{X})$, and indeed a $\widetilde{\mathbb{R}}_+^\mho$-cone.

If $\mathcal{X}$ is a space of distributions (including functions, measures etc.), and $\Gamma\doteq\mathcal{X}_+$ is the $\mathbb{R}_+$-cone of non-negative elements, then $\mathcal{G}^\mho(\mathcal{X}_+)\doteq\mathcal{G}^\mho_\mathcal{X}(\mathcal{X}_+)$ will be the corresponding non-negative extension. For $u\in\mathcal{G}^\mho(\mathcal{X})$, we will write $u\ge0$ to mean $u\in\mathcal{G}^\mho(\mathcal{X}_+)$.

On the vector space $\mathcal{G}^\mho(\mathcal{X})$, we will impose the so-called $\mho$-sharp topology, given by the basis of neighbourhoods $\|U_{\alpha,\ell}\|_{\alpha\in\aleph,\ell\in\Upsilon}$ of $0$ defined as follows:
\begin{equation}
	U_{\alpha,\ell}\doteq\left\{[h_\epsilon]_{\epsilon\in(0,1)}\in\mathcal{M}^\mho(\mathcal{X})\;\vline\quad\|h_\epsilon\|_\alpha=\mathfrak{o}[\varsigma_\ell(\epsilon)]_{\epsilon\to0}\right\},\quad\forall\alpha\in\aleph,\quad\forall\ell\in\Upsilon.\label{SharpNeighbourhoods}
\end{equation}
With this topology, $\widetilde{\mathbb{F}}^\mho$ is a topological ring, and $\mathcal{G}^\mho(\mathcal{X})$ is a topological $\widetilde{\mathbb{F}}^\mho$-bimodule. A very detailed development of the topology of Colombeau extensions, rather parallel to the standard theory of locally convex spaces, can be found in \cite{Gar05} for the case of the usual asymptotic scale $\mho=\{\epsilon^a\}_{a\in\mathbb{R}}$. The situation with a more general scale $\mho$ is mostly analogous, but we will not need these sophistications in this work.

Let $\mathcal{X}$ and $\mathcal{Y}$ be two $\mathbb{F}$-linear locally convex spaces, with the systems of seminorms $\{\|\cdot\|_\alpha\}_{\alpha\in\aleph}$ and $\{\|\cdot\|_\beta\}_{\beta\in\beth}$, respectively. We will present a scheme for producing linear continuous maps $\mathcal{G}^\mho(\mathcal{X})\to\mathcal{G}^\mho(\mathcal{Y})$.

\begin{proposition}\label{VWContProp} Let $\{\operatorname{T}_\epsilon\}_{\epsilon\in(0,1)}\subset\mathcal{B}(\mathcal{X},\mathcal{Y})$ be a net of linear continuous operators. Suppose that
$$
(\forall\beta\in\beth)\quad(\exists\ell_\beta\in\Upsilon)\quad(\exists A_\beta\subset\aleph)\quad\sharp A_\beta<\infty\quad\wedge\quad\varsigma_{\ell_\beta}(\epsilon)\|\operatorname{T}_\epsilon\|_{\beta,A_\beta}=\mathfrak{o}[1]_{\epsilon\to0},
$$
where
$$
\|\operatorname{T}_\epsilon\|_{\beta,A_\beta}\doteq\sup\left\{\|\operatorname{T}_\epsilon h\|_\beta\;\vline\quad h\in\mathcal{X},\quad\max\limits_{\alpha\in A_\beta}\|h\|_\alpha\le1\right\}.
$$
Then the map $\operatorname{T}:\mathcal{G}^\mho(\mathcal{X})\to\mathcal{G}^\mho(\mathcal{Y})$ defined by
$$
\operatorname{T}[h_\epsilon]_{\epsilon\in(0,1)}=[\operatorname{T}_\epsilon h_\epsilon]_{\epsilon\in(0,1)},\quad\forall[h_\epsilon]_{\epsilon\in(0,1)}\in\mathcal{G}^\mho(\mathcal{X}),
$$
defines a $\widetilde{\mathbb{F}}^\mho$-linear continuous operator.
\end{proposition}
\begin{proof} 
First, we need to show that the operator is well-defined, that is, if $[h_\epsilon]_{\epsilon\in(0,1)}=0$ (or $\{h_\epsilon\}_{\epsilon\in(0,1)}\in\mathcal{N}^\mho(\mathcal{X})$) then $[\operatorname{T}_\epsilon h_\epsilon]_{\epsilon\in(0,1)}=0$ (or $\{\operatorname{T}_\epsilon h_\epsilon\}_{\epsilon\in(0,1)}\in\mathcal{N}^\mho(\mathcal{X})$). Indeed, take any $\ell\in\Upsilon$ and $\beta\in\beth$, and choose $\jmath(\ell,\beta)\in\Upsilon$ using (\ref{ScaleCond2}) so that
$$
\varsigma_{\jmath(\ell,\beta)}(\epsilon)=\mathfrak{o}[\varsigma_\ell(\epsilon)\varsigma_{\ell_\beta}(\epsilon)]_{\epsilon\to0}.
$$
If we take
$$
[h_\epsilon]_{\epsilon\in(0,1)}\in\bigcap_{\alpha\in A_\beta}U_{\alpha,\jmath(\ell,\beta)},
$$
in terms of the definition (\ref{SharpNeighbourhoods}) of neighborhoods, then
$$
\max\limits_{\alpha\in A_\beta}\|h_\epsilon\|_\alpha=\mathfrak{o}[\varsigma_{\jmath(\ell,\beta)}(\epsilon)]_{\epsilon\to0}=\mathfrak{o}[\varsigma_\ell(\epsilon)\varsigma_{\ell_\beta}(\epsilon)]_{\epsilon\to0},
$$
whence
$$
\|\operatorname{T}_\epsilon h_\epsilon\|_\beta\le\|\operatorname{T}_\epsilon\|_{\beta,A_\beta}\cdot\max\limits_{\alpha\in A_\beta}\|h_\epsilon\|_\alpha=\|\operatorname{T}_\epsilon\|_{\beta,A_\beta}\cdot\mathfrak{o}[\varsigma_\ell(\epsilon)\varsigma_{\ell_\beta}(\epsilon)]_{\epsilon\to0}=\mathfrak{o}[\varsigma_\ell(\epsilon)]_{\epsilon\to0},
$$
meaning that
$$
[\operatorname{T}_\epsilon h_\epsilon]_{\epsilon\in(0,1)}\in U_{\beta,\ell}.
$$
Thus, our map acts as
$$
\bigcap_{\alpha\in A_\beta}U_{\alpha,\jmath(\ell,\beta)}\ni[h_\epsilon]_{\epsilon\in(0,1)}\mapsto[\operatorname{T}_\epsilon h_\epsilon]_{\epsilon\in(0,1)}\in U_{\beta,\ell},\quad\forall\beta\in\beth,\quad\forall\ell\in\Upsilon.
$$
Taking the intersection over all $\ell\in\Upsilon$ shows that $\operatorname{T}0=0$, while otherwise it shows that the preimage of a neighbourhood is a neighbourhood in the sharp topology, proving the continuity of $\operatorname{T}$. The $\widetilde{\mathbb{F}}^\mho$-linearity is checked by direct substitution.
\end{proof}

For further reference, let us denote the $\widetilde{\mathbb{F}}^\mho$-module of continuous $\widetilde{\mathbb{F}}^\mho$-linear operators $\operatorname{T}:\mathcal{G}^\mho(\mathcal{X})\to\mathcal{G}^\mho(\mathcal{Y})$ by
$$
\mathcal{B}(\mathcal{G}^\mho(\mathcal{X}),\mathcal{G}^\mho(\mathcal{Y})).
$$
When $\mathcal{X}=\mathcal{Y}$, we will write
$$
\mathcal{B}(\mathcal{G}^\mho(\mathcal{X}))\doteq\mathcal{B}(\mathcal{G}^\mho(\mathcal{X}),\mathcal{G}^\mho(\mathcal{X})).
$$

\begin{corollary}\label{VWContCorr} Every $\operatorname{T}\in\mathcal{B}(\mathcal{X},\mathcal{Y})$ extends to $\operatorname{T}\in\mathcal{B}(\mathcal{G}^\mho(\mathcal{X}),\mathcal{G}^\mho(\mathcal{Y}))$ by setting
	$$
	[(\operatorname{T}h)_\epsilon]_{\epsilon\in(0,1)}\doteq[\operatorname{T}h_\epsilon]_{\epsilon\in(0,1)},\quad\forall[h_\epsilon]_{\epsilon\in(0,1)}\in\mathcal{G}^\mho(\mathcal{X}).
	$$
\end{corollary}
\begin{proof} We will apply Proposition \ref{VWContProp} with $\operatorname{T}_\epsilon=\operatorname{T}$ for $\forall\epsilon\in(0,1)$. First note that, as is well-known for bounded operators in locally convex spaces,
$$
(\forall\beta\in\beth)\quad(\exists A_\beta\subset\aleph)\quad\sharp A_\beta<\infty\quad\wedge\quad\|\operatorname{T}\|_{\beta,A_\beta}<\infty.
$$
Since $\Upsilon\neq\emptyset$, take any $\ell_0\in\Upsilon$. By definition, $\varsigma_{\ell_0}\in C((0,1),\mathbb{R}_+)$ is increasing in a neighborhood of $0$, therefore $\limsup\limits_{\epsilon\to0}\varsigma_{\ell_0}(\epsilon)<\infty$. By assumption (\ref{ScaleCond2}), there exists $\hbar\in\Upsilon$ such that
$$
\varsigma_\hbar(\epsilon)=\mathfrak{o}[\varsigma_{\ell_0}(\epsilon)^2]_{\epsilon\to0}=\mathfrak{o}[1]_{\epsilon\to0}.
$$
Setting $\ell_\beta\doteq\hbar$, we see that
$$
\varsigma_{\ell_\beta}(\epsilon)\|\operatorname{T}_\epsilon\|_{\beta,A_\beta}=\varsigma_\hbar(\epsilon)\|\operatorname{T}\|_{\beta,A_\beta}=\mathfrak{o}[1]_{\epsilon\to0},
$$
so that the assumptions of Proposition \ref{VWContProp} are satisfied, yielding the desired assertion.
\end{proof}

\subsection{The very weak formulation of the Cauchy problem}

Let us now turn to the very weak formulation of our heat equation for the case when the diffusivity matrix $a$ is very singular. Thus, we consider the equation
\begin{equation}
	\partial_t u-\mathcal{L}_tu=f,\quad \mathcal{L}_t=\langle\partial_x,a\partial_x\rangle,\label{VWEquation}
\end{equation}
where $a\in\mathcal{G}^\mho(C^\infty(\mathbb{R}_+,\mathrm{GL}(n))\cap C([0,\infty),\mathrm{GL}(n)))$, $f\in\mathcal{G}^\mho(C^\infty(\mathbb{R}_+,\mathcal{S}(\mathbb{R}^n))\cap C([0,\infty),\mathcal{S}(\mathbb{R}^{n})))$ and

 $u\in\mathcal{G}^\mho(C^\infty(\mathbb{R}_+,\mathcal{S}(\mathbb{R}^n)))$. Here, the operators
$$
\partial_t:\mathcal{G}^\mho(C^\infty(\mathbb{R}_+,\mathcal{S}(\mathbb{R}^n)))\longrightarrow\mathcal{G}^\mho(C^\infty(\mathbb{R}_+,\mathcal{S}(\mathbb{R}^n))),
$$
and
$$
\mathcal{L}_t:\mathcal{G}^\mho(C^\infty(\mathbb{R}_+,\mathcal{S}(\mathbb{R}^n)))\longrightarrow\mathcal{G}^\mho(C^\infty(\mathbb{R}_+,\mathcal{S}(\mathbb{R}^n))),
$$
are defined by
\begin{equation}\label{VWDtDef}
	[(\partial_tu)_\epsilon]_{\epsilon\in(0,1)}=[\partial_tu_\epsilon]_{\epsilon\in(0,1)},\quad\forall u=[u_\epsilon]_{\epsilon\in(0,1)}\in\mathcal{G}^\mho(C^\infty(\mathbb{R}_+,\mathcal{S}(\mathbb{R}^n))),
\end{equation}
and
\begin{equation}\label{VWLtDef}
	[(\mathcal{L}_tu)_\epsilon]_{\epsilon\in(0,1)}=\left[\langle\partial_x,a_\epsilon\partial_x\rangle u_\epsilon\right]_{\epsilon\in(0,1)},\quad\forall u=[u_\epsilon]_{\epsilon\in(0,1)}\in\mathcal{G}^\mho(C^\infty(\mathbb{R}_+,\mathcal{S}(\mathbb{R}^n))).
\end{equation}

\begin{proposition}\label{pr8}
The formulae (\ref{VWDtDef}) and (\ref{VWLtDef}) define operators $\partial_t\in\mathcal{B}(\mathcal{G}^\mho(C^\infty(\mathbb{R}_+,\mathcal{S}(\mathbb{R}^n))))$ and $\mathcal{L}_t\in\mathcal{B}(\mathcal{G}^\mho(C^\infty(\mathbb{R}_+,\mathcal{S}(\mathbb{R}^n))))$, respectively.
\end{proposition}
\begin{proof} 
As the following two operators 
\begin{eqnarray}\label{mapi}
	\partial_t:&C^\infty(\mathbb{R}_+, \mathcal{S}(\mathbb{R}^n)) \longrightarrow C^\infty(\mathbb{R}_+, \mathcal{S}(\mathbb{R}^n)), \notag\\
	&u \mapsto \partial_t u,
\end{eqnarray}
and 
\begin{eqnarray}\label{mapii}
	\mathcal{L}_t:&C^\infty(\mathbb{R}_+, \mathcal{S}(\mathbb{R}^n)) \longrightarrow C^\infty(\mathbb{R}_+, \mathcal{S}(\mathbb{R}^n)), \notag\\
	&u \mapsto \mathcal{L}_t u.
\end{eqnarray}
are linear and continuous on the space $ C^\infty(\mathbb{R}_+, \mathcal{S}(\mathbb{R}^n)) $, the result will then follow from Corollary \ref{VWContCorr}.
\end{proof}
Thus, we have $\partial_t-\mathcal{L}_t\in\mathcal{B}(\mathcal{G}^\mho(C^\infty(\mathbb{R}_+,\mathcal{S}(\mathbb{R}^n))))$, and the very weak heat equation (\ref{VWEquation}) makes sense. In order to produce a sensible initial condition at $t=0$, we need to work a little more, since the limit $t\to0$ is taken not in the topology of $\mathcal{S}(\mathbb{R}^n)$ but in the distributional topology of $\mathcal{S}'(\mathbb{R}^n)$. In terms of Proposition \ref{pr3}, if a distributional solution $u\in C^\infty(\mathbb{R}_+,\mathcal{S}(\mathbb{R}^n))$ satisfies the initial condition in the weak sense,
$$
\lim_{t\to0+}u(t)=u_0\in\mathcal{S}'(\mathbb{R}^n),\quad\mathcal{S}(\mathbb{R}^n)\hookrightarrow\mathcal{S}'(\mathbb{R}^n),
$$
then we interpret this as
$$
u\in C^\infty(\mathbb{R}_+,\mathcal{S}(\mathbb{R}^n))\,\cap\,C([0,+\infty),\mathcal{S}'(\mathbb{R}^n))
$$
and
$$
\left.u\right|_0=u_0.
$$
It is this formulation that will be extended to the very weak setting. We will consider the locally convex space $C^\infty(\mathbb{R}_+,\mathcal{S}(\mathbb{R}^n))\,\cap\,C([0,+\infty),\mathcal{S}'(\mathbb{R}^n))$ equipped with the union of the families of seminorms of $C^\infty(\mathbb{R}_+,\mathcal{S}(\mathbb{R}^n))$ and $C([0,+\infty),\mathcal{S}'(\mathbb{R}^n))$, respectively.

The operation of restriction to (or evaluation at) $t=0$ induces a linear continuous surjection
$$
\left.\phantom{A}\right|_0:C^\infty(\mathbb{R}_+,\mathcal{S}(\mathbb{R}^n))\cap C([0,+\infty),\mathcal{S}'(\mathbb{R}^n))\longrightarrow\mathcal{S}'(\mathbb{R}^n),
$$
which can be extended to
$$
\left.\phantom{A}\right|_0:\mathcal{G}^\mho(C^\infty(\mathbb{R}_+,\mathcal{S}(\mathbb{R}^n))\,\cap\, C([0,+\infty),\mathcal{S}'(\mathbb{R}^n)))\longrightarrow\mathcal{G}^\mho(\mathcal{S}'(\mathbb{R}^n)).
$$
This allows us to set up initial conditions as
$$
\left.u\right|_0=u_0,\quad u\in\mathcal{G}^\mho(C^\infty(\mathbb{R}_+,\mathcal{S}(\mathbb{R}^n))\cap C([0,+\infty),\mathcal{S}'(\mathbb{R}^n))),\quad u_0\in\mathcal{G}^\mho(\mathcal{S}'(\mathbb{R}^n)).
$$
Together with equation (\ref{VWEquation}), this gives the very weak analogue of the Cauchy problem in Proposition \ref{pr3}.

\subsection{The very weak solution maps}

We will need our heat semigroup operators to act on singular items. Operators $\{\mathrm{W}_{s,t}\}_{(s,t)\in\bar\Delta}$, defined originally as
$$
\mathrm{W}_{s,t}\in\mathcal{B}(\mathcal{S}(\mathbb{R}^n))\,\cap\,\mathcal{B}(\mathcal{S}'(\mathbb{R}^n),\mathcal{S}(\mathbb{R}^n)),\quad\forall(s,t)\in\bar\Delta,
$$
extend naturally with the help of Corollary \ref{VWContCorr} to
$$
\mathrm{W}_{s,t}\in\mathcal{B}(\mathcal{G}^\mho(\mathcal{S}(\mathbb{R}^n)))\,\cap\,\mathcal{B}(\mathcal{G}^\mho(\mathcal{S}'(\mathbb{R}^n)),\mathcal{G}^\mho(\mathcal{S}(\mathbb{R}^n))),\quad\forall(s,t)\in\bar\Delta.
$$
In the following remark, we denote some necessary notation concerning seminorms of locally convex topological spaces $C^{\infty}(\mathbb{R}_+, \mathcal{S}(\mathbb{R}^n))$, $C([0,+\infty),\mathcal{S}(\mathbb{R}^n))$, and $C([0,+\infty),\mathcal{S}'(\mathbb{R}^n))$, as well as $C^\infty(\mathbb{R}_+,W^{q,\infty}(\mathbb{R}^n))$ and $C([0,+\infty),L^q(\mathbb{R}^n))$ which will be used in the sequel.
\begin{remark}\label{notation2}
Let $h\in C^{\infty}(\mathbb{R}_+, \mathcal{S}(\mathbb{R}^n))$ and $K\subset\mathbb{R}_+$ be a compact set. Then, 
for every $\alpha, \beta\in\mathbb{N}_0^n$, $\gamma\in\mathbb{N}_0$, we denote
\begin{equation}
	\left\| h \right\|_{\alpha,\beta,\gamma, K} 
	\doteq \sup\limits_{t \in K} \sup\limits_{x \in \mathbb{R}^n} \left| x^\alpha \partial_x^\beta \partial_t^\gamma  h (t, x) \right|.
\end{equation}
Moreover, let $h\in C([0,+\infty),\mathcal{S}(\mathbb{R}^n))$ and $b\in\mathbb{N}$. Then, we denote for every $\alpha,\beta \in \mathbb{N}_0^n$
\begin{equation}
	\|h\|_{\alpha,\beta,b}\doteq \sup_{t \in [0,b]}\|h(t)\|_{\alpha,\beta}.
\end{equation}
Furthermore, let $h\in C([0,+\infty),\mathcal{S}'(\mathbb{R}^n))$ and $b\in\mathbb{N}$. Then, there exists $F=\{(\alpha_{i},\beta_{i})\}_{i=1}^d$ for some $d\in\mathbb{N}$ such that
\begin{equation}
	\|h\|_{F,b}\doteq \sup_{t\in[0,b]}\left\|h(t)  \right\|_{F},
\end{equation} 
where $\|\cdot\|_F$ is introduced in Remark \ref{notation}.

In addition, let $h\in C^\infty(\mathbb{R}_+,W^{q,\infty}(\mathbb{R}^n))$ and $K\subset\mathbb{R}_+$ be a compact set. Then, 
for every $j\in\mathbb{N}_0$, $\ell\in\mathbb{N}_0^n$, and $q\in [1,\infty]$ we denote
$$
\|h\|_{q,j,\ell,K}\doteq\sup_{t\in K}\left\|\partial^{j}_{t}\partial_{x}^{\ell} h (t)\right\|_q.
$$
Finally,
let $h\in C([0,+\infty),L^q(\mathbb{R}^n))$ and $b\in\mathbb{N}$, we denote
$$
\|h\|_{q,b}\doteq \sup_{t\in[0,b]}\|h(t)\|_q.
$$
\end{remark}
Moreover, the following homogeneous and inhomogeneous solution maps can be constructed.
\begin{proposition}\label{pr9}
 The solution maps, homogeneous
$$
\mathrm{W}_{0,\natural}:\mathcal{S}'(\mathbb{R}^n)\longrightarrow C^\infty(\mathbb{R}_+,\mathcal{S}(\mathbb{R}^n))\,\cap\,C([0,+\infty),\mathcal{S}'(\mathbb{R}^n)),
$$
and inhomogeneous
$$
\int\limits_0^\natural\mathrm{W}_{\sharp,\natural}:C^\infty(\mathbb{R}_+,\mathcal{S}(\mathbb{R}^n))\cap C([0,\infty),\mathcal{S}(\mathbb{R}^n))\longrightarrow C^\infty(\mathbb{R}_+,\mathcal{S}(\mathbb{R}^n))\,\cap\,C([0,+\infty),\mathcal{S}'(\mathbb{R}^n)),
$$
defined by
$$
\mathrm{W}_{0,\natural}h(t)=\mathrm{W}_{0,t}h,\quad\forall t\in[0,+\infty),\quad\forall h\in\mathcal{S}'(\mathbb{R}^n),
$$
and
$$
\int\limits_0^\natural\mathrm{W}_{\sharp,\natural}h(t)=\int\limits_0^t\mathrm{W}_{s,t}h(s)ds,\quad\forall t\in[0,+\infty),\quad\forall h\in C^\infty(\mathbb{R}_+,\mathcal{S}(\mathbb{R}^n))\cap C([0,\infty),\mathcal{S}(\mathbb{R}^n)),
$$
respectively, extend to
\begin{equation}\label{homo}
	\mathrm{W}_{0,\natural}\in\mathcal{B}\bigl(\mathcal{G}^\mho(\mathcal{S}'(\mathbb{R}^n)),\mathcal{G}^\mho(C^\infty(\mathbb{R}_+,\mathcal{S}(\mathbb{R}^n))))\,\cap\,C([0,+\infty),\mathcal{S}'(\mathbb{R}^n))\bigr),
\end{equation}
and
\begin{equation}\label{inhomo}
	\int\limits_0^\natural\mathrm{W}_{\sharp,\natural}\in\mathcal{B}\bigl(\mathcal{G}^\mho(C^\infty(\mathbb{R}_+,\mathcal{S}(\mathbb{R}^n))\cap\,C([0,+\infty),\mathcal{S}(\mathbb{R}^n))),\mathcal{G}^\mho(C^\infty(\mathbb{R}_+,\mathcal{S}(\mathbb{R}^n))\,\cap\,C([0,+\infty),\mathcal{S}'(\mathbb{R}^n))\bigr).
\end{equation}
\end{proposition}
\begin{proof}
To prove \eqref{homo} and \eqref{inhomo}, we begin by establishing the continuity and linearity of the following maps:
\begin{itemize}
\item [i.]
\begin{align}\label{i}
	\mathrm{W}_{0,\natural}&:\mathcal{S}'(\mathbb{R}^n)\longrightarrow C^\infty(\mathbb{R}_+,\mathcal{S}(\mathbb{R}^n)),\notag\\
	&\mathrm{W}_{0,\natural}h(t)=\mathrm{W}_{0,t}h,\quad\forall t\in[0,+\infty),\quad\forall h\in\mathcal{S}'(\mathbb{R}^n),
\end{align}
\item[ii.]
\begin{align}\label{ii}
	\mathrm{W}_{0,\natural}&:\mathcal{S}'(\mathbb{R}^n)\longrightarrow C([0,+\infty),\mathcal{S}'(\mathbb{R}^n)),\notag\\
	&\mathrm{W}_{0,\natural}h(t)=\mathrm{W}_{0,t}h,\quad\forall t\in[0,+\infty),\quad\forall h\in\mathcal{S}'(\mathbb{R}^n),
\end{align}
as well as
\item [iii.]
\begin{align}\label{iii}
	\int\limits_0^\natural\mathrm{W}_{\sharp,\natural}&:C^\infty(\mathbb{R}_+,\mathcal{S}(\mathbb{R}^n))\cap\,C([0,+\infty),\mathcal{S}(\mathbb{R}^n))\longrightarrow C^\infty(\mathbb{R}_+,\mathcal{S}(\mathbb{R}^n)),\notag\\
	&\int\limits_0^\natural\mathrm{W}_{\sharp,\natural}h(t)=\int\limits_0^t\mathrm{W}_{s,t}h(s)ds,\quad\forall t\in[0,+\infty),\quad\forall h\in C^\infty(\mathbb{R}_+,\mathcal{S}(\mathbb{R}^n))\cap\,C([0,+\infty),\mathcal{S}(\mathbb{R}^n)),
\end{align}
\item[iv.]
\begin{align}\label{iv}
	\int\limits_0^\natural\mathrm{W}_{\sharp,\natural}&:C^\infty(\mathbb{R}_+,\mathcal{S}(\mathbb{R}^n))\cap\,C([0,+\infty),\mathcal{S}(\mathbb{R}^n))\longrightarrow C([0,+\infty),\mathcal{S}'(\mathbb{R}^n)),\notag\\
	&\int\limits_0^\natural\mathrm{W}_{\sharp,\natural}h(t)=\int\limits_0^t\mathrm{W}_{s,t}h(s)ds,\quad\forall t\in[0,+\infty),\quad\forall h\in C^\infty(\mathbb{R}_+,\mathcal{S}(\mathbb{R}^n))\cap\,C([0,+\infty),\mathcal{S}(\mathbb{R}^n)).
\end{align}
Then, Corollary \ref{VWContCorr} allows us to complete the proof.
\end{itemize}
Linearity of operators \eqref{i}, \eqref{ii}, \eqref{iii}, and \eqref{iv} are immediate.

To prove the continuity of \eqref{i}, let $K \subset \mathbb{R}_+$ be a compact set and $\alpha, \beta\in\mathbb{N}_0^n$, $\gamma\in\mathbb{N}_0$. Then, since $h\in\mathcal{S}'(\mathbb{R}^n)$, it follows from Remark \ref{notation} that there exists a finite set $F=\{(\alpha_{i}, \beta_{i})\}_{i=1}^d\subset\mathbb{N}_0^{2n}$ for some $d\in\mathbb{N}$ such that
\begin{align*}
	\left\| \mathrm{W}_{0,\natural} h \right\|_{\alpha,\beta,\gamma, K} 
	&= \sup_{t \in K} \sup_{x \in \mathbb{R}^n} \left| x^\alpha \partial_x^\beta \partial_t^\gamma \left( \mathcal{W}(\cdot;0,t) * h  (x)\right)\right| = \sup_{t\in K}\sup_{x\in\mathbb{R}^n}\left|x^{\alpha}\partial_{x}^{\beta}\partial_t^{\gamma}(h(\mathcal{W}(x-\cdot;0,t)))\right|\\
	&=\sup_{t\in K}\sup_{x\in\mathbb{R}^n}\left|(h(x^{\alpha}\partial_{x}^{\beta}\partial_t^{\gamma}\mathcal{W}(x-\cdot;0,t)))\right| \leq \| h \|_{F} \sum_{i=1}^d \sup_{t \in K} \left\| (\cdot)^\alpha \partial_x^\beta \partial_t^\gamma \mathcal{W}(\cdot - \colon; 0, t) \right\|_{\alpha_i, \beta_i}\\
	&\lesssim \| h \|_{F}.
\end{align*}
It is worth noting that in the above estimate, we used the fact that since $\mathcal{W} \in \mathcal{S}(\mathbb{R}^n) \otimes C^\infty(\Delta)$, the mapping
$$
 t \mapsto \mathcal{W}(x - \cdot; 0, t) \in \mathcal{S}(\mathbb{R}^n),
$$
belongs to $C^\infty(\mathbb{R}_+, \mathcal{S}(\mathbb{R}^n))$. Hence, the proof is complete.

To prove \eqref{ii}, let $b\in\mathbb{N}$ and $F'=\{(\alpha_i',\beta_i')\}_{i=1}^{d'}\subset\mathbb{N}_0^{2n}$ for some $d'\in\mathbb{N}$. Since $h\in\mathcal{S}'(\mathbb{R}^n)$, it follows from Remark \ref{notation} that there exists a finite set $F=\{(\alpha_{i}, \beta_{i})\}_{i=1}^d\subset \mathbb{N}_0^{2n}$ for some $d\in\mathbb{N}$, such that for every $\phi\in\mathcal{S}(\mathbb{R}^n)$,
\begin{align}
	&\|\mathrm{W}_{0,\natural} h\|_{F',b}=\sup_{t\in[0,b]}\left\|\left(\mathrm{W}_{0,\natural} h \right)(t)  \right\|_{F'}=\sup_{t\in[0,b]}\left\|\mathrm{W}_{0,t} h \right\|_{F'}\notag\\
	&=\sup_{t\in[0,b]}\sup_{\substack{\phi\in \mathcal{S}\\\sum_{i=1}^{d'}\|\phi\|_{\alpha'_i,\beta'_i}\leq 1}}\left| \mathrm{W}_{0,t} h(\phi)\right| =\sup_{t\in[0,b]}\sup_{\substack{\phi\in \mathcal{S}\\\sum_{i=1}^{d'}\|\phi\|_{\alpha'_i,\beta'_i}\leq 1}}\left|\left(\mathcal{W}(\cdot; 0, t)*h\right)(\phi)\right| \notag\\
 	&=\sup_{t\in[0,b]}\sup_{\substack{\phi\in \mathcal{S}\\\sum_{i=1}^{d'}\|\phi\|_{\alpha'_i,\beta'_i}\leq 1}}\left|h(\phi*\mathcal{W}(\cdot; 0, t))\right| \notag\\
	&\leq \sup_{t\in[0,b]}\sup_{\substack{\phi\in \mathcal{S}\\\sum_{i=1}^{d'}\|\phi\|_{\alpha'_i,\beta'_i}\leq 1}}\|h\|_F\sum^{d}_{i=1}\left\| \phi*\mathcal{W}( \cdot; 0, t)\right\| _{\alpha_{i},\beta_{i}},
\end{align}
By a similar argument as in \eqref{convolution}, we get that 
\begin{align}
	&\|\mathrm{W}_{0,\natural} h\|_{F',b}\leq\sup_{t\in[0,b]}\sup_{\substack{\phi\in \mathcal{S}\\\sum_{i=1}^{d'}\|\phi\|_{\alpha'_i,\beta'_i}\leq 1}}\|h\|_F\sum^{d}_{i=1}\left\| \phi*\mathcal{W}( \cdot; 0, t)\right\| _{\alpha_{i},\beta_{i}}\notag\\
	&\leq \sup_{t\in[0,b]}\sup_{\substack{\phi\in \mathcal{S}\\\sum_{i=1}^{d'}\|\phi\|_{\alpha'_i,\beta'_i}\leq 1}} \|h\|_F\sum^{d}_{i=1}\left(C_{\alpha_{i},\beta_{i}}\left\| \phi\right\|_{0,0}+C_{\alpha_{i}}\left\| \phi\right\| _{\alpha_{i},0}\right)  \notag\\
	&= \|h\|_F\sum^{d}_{i=1}\left(C_{\alpha_{i},\beta_{i}}\left\| \phi\right\|_{0,0}+C_{\alpha_{i}}\left\| \phi\right\| _{\alpha_{i},0}\right) \lesssim \|h\|_F.
	\end{align}
Thus, the proof is complete.

To prove the continuity of \eqref{iii}, let $h(t)\in\mathcal{S}(\mathbb{R}^{n})$ for every $t\in[0,\infty)$ and $K\subset \mathbb{R}_{+}$ be a compact set. Hence, for every fixed $\alpha,\beta\in\mathbb{N}^{n}_{0}$, $\gamma\in\mathbb{N}_{0}$, it follows that 
\begin{align}\label{normsigma}
	\left\|\int\limits_0^\natural\mathrm{W}_{\sharp,\natural}h\right\|_{\alpha,\beta,\gamma,K}&=\sup_{t\in K}\sup_{x\in\mathbb{R}^n}\left|x^{\alpha}\partial_{x}^{\beta}\partial_t^\gamma\int\limits_0^t\mathrm{W}_{s,t}h(s,x)ds\right|\notag\\
	&=\sup_{t\in K}\sup_{x\in\mathbb{R}^n}\left|x^{\alpha}\partial_{x}^{\beta}\left(\gamma\partial_t^{\gamma-1}h(t,x)+\int\limits_0^t\partial_{t}^\gamma\mathrm{W}_{s,t}h(s,x)ds\right)\right|\notag\\
	&\leq \sup_{t\in K}\sup_{x\in\mathbb{R}^n}\left|x^{\alpha}\partial_{x}^{\beta}\gamma\partial_t^{\gamma-1}h(t,x)\right|+\sup_{t\in K}\sup_{x\in\mathbb{R}^n}\left|x^{\alpha}\partial_x^{\beta}\int\limits_0^t\partial_{t}^\gamma\mathrm{W}_{s,t}h(s,x)ds\right|\notag\\
	&\leq \gamma\|h\|_{\alpha,\beta,\gamma-1,K}+\sup_{t\in K}\sup_{x\in\mathbb{R}^n}\int\limits_0^t\left|x^{\alpha}\partial_x^{\beta}\partial_{t}^\gamma\left(\mathcal{W}(\cdot;s,t)\ast h(s,\cdot)(x)\right)\right|ds.
\end{align}
One can continue   \eqref{normsigma} as follows:
\begin{align}\label{normsigma0}
	\sup_{t\in K}\sup_{x\in\mathbb{R}^n}\int\limits_0^t\left|x^{\alpha}\partial_x^{\beta}\partial_{t}^\gamma\left(\mathcal{W}(\cdot;s,t)\ast h(s)(x)\right) \right|ds&=\sup_{t\in K}\sup_{x\in\mathbb{R}^n}\int\limits_0^t\left|\left(x^{\alpha}\partial_x^{\beta}\partial_{t}^\gamma\mathcal{W}(\cdot;s,t)\right)\ast h(s)(x)\right|ds\notag\\
	&\leq\sup_{t\in K}\int\limits_0^t\left\|\left((\cdot)^{\alpha}\partial_x^{\beta}\partial_{t}^\gamma\mathcal{W}(\cdot;s,t)\right)\ast h(s)\right\|_{\infty}ds\notag\\
	&\leq\sup_{t\in K}\int\limits_0^t\left\|(\cdot)^{\alpha}\partial_x^{\beta}\partial_{t}^\gamma\mathcal{W}(\cdot;s,t)\right\|_{1} \left\|h(s)\right\|_{\infty}ds\notag\\
	&\leq \max_{s\in[0,[\max K]+1]}\left\|h(s)\right\|_{\infty}\sup_{t\in K}\int\limits_0^t\left\|(\cdot)^{\alpha}\partial_x^{\beta}\partial_{t}^\gamma\mathcal{W}(\cdot;s,t)\right\|_{1} ds,
\end{align} 
where in the last line $\sup\limits_{t\in K}\int\limits^{t}_{0}\left\|(\cdot)^{\alpha}\partial^{\beta}_{x}\partial^{\gamma}_{t} \mathcal{W}(\cdot;s,t)\right\|_{1}ds<\infty$ holds because  $\mathcal{W}\in\mathcal{S}(\mathbb{R}^n)\otimes C^\infty(\Delta)$, so the function $(s,t)\mapsto\left\|(\cdot)^{\alpha}\partial^{\beta}_{x}\partial^{\gamma}_{t} \mathcal{W}(\cdot;s,t)\right\|_{1}$ is continuous for $\forall (s,t)\in \Delta$. Thereby, for each fixed $t\in\mathbb{R}_{+}$, the integrand is continuous in $s$, and integrable over $[0,t]$ and the resulting integral depends continuously on $t$. Since $K\subset \mathbb{R}_{+}$ is compact, the supremum over $t\in K$  is finite.
Now, using \eqref{normsigma0} in \eqref{normsigma}, we have
\begin{align}
	\left\|\int\limits_0^\natural\mathrm{W}_{\sharp,\natural}h\right\|_{\alpha,\beta,\gamma,K}
	&\leq \gamma\|h\|_{\alpha,\beta,\gamma-1,K}+C_{\alpha,\beta,\gamma,K}\left\| h\right\|_{0,0,[\max K]+1}.
\end{align}
Thus, the proof of part $iii$ is complete.

To prove \eqref{iv}, let $ h\in C^\infty(\mathbb{R}_+,\mathcal{S}(\mathbb{R}^n))\cap\,C([0,+\infty),\mathcal{S}(\mathbb{R}^n))$ for every $t\in[0,+\infty)$ and $b\in \mathbb{N}$ as well as $F'=\{(\alpha'_{i},\beta'_{i})\}^{\ell}_{i=1}\subset \mathbb{N}^{2n}_{0}$ for some $\ell\in\mathbb{N}$. Hence, one can see that for every $\phi\in\mathcal{S}(\mathbb{R}^{n})$,
\begin{align}\label{iv''}
	\left\|	\int\limits_0^\natural\mathrm{W}_{\sharp,\natural}h  \right\|_{F',b}&=\sup_{t\in[0,b]}\left\|\left( 	\int\limits_0^\natural\mathrm{W}_{\sharp,\natural}h\right) (t)  \right\|_{F'}\notag=\sup_{t\in[0,b]}\left\|	\int\limits_0^t\mathrm{W}_{s,t}h(s)ds \right\|_{F'}\notag\\
	&\leq\sup_{t\in[0,b]}\int\limits_0^t\left\| \mathrm{W}_{s,t}h(s)\right\| _{F'}ds=\sup_{t\in[0,b]}\int\limits_0^t  \left\| \mathcal{W}( \cdot; s, t)*h(s)\right\| _{F'}ds\notag\\\notag
	&=\sup_{t\in[0,b]}\int\limits_0^t  \sup_{\substack{\phi\in \mathcal{S}\\\sum^{\ell}_{i=1}\|\phi\|_{\alpha'_{i},\beta'_{i}}\leq 1}}\left|  \mathcal{W}( \cdot; s, t)*h(s)(\phi)\right|  ds\notag\\
	&=\sup_{t\in[0,b]}\int\limits_0^t  \sup_{\substack{\phi\in \mathcal{S}\\\sum^{\ell}_{i=1}\|\phi\|_{\alpha'_{i},\beta'_{i}}\leq 1}}\left| h(s)\left(\phi*\mathcal{W}( \cdot; s, t) \right)\right| ds\notag\\
	&=\sup_{t\in[0,b]}\int\limits_0^t  \sup_{\substack{\phi\in \mathcal{S}\\\sum^{\ell}_{i=1}\|\phi\|_{\alpha'_{i},\beta'_{i}}\leq 1}}\left|\int\limits_{\mathbb{R}^n} h(s,x)\left(\phi\ast\mathcal{W}(\cdot;s,t)\right)(x)dx\right|ds\notag\\
	&\leq \sup_{t\in[0,b]}\int\limits_0^t  \sup_{\substack{\phi\in \mathcal{S}\\\sum^{\ell}_{i=1}\|\phi\|_{\alpha'_{i},\beta'_{i}}\leq 1}}\int\limits_{\mathbb{R}^n}|h(s,x)|\left|\left(\phi\ast\mathcal{W}(\cdot;s,t)\right)(x)\right|dxds\notag\\
	&\leq \sup_{t\in[0,b]}\int\limits_0^t  \sup_{\substack{\phi\in \mathcal{S}\\\sum^{\ell}_{i=1}\|\phi\|_{\alpha'_{i},\beta'_{i}}\leq 1}}\sup_{x\in\mathbb{R}^n}|h(s,x)|\left\|\left(\phi\ast\mathcal{W}(\cdot;s,t)\right)\right\|_1ds
	\notag\\\notag\\
	&\leq \sup_{t\in[0,b]}\int\limits_0^t  \sup_{\substack{\phi\in \mathcal{S}\\\sum^{\ell}_{i=1}\|\phi\|_{\alpha'_{i},\beta'_{i}}\leq 1}}\sup_{x\in\mathbb{R}^n}|h(s,x)|\left\|\phi\right\|_1\left\|\mathcal{W}(\cdot;s,t)\right\|_1ds
	\notag\\\notag\\
	&\leq \left\|\phi\right\|_1 \sup_{t\in[0,b]}\int\limits_0^t \|h(s,\cdot)\|_{\infty} ds\leq b \|\phi\|_1 \sup_{s\in[0,b]}\|h(s,\cdot)\|_{\infty}\notag\\
	&\lesssim_{b} \left\| h\right\|_{0,0,b}.
\end{align}
Thus, the proof of part iv is complete. 
\end{proof}

With the help of the solutions maps, we can state and prove the very weak analogue of Proposition \ref{pr3}.

\begin{proposition} 
Let $u_0\in\mathcal{G}^\mho(\mathcal{S}'(\mathbb{R}^n))$ and $f\in\mathcal{G}^\mho(C^\infty(\mathbb{R}_+,\mathcal{S}(\mathbb{R}^n))\cap C([0,\infty),\mathcal{S}(\mathbb{R}^{n}))$. Then
$$
u\doteq\mathrm{W}_{0,\natural}u_0+\int\limits_0^\natural\mathrm{W}_{\sharp,\natural}f\in\mathcal{G}^\mho(C^\infty(\mathbb{R}_+,\mathcal{S}(\mathbb{R}^n))\,\cap\,C([0,+\infty),\mathcal{S}'(\mathbb{R}^n))),
$$
solves the Cauchy problem
\begin{equation}\label{CauchyvwSc}
\begin{cases}
	\partial_tu-\mathcal{L}_tu=f,\\
	\left.u\right|_0=u_0.
\end{cases}
\end{equation}
Moreover, if $u_0\ge0$ and $f\ge0$ then $u\ge0$.
\end{proposition}
\begin{proof}
Since $u_0 \in \mathcal{G}^\mho(\mathcal{S}'(\mathbb{R}^n))$, it follows from Proposition \ref{pr9} that
$$
\mathrm{W}_{0,\natural}u_0 \in \mathcal{G}^\mho\big(C^\infty(\mathbb{R}_+, \mathcal{S}(\mathbb{R}^n)) \cap C([0,+\infty), \mathcal{S}'(\mathbb{R}^n))\big).
$$
Similarly, since $ f \in \mathcal{G}^\mho(C^\infty(\mathbb{R}_+, \mathcal{S}(\mathbb{R}^n))\cap C([0,\infty),\mathcal{S}(\mathbb{R}^{n}))$, we again obtain from Proposition \ref{pr9} that
$$
\int_0^\natural \mathrm{W}_{\sharp,\natural}f \in \mathcal{G}^\mho\big(C^\infty(\mathbb{R}_+, \mathcal{S}(\mathbb{R}^n)) \cap C([0,+\infty), \mathcal{S}'(\mathbb{R}^n))\big).
$$
Therefore,
$$
u \in \mathcal{G}^\mho\big(C^\infty(\mathbb{R}_+, \mathcal{S}(\mathbb{R}^n)) \cap C([0,+\infty), \mathcal{S}'(\mathbb{R}^n))\big).
$$
Now, we aim to show that $u$ satisfies the Cauchy problem \eqref{CauchyvwSc}. To this end, let $u_0 \in \mathcal{G}^\mho(\mathcal{S}'(\mathbb{R}^n))$. This means that
$
u_0 \in \mathcal{M}^\mho(\mathcal{S}'(\mathbb{R}^n)) / \mathcal{N}^\mho(\mathcal{S}'(\mathbb{R}^n)).$
Therefore,
\begin{align}\label{u0colombeu}
	u_0 = [u_{0,\epsilon}]_{\epsilon \in (0,1)}  
	\doteq \{u_{0,\epsilon}\}_{\epsilon\in(0,1)} \mod \mathcal{N}^\mho(\mathcal{S}'(\mathbb{R}^n)),
\end{align}
where $\{u_{0,\epsilon}\}_{\epsilon\in(0,1)} \in \mathcal{M}^\mho(\mathcal{S}'(\mathbb{R}^n))$, which by the definition of moderateness means that $u_{0,\epsilon} \in \mathcal{S}'(\mathbb{R}^n)$ for every $\epsilon\in(0,1)$.
With a similar reasoning, since $ f \in \mathcal{G}^\mho(C^\infty(\mathbb{R}_+, \mathcal{S}(\mathbb{R}^n))\cap C([0,\infty),\mathcal{S}(\mathbb{R}^{n})))$, we write
$$
f = [f_\epsilon]_{\epsilon \in (0,1)} \doteq \{f_\epsilon\}_{\epsilon\in(0,1)} \mod \mathcal{N}^\mho(C^\infty(\mathbb{R}_+, \mathcal{S}(\mathbb{R}^n))\cap C([0,\infty),\mathcal{S}(\mathbb{R}^{n}))),
$$
where $\{f_\epsilon\}_{\epsilon\in(0,1)} \in \mathcal{M}^\mho(C^\infty(\mathbb{R}_+, \mathcal{S}(\mathbb{R}^n))\cap C([0,\infty),\mathcal{S}(\mathbb{R}^{n})))$, which means that $
f_\epsilon \in C^\infty(\mathbb{R}_+, \mathcal{S}(\mathbb{R}^n)) \cap C([0,\infty),\mathcal{S}(\mathbb{R}^{n}))
$
for every $\epsilon\in(0,1)$.
Furthermore, from
$
\mathrm{W}_{0,\natural}u_0 \in \mathcal{G}^\mho\big(C^\infty(\mathbb{R}_+, \mathcal{S}(\mathbb{R}^n)) \cap C([0,+\infty), \mathcal{S}'(\mathbb{R}^n))\big),
$
we have
$$
\mathrm{W}_{0,\natural}u_0=[\mathrm{W}^{\epsilon}_{0,\natural}u_0]_{\epsilon\in(0,1)} \doteq \{\mathrm{W}^\epsilon_{0,\natural}u_0\}_{\epsilon\in(0,1)}\mod \mathcal{N}^\mho\big(C^\infty(\mathbb{R}_+, \mathcal{S}(\mathbb{R}^n)) \cap C([0,+\infty),\mathcal{S}'(\mathbb{R}^n))\big),
$$
where $\mathrm{W}^\epsilon_{0,\natural}u_0 \doteq \mathrm{W}_{0,\natural}u_{0,\epsilon} \in C^\infty(\mathbb{R}_+, \mathcal{S}(\mathbb{R}^n)) \cap C([0,+\infty),\mathcal{S}'(\mathbb{R}^n))$ for every $\epsilon\in(0,1)$.
In addition, using the fact that
$
\int_0^\natural \mathrm{W}_{\sharp,\natural}f \in \mathcal{G}^\mho\big(C^\infty(\mathbb{R}_+, \mathcal{S}(\mathbb{R}^n)) \cap C([0,+\infty), \mathcal{S}'(\mathbb{R}^n))\big),
$
it follows that
$$
\int_0^\natural \mathrm{W}_{\sharp,\natural}f=\left[\int_0^{\natural,\epsilon} \mathrm{W}_{\sharp,\natural}f\right]_{\epsilon\in(0,1)} \doteq \left\{\int_0^{\natural,\epsilon} \mathrm{W}_{\sharp,\natural}f\right\}_{\epsilon\in(0,1)} \mod \mathcal{N}^\mho\big(C^\infty(\mathbb{R}_+, \mathcal{S}(\mathbb{R}^n)) \cap C([0,+\infty), \mathcal{S}'(\mathbb{R}^n))\big),
$$
where
$$
\int_0^{\natural,\epsilon} \mathrm{W}_{\sharp,\natural}f \doteq \int_0^\natural \mathrm{W}_{\sharp,\natural}f_\epsilon \in C^\infty(\mathbb{R}_+, \mathcal{S}(\mathbb{R}^n)) \cap C([0,+\infty), \mathcal{S}'(\mathbb{R}^n)), ~ \forall \epsilon\in(0,1).
$$
For $u=\mathrm{W}_{0,\sharp}u_0 + \int_0^\natural \mathrm{W}_{\sharp,\natural}f$, it follows that
$$
u_\epsilon = \mathrm{W}_{0,\natural}u_{0,\epsilon} + \int_0^\natural \mathrm{W}_{\sharp,\natural}f_\epsilon \mod \mathcal{N}^\mho\big(C^\infty(\mathbb{R}_+, \mathcal{S}(\mathbb{R}^n)) \cap C([0,+\infty), \mathcal{S}'(\mathbb{R}^n))\big), ~ \forall \epsilon\in(0,1),
$$
where
$$
\mathrm{W}_{0,\natural}u_{0,\epsilon} + \int_0^\natural \mathrm{W}_{\sharp,\natural}f_\epsilon\in C^\infty(\mathbb{R}_+, \mathcal{S}(\mathbb{R}^n)) \cap C([0,+\infty), \mathcal{S}'(\mathbb{R}^n)),~ \forall \epsilon\in(0,1).
$$
It now follows from the definitions \eqref{VWDtDef} and \eqref{VWLtDef}, together with Proposition \ref{pr8}, that
$$
\partial_t u =[\partial_t u_\epsilon]_{\epsilon\in(0,1)}\doteq \{\partial_t u_\epsilon\}_{\epsilon\in(0,1)}  \mod \mathcal{N}^\mho(C^\infty(\mathbb{R}_+, \mathcal{S}(\mathbb{R}^n))\cap C([0,\infty),\mathcal{S}(\mathbb{R}^{n}))),
$$
and
\begin{align*}
	\mathcal{L}_t u &=[\mathcal{L}^\epsilon_t u_\epsilon]_{\epsilon\in(0,1)}\doteq \{\mathcal{L}^\epsilon_t u_\epsilon\}_{\epsilon\in(0,1)}  \mod \mathcal{N}^\mho(C^\infty(\mathbb{R}_+, \mathcal{S}(\mathbb{R}^n))\cap C([0,\infty),\mathcal{S}(\mathbb{R}^{n}))) \\
	&= \{\langle \partial_x, a_\epsilon \partial_x \rangle u_\epsilon \}_{\epsilon\in(0,1)} \mod \mathcal{N}^\mho(C^\infty(\mathbb{R}_+, \mathcal{S}(\mathbb{R}^n))\cap C([0,\infty), \mathcal{S}(\mathbb{R}^{n}))),
\end{align*}
where $\partial_t u_\epsilon \in C^\infty(\mathbb{R}_+, \mathcal{S}(\mathbb{R}^n))\cap C([0,\infty), \mathcal{S}(\mathbb{R}^{n}))$ and $\mathcal{L}^\epsilon_t u_\epsilon \in C^\infty(\mathbb{R}_+, \mathcal{S}(\mathbb{R}^n))\cap C([0,\infty), \mathcal{S}(\mathbb{R}^{n}))$ for every $\epsilon\in(0,1)$.
Thus, using Proposition \ref{pr3}, we obtain
\begin{align*}
	\partial_t u - \mathcal{L}_t u &= \{\partial_t u_\epsilon - \mathcal{L}^\epsilon_t u_\epsilon\}_{\epsilon\in(0,1)}\mod \mathcal{N}^\mho(C^\infty(\mathbb{R}_+, \mathcal{S}(\mathbb{R}^n))\cap C([0,\infty),\mathcal{S}(\mathbb{R}^{n})))\\ 
	&= \{f_\epsilon\}_{\epsilon\in(0,1)} \mod \mathcal{N}^\mho(C^\infty(\mathbb{R}_+, \mathcal{S}(\mathbb{R}^n))\cap C([0,\infty),\mathcal{S}(\mathbb{R}^{n}))).
\end{align*}
Hence, we have constructed a net of solutions $\{u_\epsilon\}_{\epsilon \in (0,1)}$ to the Cauchy problem \eqref{CauchyvwSc}.

To establish the non-negativity of the solution \( u \in \mathcal{G}^\mho\big(C^\infty(\mathbb{R}_+, \mathcal{S}(\mathbb{R}^n)) \cap C([0,+\infty), \mathcal{S}'(\mathbb{R}^n))\big) \), we follow the same reasoning as in Proposition \ref{pr3}.
Here, the solution is given by
$$
u = \left[ \mathrm{W}_{0,\natural}u_{0,\epsilon} + \int_0^\natural \mathrm{W}_{\sharp,\natural}f_\epsilon \right]_{\epsilon \in (0,1)}.
$$
Assuming  $ u_{0,\epsilon} \geq 0 $, $f_\epsilon \geq 0 $ for every $\epsilon \in (0,1)$, and using the positivity of our heat kernel, it follows that each $ u_\epsilon(t) \geq 0 $. Hence, the net $ \{u_\epsilon\}_{\epsilon\in(0,1)} $ is non-negative.
Therefore, the very weak solution $ u$ is non-negative. 
\end{proof}

\subsection{The very weak $L^q$-Cauchy problem}

In this section we will set up the very weak formulation of the $L^q$-well-posedness of the Cauchy problem. This will involve restricting initial data to $u_0\in L^q(\mathbb{R}^n)$ and relaxing the source term to $f\in C^\infty(\mathbb{R}_+,L^q(\mathbb{R}^n))\cap C([0,\infty),L^{q}(\mathbb{R}^{n}))$. The initial conditions at $t=0$ should be understood as a limit as $t\to0+$ in the topology of $L^q(\mathbb{R}^n)$. As per Proposition \ref{pr5}, once the solution $u\in C^\infty(\mathbb{R}_+,W^{q,\infty}(\mathbb{R}^n))$ satisfies the initial condition
$$
\lim\limits_{t\to0+}u(t)=u_0\in L^q(\mathbb{R}^n),\quad q\in[1,+\infty],
$$
we can interpret this as
$$
u\in C^\infty(\mathbb{R}_+,W^{q,\infty}(\mathbb{R}^n))\,\cap\,C([0,+\infty),L^q(\mathbb{R}^n)),
$$
and
$$
\left.u\right|_0=u_0.
$$
The space $C^\infty(\mathbb{R}_+,W^{q,\infty}(\mathbb{R}^n))\,\cap\,C([0,+\infty),L^q(\mathbb{R}^n))$ will be taken with the union of the families of seminorms of $C^\infty(\mathbb{R}_+,W^{q,\infty}(\mathbb{R}^n))$ and $C([0,+\infty),L^q(\mathbb{R}^n))$, respectively. With this system of seminorms, we will consider the Colombeau extension
$$
\mathcal{G}^\mho(C^\infty(\mathbb{R}_+,W^{q,\infty}(\mathbb{R}^n))\,\cap\,C([0,+\infty),L^q(\mathbb{R}^n))).
$$
The operation of restriction to (or evaluation at) $t=0$ defines a linear continuous surjection
$$
\left.\phantom{A}\right|_0:C^\infty(\mathbb{R}_+,W^{q,\infty}(\mathbb{R}^n))\,\cap\,C([0,+\infty),L^q(\mathbb{R}^n))\to L^q(\mathbb{R}^n),
$$
which, by Corollary \ref{VWContCorr}, can be extended to
$$
\left.\phantom{A}\right|_0:\mathcal{G}^\mho(C^\infty(\mathbb{R}_+,W^{q,\infty}(\mathbb{R}^n))\,\cap\,C([0,+\infty),L^q(\mathbb{R}^n)))\to\mathcal{G}^\mho(L^q(\mathbb{R}^n)).
$$
This allows us to set up initial conditions as
$$
\left.u\right|_0=u_0,\quad u\in\mathcal{G}^\mho(C^\infty(\mathbb{R}_+,W^{q,\infty}(\mathbb{R}^n))\,\cap\,C([0,+\infty),L^q(\mathbb{R}^n))),\quad u_0\in\mathcal{G}^\mho(L^q(\mathbb{R}^n)).
$$
By Proposition \ref{pr4}, the heat semigroup operators $\{\mathrm{W}_{s,t}\}_{(s,t)\in\bar\Delta}$ satisfy
$$
\mathrm{W}_{s,t}\in\mathcal{B}(L^q(\mathbb{R}^n),W^{q,\infty}(\mathbb{R}^n)),\quad\forall(s,t)\in\bar\Delta,\quad q\in[1,+\infty],
$$
and by Corollary \ref{VWContCorr} they can be promoted to
$$
\mathrm{W}_{s,t}\in\mathcal{B}(\mathcal{G}^\mho(L^q(\mathbb{R}^n)),\mathcal{G}^\mho(W^{q,\infty}(\mathbb{R}^n))),\quad\forall(s,t)\in\bar\Delta,\quad q\in[1,+\infty].
$$
As before, the homogeneous and inhomogeneous solution maps can be now introduced.

\begin{proposition} For any $q\in[1,+\infty]$, the solution maps, homogeneous
\begin{equation}\label{homogeneoussolution'1}
	\mathrm{W}_{0,\natural}:L^q(\mathbb{R}^n)\longrightarrow C^\infty(\mathbb{R}_+,W^{q,\infty}(\mathbb{R}^n))\,\cap\,C([0,+\infty),L^q(\mathbb{R}^n)),
\end{equation}
and inhomogeneous
\begin{equation}\label{inhomogeneoussolution'2}
	\int\limits_0^\natural\mathrm{W}_{\sharp,\natural}:C^\infty(\mathbb{R}_+,W^{q,\infty}(\mathbb{R}^n))\,\cap\,C([0,+\infty),L^q(\mathbb{R}^n))\longrightarrow C^\infty(\mathbb{R}_+,W^{q,\infty}(\mathbb{R}^n))\,\cap\,C([0,+\infty),L^q(\mathbb{R}^n)),
\end{equation}
defined by
$$
\mathrm{W}_{0,\natural}h(t)=\mathrm{W}_{0,t}h,\quad\forall t\in[0,+\infty),\quad\forall h\in L^q(\mathbb{R}^n),
$$
and
$$
\int\limits_0^\natural\mathrm{W}_{\sharp,\natural}h(t)=\int\limits_0^t\mathrm{W}_{s,t}h(s)ds,\quad\forall t\in[0,+\infty),\quad\forall h\in C^\infty(\mathbb{R}_+,W^{q,\infty}(\mathbb{R}^n))\cap\,C([0,+\infty),L^q(\mathbb{R}^n)),
$$
respectively, extend to
\begin{equation}\label{homogeneoussolutionoperator'1}
	\mathrm{W}_{0,\natural}\in\mathcal{B}(\mathcal{G}^\mho(L^q(\mathbb{R}^n)),\mathcal{G}^\mho(C^\infty(\mathbb{R}_+,W^{q,\infty}(\mathbb{R}^n))\,\cap\,C([0,+\infty),L^q(\mathbb{R}^n)))),
\end{equation}
and
\begin{align}\label{inhomogeneoussolutionoperator'2}
	\int\limits_0^\natural\mathrm{W}_{\sharp,\natural}\in\mathcal{B}(\mathcal{G}^\mho(C^\infty(\mathbb{R}_+,W^{q,\infty}(\mathbb{R}^n))\cap\,C([0,+\infty),L^q(\mathbb{R}^n))),\mathcal{G}^\mho(C^\infty(\mathbb{R}_+,W^{q,\infty}(\mathbb{R}^n))\,\cap\,C([0,+\infty),L^q(\mathbb{R}^n)))).
\end{align}
\end{proposition}
\begin{proof}
To have the extensions \eqref{homogeneoussolutionoperator'1} and \eqref{inhomogeneoussolutionoperator'2}, similar to the argument in Proposition \ref{pr9}, referring to  Corollary \ref{VWContCorr},  we need to establish the continuity and linearity of the following maps:
\begin{itemize}
\item [i.]
\begin{align}\label{i1}
	\mathrm{W}_{0,\natural}&:L^{q}(\mathbb{R}^n)\longrightarrow C^\infty(\mathbb{R}_+,W^{q,\infty}(\mathbb{R}^n)),\notag\\
	&\mathrm{W}_{0,\natural}h(t)=\mathrm{W}_{0,t}h,\quad\forall t\in[0,+\infty),\quad\forall h\in L^{q}(\mathbb{R}^n),
\end{align}
and 
\item[ii.]
\begin{align}\label{ii1}			\mathrm{W}_{0,\natural}&: 	L^{q}(\mathbb{R}^n)\longrightarrow	C([0,+\infty),L^{q}(\mathbb{R}^n)),\notag\\
	&\mathrm{W}_{0,\natural}h(t)=\mathrm{W}_{0,t}h,\quad\forall t\in[0,+\infty),\quad\forall h\in L^{q}(\mathbb{R}^n),
\end{align}
and
\item [iii.]
\begin{align}\label{iii1}	&\int\limits_0^\natural\mathrm{W}_{\sharp,\natural}:C^\infty(\mathbb{R}_+,W^{q,\infty}(\mathbb{R}^n))\cap C([0,+\infty),L^q(\mathbb{R}^n))\longrightarrow C^\infty(\mathbb{R}_+,W^{q,\infty}(\mathbb{R}^n)),\notag\\	&\int\limits_0^\natural\mathrm{W}_{\sharp,\natural}h(t)=\int\limits_0^t\mathrm{W}_{s,t}h(s)ds,\quad\forall t\in[0,+\infty),\quad\forall h\in C^\infty(\mathbb{R}_+,W^{q,\infty}(\mathbb{R}^n))\cap C([0,+\infty),L^q(\mathbb{R}^n)),
\end{align}
as well as 
\item[iv.]
\begin{align}\label{iiii1}	&\int\limits_0^\natural\mathrm{W}_{\sharp,\natural}:C^{\infty}(\mathbb{R}_{+}, W^{q,\infty}(\mathbb{R}^{n}))\cap C([0,+\infty),L^q(\mathbb{R}^n))\longrightarrow C([0,+\infty),L^{q}(\mathbb{R}^n)),\notag\\	&\int\limits_0^\natural\mathrm{W}_{\sharp,\natural}h(t)=\int\limits_0^t\mathrm{W}_{s,t}h(s)ds,\quad\forall t\in[0,+\infty),\quad\forall h\in C^{\infty}(\mathbb{R}_{+}, W^{q,\infty}(\mathbb{R}^{n}))\cap C([0,+\infty),L^q(\mathbb{R}^n)).
\end{align}
\end{itemize} 
Linearity of the operators \eqref{i1}, \eqref{ii1}, \eqref{iii1}, and \eqref{iiii1} is immediate. To prove the continuity of the mapping \eqref{i1},
let $K\subset\mathbb{R}_+$ be a compact set and $q\in[1,\infty]$, $j\in\mathbb{N}_0$ as well as $\ell\in\mathbb{N}_0^n$. Then, since $h\in L^q(\mathbb{R}^n)$,
considering Remark \ref{notation2} and Young's inequality, we have
\begin{align}
	\|\mathrm{W}_{0,\natural}h\|_{q,j,\ell,K} &= \sup_{t\in K}\|\partial^{j}_{t}\partial_x^{\ell}(\mathrm{W}_{0,\natural}h(t))\|_{q}=\sup_{t\in K}\|\partial^{j}_{t}\partial^{\ell}_{x}\mathrm{W}_{0,t}h\|_{q}
	=\sup_{t\in K}\|\partial^{j}_{t}\partial^{\ell}_{x}(	\mathcal{W}(\cdot;0,t)*h)\|_{q}\notag\\&=	\sup_{t\in K}\|(\partial^{j}_{t}\partial^{\ell}_{x}\mathcal{W}(\cdot;0,t))*h)\|_{q}\leq \sup_{t\in K}\|\partial^{j}_{t}\partial^{\ell}_{x}\mathcal{W}(\cdot;0,t)\|_{1}\|h\|_{q}\lesssim_{j,\ell,K}\|h\|_q.
\end{align}
To prove \eqref{ii1}, let $b\in\mathbb{N}$ and $h\in L^q(\mathbb{R}^n)$. Then, by Remark \ref{notation2}, one has
\begin{align}\label{tnotzero}
	\|\mathrm{W}_{0,\natural}h\|_{q,b}&=\sup_{t\in[0,b]}\left\| \left( \mathrm{W}_{0,\natural}h\right) (t)\right\|_{q}=	\sup_{t\in[0,b]}\left\| \mathrm{W}_{0,t}h\right\| _{q}=\sup_{t\in[0,b]}\left\| \mathcal{W}(\cdot;0,t)*h\right\| _{q}\\
	&\leq 	\sup_{t\in[0,b]}\left\|\mathcal{W}(\cdot;0,t) \right\| _{1}\left\| h\right\| _{q}=\left\| h\right\| _{q}\notag.
\end{align}
To prove \eqref{iii1},  let $K\subset\mathbb{R}_+$ be a compact set and $q\in[1,\infty]$, $j\in\mathbb{N}_0$ as well as $\ell\in\mathbb{N}_0^n$. Then, since $h\in C^\infty(\mathbb{R}_+,W^{q,\infty}(\mathbb{R}^n))\cap C([0,+\infty),L^q(\mathbb{R}^n))$,
considering Remark \ref{notation2} and Young's inequality, we have
\begin{align}\label{wqinfty}
	&\left\|\int\limits_0^\natural\mathrm{W}_{\sharp,\natural}h\right\|_{q,j,\ell,k}=	\sup_{t\in K}\left\|\partial^{j}_{t}\partial_{x}^{\ell}	\int\limits_0^t\mathrm{W}_{s,t}h(s)ds\right\|_q\notag\\
	&=\sup_{t\in K}\left\|\partial^{\ell}_{x}\partial^{j}_{t}	\int\limits_0^t\mathcal{W}(\cdot;s,t)*h(s)ds\right\|_{q}\notag\\
	&\leq j\sup_{t\in K}\left\|\partial^{\ell}_{x}\partial^{j-1}_{t}h(t)\right\|_{q}+\sup_{t\in K}	\left\| \int\limits_0^t \partial^{\ell}_{x}\partial^{j}_{t}\left( \mathcal{W}(\cdot;s,t)*h(s)\right) ds\right\|_{q}\notag\\
	&= j\sup_{t\in K}\left\|\partial^{\ell}_{x}\partial^{j-1}_{t}h(t)\right\|_{q}+\sup_{t\in K}\left\|\int\limits_0^t \left(\partial^{\ell}_{x}\partial^{j}_{t} \mathcal{W}(\cdot;s,t)\right)*h(s) ds\right\|_{q}\notag\\
	&\leq j\sup_{t\in K}\left\|\partial^{\ell}_{x}\partial^{j-1}_{t}h(t)\right\|_{q}+\sup_{t\in K}	\int\limits_0^t \left\|\left(\partial^{\ell}_{x}\partial^{j}_{t} \mathcal{W}(\cdot;s,t)\right)*h(s) \right\|_{q}ds\notag\\
	&\leq j\sup_{t\in K}\left\|\partial^{\ell}_{x}\partial^{j-1}_{t}h(t)\right\|_{q}+\sup_{t\in K}	\int\limits_0^t \left\|\partial^{\ell}_{x}\partial^{j}_{t} \mathcal{W}(\cdot;s,t)\right\|_{1}\left\|h(s) \right\|_{q}ds\notag\\
	&\leq j\sup_{t \in K} \left\|\partial^{\ell}_{x}\partial^{j-1}_{t}h(t)\right\|_{q}+ \sup_{s \in [0,[\max{K}]+1]}\left\|h(s) \right\|_{q}\sup_{t\in K}\int\limits^{t}_{0}\left\|\partial^{\ell}_{x}\partial^{j}_{t} \mathcal{W}(\cdot;s,t)\right\|_{1}ds,\notag\\
	&=j\left\| h\right\| _{q,j-1,\ell,K}+C_{j,\ell,K}\left\| h\right\|_{q,[\max{K}]+1}
\end{align}
where in the last line $C_{j,\ell,K}\doteq\sup\limits_{t\in K}\int\limits^{t}_{0}\left\|\partial^{\ell}_{x}\partial^{j}_{t} \mathcal{W}(\cdot;s,t)\right\|_{1}ds<\infty$ holds because  $\mathcal{W}\in\mathcal{S}(\mathbb{R}^n)\otimes C^\infty(\Delta)$, so the function $(s,t)\mapsto\left\|\partial^{\ell}_{x}\partial^{j}_{t} \mathcal{W}(\cdot;s,t)\right\|_{1}$ is continuous for $\forall (s,t)\in \Delta$. Therefore, for each fixed $t\in\mathbb{R}_{+}$, the integrand is continuous in $s$, and integrable over $[0,t]$ and the resulting integral depends continuously on $t$. Since $K\subset \mathbb{R}_{+}$ is compact, the supremum over $t\in K$  is finite.

Now, to prove \eqref{iiii1}, let $b\in\mathbb{N}$ and  $h\in C^\infty(\mathbb{R}_+,W^{q,\infty}(\mathbb{R}^n))\cap C([0,\infty),L^q(\mathbb{R}^n))$. Then, by Remark \ref{notation2}, one can see 
\begin{align}\label{finallq}
	\left\|\int\limits_0^\natural\mathrm{W}_{\sharp,\natural}h\right\|_{q,b}&= \sup_{t\in[0,b]}\left\| \left( 	\int\limits_0^\natural\mathrm{W}_{\sharp,\natural}h\right) (t)\right\| _{q}
	=\sup_{t\in[0,b]}\left\| \int\limits_0^t\mathrm{W}_{s,t}h(s)ds\right\| _{q}\notag\\
	&\leq \sup_{t\in[0,b]}\int\limits_0^t\left\| \mathrm{W}_{s,t}h(s)\right\| _{q}ds=\sup_{t\in[0,b]}\int\limits_0^t\left\| \mathcal{W}{(\cdot;s,t)}*h(s)\right\| _{q}ds\notag\\
	&\leq \sup_{t\in[0,b]}\int\limits_0^t\left\| \mathcal{W}{(\cdot;s,t)}\right\| _{1}\left\| h(s)\right\|_{q}ds= \sup_{t\in[0,b]}\int\limits_0^t\left\| h(s)\right\|_{q}ds\notag\\
	&\leq b\sup_{s\in[0,b]}\left\| h(s)\right\| _{q}=b\left\| h\right\|_{q,b}.
\end{align}
\end{proof}

With these solutions maps, we state and prove the very weak analogue of Proposition \ref{pr5}.

\begin{proposition}\label{pr12} Let $u_0\in\mathcal{G}^\mho(L^q(\mathbb{R}^n))$ and $f\in\mathcal{G}^\mho(C^\infty(\mathbb{R}_+,\mathrm{W}^{q,\infty}(\mathbb{R}^n))\cap C([0,\infty),L^{q}(\mathbb{R}^{n})))$, for every $q\in[1,+\infty]$. Then
$$
u\doteq\mathrm{W}_{0,\natural}u_0+\int\limits_0^\natural\mathrm{W}_{\sharp,\natural}f\in\mathcal{G}^\mho(C^\infty(\mathbb{R}_+,W^{q,\infty}(\mathbb{R}^n))\,\cap\,C([0,+\infty),L^q(\mathbb{R}^n))),
$$
solves the Cauchy problem
$$
\begin{cases}
\partial_tu-\mathcal{L}_tu=f,\\
\left.u\right|_0=u_0.
\end{cases}
$$
If $u_0\ge0$ and $f\ge0$ then $u\ge0$.
\end{proposition}
\begin{proof}
	As we know the assumption $u_{0}=[u_{0,\epsilon}]_{\epsilon\in(0,1)}\in \mathcal{G}^\mho(L^q(\mathbb{R}^n))$ implies the following:
\begin{equation}\label{u0net}
	u_{0}=[u_{0,\epsilon}]_{\epsilon\in(0,1)}
	\doteq \{u_{0,\epsilon}\}_{\epsilon\in(0,1)}\mod\mathcal{N}^\mho(L^q(\mathbb{R}^n)).
\end{equation}
Similarly, the assumption \begin{equation}\label{fnet}
	f\in\mathcal{G}^\mho(C^\infty(\mathbb{R}_+,W^{q,\infty}(\mathbb{R}^n))\cap C([0,\infty), L^{q}(\mathbb{R}^{n}))),\end{equation} implies that
\begin{equation}\label{VWsourceterm}
	f=[f_{\epsilon}]_{\epsilon\in (0,1)}\doteq \{f_{\epsilon}\}_{\epsilon\in(0,1)}\mod\mathcal{N}^\mho(C^\infty(\mathbb{R}_+,L^q(\mathbb{R}^n))\cap C([0,\infty),L^{q}(\mathbb{R}^{n}))).
\end{equation}
Let	$u\doteq\mathrm{W}_{0,\natural}u_0+\int\limits_0^\natural\mathrm{W}_{\sharp,\natural}f$ which it leads us to $[u_{\epsilon}]_{\epsilon\in(0,1)}=\left[\mathrm{W}^{\epsilon}_{0,\natural}u_0+\int\limits_0^{\natural,\epsilon}\mathrm{W}_{\sharp,\natural}f\right]_{\epsilon\in(0,1)}$, while  $\mathrm{W}^{\epsilon}_{0,\natural}u_0\doteq\mathrm{W}_{0,\natural}u_{0,\epsilon}$ and $\int\limits_0^{\natural,\epsilon}\mathrm{W}_{\sharp,\natural}f\doteq\int\limits_0^{\natural}\mathrm{W}_{\sharp,\natural}f_{\epsilon}$, for every $\epsilon\in(0,1)$.
Now, since
$$
\partial_t:  C^\infty(\mathbb{R}_+,W^{q,\infty}(\mathbb{R}^n)) \longrightarrow  C^\infty(\mathbb{R}_+,W^{q,\infty}(\mathbb{R}^n)),
$$
and
$$
\mathcal{L}_t: C^\infty(\mathbb{R}_+,W^{q,\infty}(\mathbb{R}^n))\longrightarrow C^\infty(\mathbb{R}_+,W^{q,\infty}(\mathbb{R}^n)),$$
are linear and continuous operators, Corollary \ref{VWContCorr} ensures
the extension of the operators $\partial_{t}$ and $\mathcal{L}_t$  with formulae \eqref{VWDtDef} and \eqref{VWLtDef} respectively such that $\partial_t\in\mathcal{B}(\mathcal{G}^\mho(C^\infty(\mathbb{R}_+,W^{q,\infty}(\mathbb{R}^n))))$ and $\mathcal{L}_t\in\mathcal{B}(\mathcal{G}^\mho(C^\infty(\mathbb{R}_+,W^{q,\infty}(\mathbb{R}^n))))$.
Thus, it is immediately obtained that
\begin{equation}\label{VWdt}
	\mathcal{L}_tu=\{\mathcal{L}^{\epsilon}_tu_{\epsilon}\}_{\epsilon\in(0,1)}\mod\mathcal{N}^\mho(C^\infty(\mathbb{R}_+,W^{q,\infty}(\mathbb{R}^n))\cap C([0,\infty),L^{q}(\mathbb{R}^{n}))),
\end{equation} 
and 
\begin{equation}\label{VWLtx}
	\partial_{t}u=\{\partial_{t}u_{\epsilon}\}_{\epsilon\in(0,1)}\mod \mathcal{N}^\mho(C^\infty(\mathbb{R}_+,W^{q,\infty}(\mathbb{R}^n))\cap C([0,\infty),L^{q}(\mathbb{R}^{n}))),
\end{equation}
where $u=[u_{\epsilon}]_{\epsilon\in(0,1)}\in \mathcal{G}^\mho(C^\infty(\mathbb{R}_+,W^{q,\infty}(\mathbb{R}^n))\cap C([0,\infty),L^{q}(\mathbb{R}^{n})))$. Therefore, applying Proposition \ref{pr5} and \eqref{VWsourceterm}, we can see
\begin{align*}
	\partial_{t}u-\mathcal{L}_tu&=\{\partial_{t}u_{\epsilon}\}_{\epsilon\in(0,1)}-\{\mathcal{L}^{\epsilon}_tu_{\epsilon}\}_{\epsilon\in(0,1)}\mod\mathcal{N}^\mho(C^\infty(\mathbb{R}_+,W^{q,\infty}(\mathbb{R}^n))\cap C([0,\infty),L^{q}(\mathbb{R}^{n})))\\
	&=\{f_{\epsilon}\}_{\epsilon\in(0,1)}\mod\mathcal{N}^\mho(C^\infty(\mathbb{R}_+,W^{q,\infty}(\mathbb{R}^n))\cap C([0,\infty),L^{q}(\mathbb{R}^{n})))=f.
\end{align*}
To establish the non-negativity of the solution $u \in \mathcal{G}^\mho\big(C^\infty(\mathbb{R}_+, W^{q,\infty}(\mathbb{R}^n)) \cap C([0,+\infty), L^q(\mathbb{R}^n))\big)$, we proceed analogously to Proposition~\ref{pr3}. The solution is represented as
$$
u = \left[ \mathrm{W}_{0,\natural}^\epsilon u_{0,\epsilon} + \int_0^\natural \mathrm{W}_{\sharp,\natural}^\epsilon f_\epsilon \right]_{\epsilon \in (0,1)}.
$$
Under the assumptions $ u_{0,\epsilon} \geq 0$,  $f_\epsilon \geq 0$ for all $ \epsilon \in (0,1)$, and using the positivity of the heat kernel, we deduce that each $ u_\epsilon \geq 0$. Hence, the very weak solution $u$ is non-negative.
\end{proof}

The uniqueness of the very weak solution of the $L^q$-Cauchy problem takes the following form.

\begin{proposition}
For $q\in(1,+\infty)$, every solution
$$
u\in\mathcal{G}^\mho(C^\infty(\mathbb{R}_+,W^{q,\infty}(\mathbb{R}^n))\,\cap\,C([0,+\infty),L^q(\mathbb{R}^n))),
$$
of the Cauchy problem
$$
\begin{cases}
	\partial_tu-\mathcal{L}_tu=0,\\
	\left.u\right|_0=0.
\end{cases}
$$
is identically null, $u=0$.
\end{proposition}
\begin{proof}
With similar argument as in \eqref{VWdt} and \eqref{VWLtx} and by considering Proposition \ref{pr5} with the source term $f=0$ and initial value $u_{0}=0$ as well as \eqref{solution1}, one gets
\begin{align*}
	\partial_{t}u-\mathcal{L}_tu&=
	\{\partial_{t}u_{\epsilon}\}_{\epsilon\in(0,1)}-\{\mathcal{L}^{\epsilon}_tu_{\epsilon}\}_{\epsilon\in(0,1)}\mod\mathcal{N}^\mho(C^\infty(\mathbb{R}_+,W^{q,\infty}(\mathbb{R}^n))\cap C([0,\infty),L^{q}(\mathbb{R}^{n})))\\
	&=\{0\}_{\epsilon\in(0,1)}\mod\mathcal{N}^\mho(C^\infty(\mathbb{R}_+,W^{q,\infty}(\mathbb{R}^n))\cap C([0,\infty),L^{q}(\mathbb{R}^{n})))=[0]_{\epsilon\in(0,1)}.
\end{align*}
As discussed in Proposition \ref{pr12}, we know that the $L^{q}$ very weak solutions of the Cauchy problem in Proposition \ref{pr12} are as follows:
$$[u_{\epsilon}]_{\epsilon\in(0,1)}=\left[\mathrm{W}^{\epsilon}_{0,\natural}u_0+\int\limits_0^{\natural,\epsilon}\mathrm{W}_{\sharp,\natural}f\right]_{\epsilon\in(0,1)},$$ while  $\mathrm{W}^{\epsilon}_{0,\natural}u_0\doteq\mathrm{W}_{0,\natural}u_{0,\epsilon}$ and $\int\limits_0^{\natural,\epsilon}\mathrm{W}_{\sharp,\natural}f\doteq\int\limits_0^{\natural}\mathrm{W}_{\sharp,\natural}f_{\epsilon}$, for every $\epsilon\in(0,1)$. Assuming Proposition \ref{pr5} with the source term $f=0$ and initial value $u_{0}=0$, as well as \eqref{u0net} and \eqref{VWsourceterm}, it follows that
\begin{align}
	&[u_{\epsilon}]_{\epsilon\in(0,1)}=\left[\mathrm{W}^{\epsilon}_{0,\natural}u_0+\int\limits_0^{\natural,\epsilon}\mathrm{W}_{\sharp,\natural}f \right]_{\epsilon\in(0,1)}=\left[\mathrm{W}_{0,\natural}u^{\epsilon}_0+\int\limits_0^{\natural}\mathrm{W}_{\sharp,\natural}f_{\epsilon}\right]_{\epsilon\in(0,1)}\notag\\
	&=\{\mathrm{W}_{0,\natural}u_{0,\epsilon}\}_{\epsilon\in(0,1)}+\left\{\int\limits_0^{\natural}\mathrm{W}_{\sharp,\natural}f_{\epsilon}\right\}_{\epsilon\in(0,1)}\mod \mathcal{N}^\mho(C^\infty(\mathbb{R}_+,W^{q,\infty}(\mathbb{R}^n))\,\cap\,C([0,+\infty),L^q(\mathbb{R}^n))))\notag\\
	&=\{0\}_{\epsilon\in(0,1)} \mod \mathcal{N}^\mho(C^\infty(\mathbb{R}_+,W^{q,\infty}(\mathbb{R}^n))\,\cap\,C([0,+\infty),L^q(\mathbb{R}^n))))=[0]_{\epsilon\in(0,1)}.
\end{align}
Now, the proof is complete.
\end{proof}
\section*{Acknowledgements}
Z. Avetisyan and M. Ruzhansky acknowledge the support of  the Methusalem programme of the Ghent University Special Research Fund (BOF) (Grant number 01M01021). Z. Avetisyan's work is funded by the FWO Senior Research Grant G022821N.
Z. Keyshams and M. Mikaeili Nia acknowledge the support by the Higher Education and Science Committee, in the frames of the
project \textnumero~25FAST-1A006. Z. Keyshams acknowledges the support by the Higher Education and Science Committee of the Republic of Armenia, scientific project \textnumero~10-3/23/2PostDoc-2. M. Mikaeili Nia is grateful for the support from the Higher Education and Science Committee of the Republic of Armenia, scientific project \textnumero~10-3/23/2PostDoc-1. 

\end{document}